\documentclass[12pt]{article}

\usepackage[utf8]{inputenc}
\usepackage[T1]{fontenc}
\usepackage{lmodern}
\usepackage[english]{babel}
\usepackage{amsmath,mathtools,amsthm,amssymb,amsfonts}
\usepackage{color}
\usepackage{comment}

\usepackage{hyperref}
\usepackage{mathrsfs}
\usepackage{graphicx}
\usepackage{float}
\usepackage{enumitem}
\usepackage{tikz} 
\usepackage{esint}
\usepackage{pgfplots}
\usepackage{esvect}
\usepackage{pdfpages}
\usepackage{nicefrac}
\usepackage{leftidx, tensor}
\usepackage{cleveref}
\usepackage{accents}
\usetikzlibrary{shapes,arrows}
\usetikzlibrary{calc}
\usetikzlibrary{shapes.geometric}
\usetikzlibrary{shapes.arrows}
\usetikzlibrary{arrows.meta}
\usetikzlibrary{decorations.markings}
\usetikzlibrary{fit}
\usetikzlibrary{patterns}
\usetikzlibrary{hobby}
\usepgfplotslibrary{patchplots}
\pgfplotsset{compat=1.11}

\numberwithin{equation}{section}

\renewcommand{\thefootnote}{\fnsymbol{footnote}}
\newcommand{\definedas}{\mathrel{\raise.095ex\hbox{\rm :}\mkern-5.2mu=}}

\newcommand{\R}{\mathbb{R}}
\newcommand{\N}{\mathbb{N}}
\newcommand{\Z}{\mathbb{Z}}

\newcommand{\Sbb}{\mathbb{S}}

\renewcommand{\d}{\,\mathrm{d}}

\newcommand{\Hess}{\mathrm{Hess}}

\newcommand{\ul}[1]{\underline{#1}}

\newcommand{\btr}[1]{\left\vert#1\right\vert}
\newcommand{\newbtr}[1]{\vert#1\vert}
\newcommand{\norm}[1]{\btr{\btr{#1}}}

\newcommand{\spann}[1]{\left\langle#1\right\rangle}

\newcommand{\Ric}{\mathrm{Ric}}
\newcommand{\scal}{\mathrm{R}}
\newcommand{\Riem}{\operatorname{Rm}}

\newcommand{\two}{\operatorname{II}}
\newcommand{\tr}{\text{tr}}
\newcommand{\dive}{\operatorname{div}}

\newcommand{\newnorm}[1]{\vert\vert#1\vert\vert}
\newcommand{\hatgamma}{\widehat{\gamma}}
\newcommand{\tildegamma}{\widetilde{\gamma}}

\theoremstyle{plain}
\newtheorem{thm}{Theorem}[section]
\newtheorem*{thm*}{Main Theorem}
\newtheorem{prop}[thm]{Proposition}

\newtheorem{conj}[thm]{Conjecture}
\newtheorem{lem}[thm]{Lemma}

\theoremstyle{definition}
\newtheorem{defi}[thm]{Definition}

\newtheorem{bem}[thm]{Remark}
\newtheorem{kor}[thm]{Corollary}

\begin{document}
		\begin{center}\LARGE Foliations of asymptotically Schwarzschildean lightcones by surfaces of constant spacetime mean curvature \end{center}
		\vspace{0.5cm}
		\begin{center}
			{\large Klaus Kr\"oncke\footnote[1]{kroncke@kth.se} and Markus Wolff\footnote[2]{markuswo@kth.se}}\\
			\vspace{0.4cm}
			{\large Department of Mathematics}\\
			{\large KTH Royal Institute of Technology}
		\end{center}
		\vspace{0.4cm}
		\begin{abstract}
		 	We construct asymptotic foliations of asymtotically Schwarzschildean lightcones by surfaces of constant spacetime mean curvature (STCMC). 
		 	Our construction is motivated by the approach of Huisken-Yau for the Riemannian setting in employing a geometric flow. We prove that initial data within a sufficient a-priori class converges exponentially to an STCMC surface under area preserving null mean curvature flow. Further, we show that the resulting STCMC surfaces form an asymptotic foliation that is unique within the a-priori class.
		\end{abstract}
		\renewcommand{\thefootnote}{\arabic{footnote}}
		\setcounter{footnote}{0}
	\section{Introduction}
	
	For a spacelike, codimension-$2$ surface $(\Sigma,\gamma)$ in an ambient spacetime $(\overline{M},\overline{g})$, the \emph{spacetime mean curvature} $\mathcal{H}^2$ of $\Sigma$ as introduced by Cederbaum--Sakovich \cite{cederbaumsakovich} is defined as the Lorentzian length of the codimension-$2$ mean curvature vector $\vec{\mathcal{H}}$, i.e.,
	\[
	\mathcal{H}^2=\overline{g}(\vec{\mathcal{H}},\vec{\mathcal{H}}).
	\]
	Moreover, we say $(\Sigma,\gamma)$ is a \emph{surface of constant spacetime mean curvature} (STCMC) if $\mathcal{H}^2$ is constant along $\Sigma$. Although we will be mostly interested in surfaces where $\mathcal{H}^2$ is strictly positive, we note that $\vec{\mathcal{H}}$ has a-priori no fixed causal character and trapped surfaces, where $\mathcal{H}^2<0$ on $\Sigma$, naturally occur in the context of General Relativity.
	
	In the case that $\Sigma$ is contained in an initial data set $(M,g,K)$, which we can think of as a spacelike hypersurface $(M,g)$ in $(\overline{M},\overline{g})$ with second fundamental form $K$, one finds
	\[
	\mathcal{H}^2=H^2-(\tr_\Sigma K)^2,
	\]
	where $H$ denotes the mean curvature of $(\Sigma,\gamma)$ in $(M,g)$. In the context of initial data sets, STCMC surfaces have been studied as a generalization of CMC surfaces, see e.g. \cite{cederbaumsakovich, tenan, wolff3}. The famous Alexandrov Theorem \cite{aleks} states that any CMC surface in Euclidean space is a round sphere, while a version proven by Brendle \cite{brendleCMC} for more general warped product manifolds shows that any CMC surface in the spatial Schwarzschild manifold of positive mass $m>0$ is a sphere centered around the minimal surface $\{r=2m\}$. See also Borghini--Fogagnolo--Pinamonti \cite{borghinifogagnolopinamonti}. 
	
	In their seminal work \cite{huiskenyau}, Huisken--Yau first proposed a geometric notion of center of mass via asymptotic foliations of CMC surfaces. In their work, they construct such a foliation in asymptotically Schwarzschildean Riemannian manifolds using volume preserving mean curvature flow. In a recent work \cite{guileesun}, Gui--Lee--Sun give a new construction of the foliation by Huisken--Yau using area preserving mean curvature flow. Due to their properties,
	(asymptotic) foliations by CMC surfaces have been studied extensively in the context of General Relativity. See \cite{brendleeichmair, eichmairkoerber, huang, ma, ye2, ye3} and references therein for a non-exhaustive list of contributions. In the case of an asymptotically flat initial data set $(M,g,K)$, generalizations have been proposed by Metzger \cite{metzger} and by Cederbaum--Sakovich \cite{cederbaumsakovich} via foliations by surfaces of constant expansion and by STCMC surfaces, respectively. We also refer to Tenan \cite{tenan} for a recent alternative construction of the foliation by STCMC surfaces using a modified volume preserving flow. Indeed, explicit examples by Cederbaum--Nerz \cite{cednerz} suggest that the center of mass formulation by Cederbaum--Sakovich via STCMC surfaces seems to remedy some of the deficiencies and convergence issues of the Huisken--Yau center of mass in the case of a non-time symmetric initial data set $(M,g,K)$.
	
	Here, our aim is to construct asymptotic foliations by STCMC surfaces for a given null hypersurface with suitable asymptotics. We recall that a null hypersurface $\mathcal{N}$ carries a degenerate induced metric and is ruled by null geodesics. In particular, null hypersurfaces model the  trajectories of light that is emanating off of a given light source, such as of a far away celestial object, and considerations for conserved quantities such as energy, linear momentum, angular momentum, and center of mass are of similar interest as for initial data sets. Due to the degeneracy of the induced metric, we directly study the codimension-$2$ geometry of a spacelike cross section in the ambient spacetime. In particular, due to the gauge freedom in choosing a null frame $\{\ul{L},L\}$ normal to the spacelike cross section, the gauge independent spacetime mean curvature $\mathcal{H}^2$ presents itself as a direct analogue to the mean curvature of a surface in a Riemannian manifold in this setting. We note that with respect to a null frame $\{\ul{L},L\}$ with $\overline{g}(\ul{L},L)=2$, the spacetime mean curvature $\mathcal{H}^2$ is given by
	\[
	\mathcal{H}^2=\ul{\theta}\theta,
	\]
	where $\ul{\theta}$, $\theta$ denote the expansions with respect to $\ul{L}$, $L$, respectively. In particular, the intrinsic GCM spheres constructed by Klainerman--Szeftel \cite{klainermanszeftel} are also STCMC surfaces.
	
	In analogy with the Riemannian case, STCMC surfaces in the standard lightcone of the Minkowksi spacetime are intrinsically round and arise from spheres after a suitable Lorentz transformation, while STCMC surfaces in the Schwarzschild lightcone are round spheres centered around the marginally outer trapped surface $\{r=2m\}$, see Chen--Wang \cite{chenwang}. In particular, the Schwarzschild lightcone admits a unique (global) foliation by STCMC surfaces. In analogy with \cite{huiskenyau}, our main result is the construction of such asymptotic foliations outside of spherical symmetry under sufficiently strong asymptotics:

	\begin{thm*}
		Let $\mathcal{N}$ be an asymptotically Schwarzschildean lightcone with positive mass parameter $m>0$. Then there exists an asymptotic foliation of $\mathcal{N}$ by STCMC surfaces that is unique within a suitable a-priori class of surfaces.
	\end{thm*}
	For the precise statement and our definition of asymptotically Schwarzschildean lightcones, we refer to Theorem \ref{thm_main_foliation} and Definition \ref{defi_asymclassS} below, respectively. We note that {the leaves of the foliation are sufficiently round such }that one can evaluate energy, linear momentum and mass along the foliation in a meaningful way. In fact, one directly sees that the foliation has vanishing Bondi-momentum, and thus $E_{Bondi}=m_{Bondi}=m$ with respect to the constructed STCMC foliation. Hence, the constructed foliation is centered at least in the sense of vanishing linear momentum. We moreover conjecture that the foliation is possibly suitable to define a notion of center of mass. We refer to Section \ref{sec_outlook} for a more detailed discussion. 
	\setcounter{section}{1}
	\setcounter{thm}{0}
	
	Motivated by the work of Huisken--Yau \cite{huiskenyau}, we construct STCMC surfaces by evolving spacelike cross sections along the null hypersurface $\mathcal{N}$ under an appropriately modified mean curvature flow. Along a null hypersurface, mean curvature flow was first studied by Roesch--Scheuer in \cite{roeschscheuer}, and by the second named author in the case of the Minkowski and the de\,Sitter lightcone \cite{wolff1, wolff6}. Observe that there is no canonical way to define enclosed volume in this setting in a gauge invariant way, and that the rate of change of area for a spacelike cross section is given by the expansion $\ul{\theta}$ which is only pointwise dependent on the spacelike cross section. In particular, $\ul{\theta}$ extends to a well-defined function on $\mathcal{N}$. It thus seems natural to consider area preserving flows in this setting, and assuming that $\ul{\theta}>0$ on $\mathcal{N}$ allows us to define area preserving flows in a gauge independent way. Following the conventions of Roesch \cite{roesch2}, we refer to such null hypersurfaces as null cones and note that our {assumptions} ensure that $\ul{\theta}>0$ in the asymptotic region.
	
	Here, we say a family of spacelike cross sections $x\colon[0,T)\times\Sigma\to\mathcal{N}$ evolves along $\mathcal{N}$ along \emph{area preserving null mean curvature flow} if
	\begin{align}\label{eq_intro_flow}
	\frac{\d}{\d t}x=-\frac{1}{2\ul{\theta}}\left(\mathcal{H}^2-\fint\mathcal{H}^2\right)\ul{L},
	\end{align}
	where $\ul{L}$ denotes a choice of null generator of $\mathcal{N}$ and $\ul{\theta}$ denotes the expansion with respect to $\ul{L}$. It is direct to see that the parabolic flow preserves area and that STCMC surfaces are the stationary points of the flow. The similarities between our analysis and the work of Huisken--Yau suggest that area preserving null mean curvature flow can indeed be understood as the direct analogue to volume preserving mean curvature flow in the null setting. While not true in general, the second named author has shown in \cite{wolff1, wolff6} that area preserving null mean curvature flow directly arises as a rescaling of mean curvature flow in the lightcone of the Minkowski, de\,Sitter, and Anti de\,Sitter spacetime where it is furthermore equivalent to Hamilton's Ricci flow in $2$ dimensions.
	
	As in \cite{huiskenyau}, we consider spacelike cross sections within a suitable a-priori class of surfaces specified in Definition \ref{defi_aprioriclass} below. The a-priori class ensures $C^0$-estimates on the surface and that it is sufficiently round in the sense that the trace-free part $\accentset{\circ}{A}$ of a suitably defined notion of \emph{scalar second fundamental form} $A$ is sufficiently small in $C^1$. This notion of scalar second fundamental form was first considered by the second named author in \cite{wolff1, wolff4} in the Minkowksi lightcone. It holds that $\accentset{\circ}{A}=0$ if and only if $\mathcal{H}^2=\operatorname{const.}$ in the Minkowski lightcone, where in addition a suitable stability results holds with respect to integral norms \cite{wolff4}. We note that $A$ is gauge invariant and that $\tr_\Sigma A=\mathcal{H}^2$. However, $A$ can also be understood as a null second fundamental form for a specific choice of gauge equivalent to inverse mean curvature flow \cite{sauter}. This particular gauge choice was independently utilized by Le \cite{le1, le2, le3} in the Minkowski lightcone. 
	
	Using a slightly adapted version of a result by Shi--Wang--Wu \cite{shiwangwu}, the a-priori estimates (Proposition \ref{prop_apriori}) indeed show that any spacelike cross section in the a-priori class is $C^{2,\alpha}$-close to a round sphere under some mild additional assumptions, where the appropriate boost vector $\vec{a}$ is determined via an associated $4$-vector ${\textbf{Z}}$, see Subsection \ref{subsec_4vector}. The vector ${\textbf{Z}}$ was defined for a spacelike cross section of the Minkowski lightcone by the second named author in \cite{wolff4}, and independently for a surface in an asymptotically hyperbolic initial data set by Cederbaum--Cortier--Sakovich \cite{cederbaumcortiersakovich}. The a-priori estimates yield the stability of spacelike cross sections (Proposition \ref{prop_strictstability}), the invertibility of the Jacobi operator (Proposition \ref{prop_stability2}), and uniqueness of spacelike cross sections with prescribed spacetime mean curvature (Proposition \ref{prop_uniqueness}) within the a-priori class. Showing that the a-priori class is preserved under area preserving null mean curvature flow, the above properties lead to long time existence (Theorem \ref{thm_longtime}) and convergence of the flow (Theorem \ref{thm_convergence}). Finally, we show that the limiting STCMC surfaces starting from a smooth family of coordinate spheres form an asymptotic foliation (Theorem \ref{thm_foliation}).
	\\\\
	The paper is structured as follows: In Section \ref{sec_prelim}, we give a brief introduction to null geometry and other relevant concepts and collect the necessary preliminaries for our analysis. In Section \ref{sec_apriori}, we introduce the a-priori class, derive the crucial a-priori estimates, and prove all relevant properties of cross sections within the a-priori class. We introduce area preserving null mean curvature flow in Section \ref{sec_APNMCF} and show that the preservation of the a-priori class along the flow leads to long time existence and convergence. We show that the limiting surfaces of the flow form an asymptotic foliation and prove our main result in Section \ref{sec_foliation}. We close with some comments in Section \ref{sec_outlook}. For the convenience of the reader, we collect some more technical computations in the Appendix.
	
	\subsection*{Acknowledgments}
	The authors would like to thank Carla Cederbaum and Gerhard Huisken for their interest and enlightening discussions. The second named author appreciates the support by the Verg foundation.
	
	\section{Preliminaries}\label{sec_prelim}
	\subsection{Null Geometry and codimension $2$ surfaces}\label{subsec_generalsetup}
	In the following, we give a brief overview of the relevant concepts in null geometry and the geometry of codimension $2$ surfaces that we frequently employ throughout the work below. We refer the interested reader to \cite[Sections 4.7 and 4.8]{wolff_thesis} for a more detailed introduction using similar notation. See also \cite{marssoria, roesch, roeschscheuer, sauter}.
	
	We say a smooth, oriented hypersurface $\mathcal{N}$ in an ambient spacetime $(\overline{M},\overline{g})$ (of dimension $4$) is a \emph{null hypersurface} if $\mathcal{N}$ carries a degenerate induced metric. In particular, there exists a smooth null vector field $\ul{L}\in \Gamma(T\mathcal{N})$ such that $T_p\mathcal{N}=(\ul{L}_p)^\perp\subset T_p\overline{M}$ for any $p\in\mathcal{N}$. In this sense, $\ul{L}$ is both tangential and normal to $\mathcal{N}$ and we moreover observe that
	\begin{align}\label{eq_surfgrav}
		\overline{\nabla}_{\ul{L}}\ul{L}=\kappa\ul{L}
	\end{align}
	for a smooth function $\kappa$ on $\mathcal{N}$, where $\overline{\nabla}$ denotes the Levi-Civita connection of $(\overline{M},\overline{g})$. Hence, the integral curves of $\ul{L}$ are affine null geodesics. We call $\ul{L}$ a \emph{null generator} of $\mathcal{N}$, and note that $a\ul{L}$ is a different choice of null generator for any non-vanishing function $a\in C^\infty(\mathcal{N})$, which corresponds to a reparametrization of the integral curves. By \eqref{eq_surfgrav}, we can reparametrize the integal curves such that they are geodesics, and thus $\mathcal{N}$ is ruled by null geodesics.
	
	In the following, we consider spacelike codimension $2$ submanifolds $(\Sigma,\gamma)$ of $(\overline{M},\overline{g})$ as (spacelike) cross sections of $\mathcal{N}$, i.e., $\Sigma\subset \mathcal{N}$ and where we always assume that any integral curve of $\ul{L}$ intersects $\Sigma$ transversally and exactly once. As $T\Sigma\subset T\mathcal{N}$, we have $\ul{L}\in\Gamma(T^\perp\Sigma)$ and there exists a uniquely determined null vector field $L$ such that $\{\ul{L},L\}$ forms a null frame of $\Gamma(T^\perp\mathcal{N})$ with $\overline{g}(\ul{L},L)=2$. We recall that the \emph{null second fundamental forms} of $\Sigma$ with respect to $\{\ul{L},L\}$ are defined as
	\begin{align*}
		\underline{\chi}(X,Y):=\overline{g}(\overline{\nabla}_X\ul{L},Y),\qquad\chi(X,Y):=\overline{g}(\overline{\nabla}_XL,Y),
	\end{align*}
	where $X,Y\in\Gamma(T\Sigma)$, and that the \emph{null expansions} of $\Sigma$ with respect to $\{\ul{L},L\}$ are defined as
	\begin{align*}
		\ul{\theta}:=\tr_\gamma\ul{\chi},\qquad
		\theta:=\tr_\gamma\chi.
	\end{align*}
	We observe that the vector-valued second fundamental form $\vec{\two}$ and codimension $2$ mean curvature vector $\vec{\mathcal{H}}$ of $(\Sigma,\gamma)$ in $(\overline{M},\overline{g})$ can be written as
	\begin{align*}
		\vec{\two}=-\frac{1}{2}\chi\ul{L}-\frac{1}{2}\ul{\chi}L,\qquad
		\vec{\mathcal{H}}=-\frac{1}{2}\theta\ul{L}-\frac{1}{2}\ul{\theta}L
	\end{align*}
	with respect to the null frame $\{\ul{L},L\}$. In particular, the \emph{spacetime mean curvature} $\mathcal{H}^2$ of $\Sigma$, defined as the Lorentzian length of $\vec{\mathcal{H}}$ \cite{cederbaumsakovich}, is given by
	\[
		\mathcal{H}^2:=\overline{g}(\vec{\mathcal{H}},\vec{\mathcal{H}})=\ul{\theta}\theta.
	\]
	Further, the \emph{connection-$1$ form} of $\Sigma$ with respect to $\{\ul{L},L\}$ is defined as
	\[
		\zeta(X):=\frac{1}{2}\overline{g}(\overline{\nabla}_X\ul{L},L),
	\]
	where $X\in\Gamma(T\Sigma)$. In addition, if $\ul{\theta}$ is non-vanishing along $\Sigma$, where we can assume $\ul{\theta}>0$ without loss of generality (reparametrizating the integral curves of $\ul{L}$ if necessary), we define the \emph{scalar second fundamental form} $A$ of $\Sigma$ as
	\begin{align}\label{eq_prelim_defA}
		A:=\ul{\theta}\chi,
	\end{align}
	and \emph{torsion} $\tau$ of $\Sigma$ as
	\begin{align}\label{eq_prelim_deftau}
		\tau:=\zeta-\d\ln(\ul{\theta}),
	\end{align}
	see also \cite[Section 3]{wolff1} and \cite[Definition 1.3]{roesch}, respectively. We observe that the quantities $\mathcal{H}^2=\tr_\gamma A$, $A$, $\tau$ are independent of the choice of null frame $\{\ul{L},L\}$ (if well-defined), and thus independent of the choice of null generator $\ul{L}$ for a spacelike cross section of $\mathcal{N}$. We now state the well-known Gauss- and Codazzi equations with respect to a null frame $\{\ul{L},L\}$. Here, $\overline{\Riem}$, $\overline{\Ric}$, $\overline{\scal}$, and $\Riem$, $\Ric$, $\scal$ denote the Riemann curvature tensor, Ricci tensor, and scalar curvature of $(\overline{M},\overline{g})$ and $(\Sigma,\gamma)$, respectively. See Appendix \ref{appx_B} for our conventions for the Riemann curvature tensor.
	\begin{prop}[{\cite[Proposition 4.24]{wolff_thesis}}]\label{prop_nullgauss}
		Let $(x^i)$ denote local coordinates of $(\Sigma,\gamma)$. Then
		\begin{align}
		\overline{\Riem}_{ijkl}&=\Riem_{ijkl}-\frac{1}{2}\chi_{jl}\ul{\chi}_{ik}-\frac{1}{2}\ul{\chi}_{jl}\chi_{ik}+\frac{1}{2}\chi_{jk}\ul{\chi}_{il}+\frac{1}{2}\ul{\chi}_{jk}\chi_{il},\label{eq_prelim_gauss1}\\
		\overline{\Ric}_{ik}-\frac{1}{2}\overline{\Riem}_{i\ul{L}kL}-\frac{1}{2}\overline{\Riem}_{iLk\ul{L}}&=\Ric_{ik}-\frac{1}{2}\theta\ul{\chi}_{ik}-\frac{1}{2}\ul{\theta}\chi_{ik}+\frac{1}{2}(\chi\cdot\ul{\chi})_{ik}+\frac{1}{2}(\ul{\chi}\cdot\chi)_{ik},\label{eq_prelim_gauss2}\\
		\overline{\scal}-2\overline{\Ric}(L,\ul{L})+\frac{1}{2}\overline{\Riem}(\ul{L},L,L,\ul{L})&=\operatorname{R}-\mathcal{H}^2+\newbtr{\vec{\two}}^2.\label{eq_prelim_gauss3}
		\end{align}
	\end{prop}
	\begin{prop}[{\cite[Proposition 4.25]{wolff_thesis}}]\label{prop_nullcodazzi}
		Let $(x^i)$ denote local coordinates of $(\Sigma,\gamma)$. Then
		\begin{align*}
		\nabla_i\ul{\chi}_{jk}-\nabla_j\ul{\chi}_{ik}
		&=\overline{\Riem}_{ijk\ul{L}}-\zeta_j\ul{\chi}_{ik}+\zeta_i\ul{\chi}_{jk},\\
		\nabla_i\chi_{jk}-\nabla_j\chi_{ik}&=\overline{\Riem}_{ijkL}+\zeta_j\chi_{ik}-\zeta_i\chi_{jk}.
		\end{align*}
	\end{prop}

	\begin{bem}\label{bem_nullcodazzi}
		If $\ul{\theta}>0$, and $A$ and $\tau$ are well-defined, we find the following Codazzi-Equation for $A$ by using the product rule:
		\begin{align}\label{eq_codazziA}
			\nabla_iA_{jk}-\nabla_jA_{ik}&=\ul{\theta}\overline{\Riem}_{ijkL}+\tau_jA_{ik}-\tau_iA_{jk}.
		\end{align}
	\end{bem}
	Using the Codazzi Equation \eqref{eq_codazziA}, we will in the following establish a null version of Simon's identity for the scalar second fundamental form $A$ (if $\ul{\theta}>0$). Such a version has first been utilized in the Minkowski lightcone by the second named author, cf. \cite[Lemma 9]{wolff1}. Here, we will state it in the general setting. We also note that the Simon's identity admits an equivalent formulation using ambient tensor derivatives of the ambient curvature tensor, cf. Remark \ref{bem_nullsimon} in Appendix \ref{appendix_nullsimon}.
	\begin{prop}[Null Simon's Idenity]\label{prop_nullsimon}
		\begin{align*}
		\nabla_k\nabla_l A_{ij}
		=&\nabla_i\nabla_j A_{kl}+\Riem_{kijm}A^m_l+\Riem_{kilm}A^m_j
		\\&
		+\tau_j\nabla_iA_{kl}+\tau_i\nabla_kA_{jl}-\tau_k\nabla_iA_{jl}-\tau_l\nabla_kA_{ij}
		\\	&
		+\nabla_i\tau_jA_{kl}+\nabla_k\tau_iA_{lj}-\nabla_i\tau_kA_{jl}-\nabla_k\tau_lA_{ij}
		\\		&
		+\nabla_i(\overline{\Riem}_{kjl(\ul{\theta}L)})+\nabla_k(\overline{\Riem}_{lij(\ul{\theta}L)}).
		\end{align*}
	\end{prop}
	\begin{proof}
Using twice the Codazzi equation \eqref{eq_codazziA}, and the fact that
		\[
		\nabla_k\nabla_lT_{ij}-\nabla_l\nabla_kT_{ij}=\Riem_{kljm}T^m_i+\Riem_{klim}T^m_j
		\]
	holds	for any symmetric $(0,2)$ tensor $T$
we obtain that
		\begin{align*}
		\nabla_k\nabla_lA_{ij}
		=&\nabla_k\left(\nabla_iA_{lj}+\Riem_{lij(\ul{\theta}L)}+\tau_iA_{lj}-\tau_lA_{ij}\right)\\
		=&\nabla_k\nabla_iA_{lj}+\nabla_k(\overline{\Riem}_{lij(\ul{\theta}L)})+\tau_i\nabla_kA_{lj}-\tau_l\nabla_kA_{ij}+\nabla_k\tau_iA_{lj}-\nabla_k\tau_lA_{ij}\\
		=&\nabla_i\nabla_k A_{jl}+\Riem_{kilm}A^m_j+\Riem_{kijm}A^m_l+\nabla_k(\overline{\Riem}_{lij(\ul{\theta}L)})\\
		&+\tau_i\nabla_kA_{lj}-\tau_l\nabla_kA_{ij}+\nabla_k\tau_iA_{lj}-\nabla_k\tau_lA_{ij}\\
		=&\nabla_i(\nabla_j A_{kl}+\overline{\Riem}_{kjl(\ul{\theta}L)}+\tau_jA_{kl}-\tau_kA_{jl})\\
		&+\Riem_{kilm}A^m_j+\Riem_{kijm}A^m_l+\nabla_k(\overline{\Riem}_{lij(\ul{\theta}L)})\\
		&+\tau_i\nabla_kA_{lj}-\tau_l\nabla_kA_{ij}+\nabla_k\tau_iA_{lj}-\nabla_k\tau_lA_{ij}\\
		=&\nabla_i\nabla_j A_{kl}+\Riem_{kijm}A^m_l+\Riem_{kilm}A^m_j\\
		&+\tau_j\nabla_iA_{kl}+\tau_i\nabla_kA_{jl}-\tau_k\nabla_iA_{jl}-\tau_l\nabla_kA_{ij}\\
		&+\nabla_i\tau_jA_{kl}+\nabla_k\tau_iA_{lj}-\nabla_i\tau_kA_{jl}-\nabla_k\tau_lA_{ij}\\
		&+\nabla_i(\overline{\Riem}_{kjl(\ul{\theta}L)})+\nabla_k(\overline{\Riem}_{lij(\ul{\theta}L)}).
		\qedhere
		\end{align*}
	\end{proof}
	Recall that for a fixed choice of null generator $\ul{L}$ and spacelike cross section $S$, one can construct a \emph{background foliation} of $\mathcal{N}$ given by a diffeomorphism
	\[
		\Phi_{{\ul{L},S}}\colon (r_1,r_2)\times S\to \mathcal{N}, (s,q)\mapsto \gamma^{\ul{L}}_q(s)
	\]
	for $-\infty\le r_1<r_2\le \infty$, where $\gamma^{\ul{L}}_q$ denotes the integral curve of $\ul{L}$ starting at $q\in S\subset \mathcal{N}$. While the foliation a-priori only exists locally, we will always assume that $\Phi_{{\ul{L},S}}$ is surjective for convenience. Let $\pi^1$ and $\pi^S$ denote the projection onto the $s$ and $S$ component, respectively. We note that for any spacelike cross section $\Sigma$, $\phi_\Sigma:=\pi^S\circ\Phi_{{\ul{L},S}}^{-1}\vert_{\Sigma}$ is a smooth diffeomorphism from $\Sigma$ to $S$, and we can define
	\[
		F\colon \Sigma\to(r_1,r_2),\qquad p\mapsto s= \pi^1\circ\Phi_{{\ul{L},S}}^{-1}(p).
	\]
	We extend $F$ constantly along the integral curves of $\ul{L}$ for technical reasons. In particular, all spacelike cross sections are homeomorphic, and in the following we will always assume that they are topological spheres. For a fixed diffeomorphism $\phi\colon \Sbb^2\to S$, $\Phi_{{\ul{L},S}}$ now induces the diffeomorphism
	\[
		\Phi_{\ul{L}}:=\Phi_{{\ul{L},S}}\circ(\mathrm{id},\phi)\colon (r_1,r_2)\times\Sbb^2\to\mathcal{N}.
	\]
	By defining the function 
	\[
		\omega\colon\Sbb^2\to(r_1,r_2),\qquad\vec{x}\mapsto F\circ\phi_{\Sigma}^{-1}\circ\phi(\vec{x})
	\]
	for a given spacelike cross section $\Sigma$, we can uniquely identify $\Sigma$ as the graph of $\omega$ over $\Sbb^2$ in the sense that $\Sigma=\{s=\omega\}=\{p\in\mathcal{N}\vert p=\Phi_{\ul{L}}(\omega(\vec{x}),\vec{x})\}$. We will write $\Sigma=\Sigma_\omega$ and denote the induced metric on $\Sigma$ as $\gamma_{\omega}$. Moreover, we will identify $F=\omega\circ\phi^{-1}\circ\pi^S$ and $\omega$ going forward by slight abuse of notation.
	
	We denote the leaves of the background foliation by $S_r:=\{s=r\}$, and denote all objects defined on $S_r$ via a subscript $r$. In particular, we observe that the null vector field $L_r\in\Gamma(T^\perp\Sigma_r)$, null second fundamental forms $\ul{\chi}_r$, $\chi_r$, null expansions $\ul{\theta}_r$, $\theta_r$, and connection $1$-form $\zeta_r$ extend to smooth objects on $\mathcal{N}$. Additionally, for any spacelike cross section $\Sigma$ and all $p\in\Sigma$ with $p=\omega(\vec{x})$, $\vec{x}\in\Sbb^2$, we note that $p\in S_r$ for $r=F(p)=\omega(\vec{x})$ and there exists a diffeomorphism
	\[
		T_pS_{r=\omega(\vec{x})}\to T_p\Sigma,\qquad V\mapsto \widetilde{V}:=V+V(\omega)\ul{L},
	\]
	and we observe that
	\[
		L_p=(L_{r=\omega(\vec{x})})_p+\btr{\nabla\omega}_{\gamma_\omega}\ul{L}-2\nabla\omega(p),
	\]
	where $L$ denotes the unique null vector field in $\Gamma(T^\perp\Sigma)$ such that $\overline{g}(\ul{L},L)=2$, and $\nabla\omega$ the gradient of $\omega$ on $\Sigma$, respectively. Using the above relations, we find  that
	\begin{align*}
		(\gamma_\omega)_p(\widetilde{V},\widetilde{W})=(\gamma_{r=\omega(\vec{x})})_{p}(V,W),\qquad
		(\ul{\chi})_p(\widetilde{V},\widetilde{W})=(\ul{\chi}_{r=\omega(\vec{x})})_p(V,W),
	\end{align*}
	and $\ul{\theta}(p)=\ul{\theta}_{r=\omega(\vec{x})}(p)$, where $\ul{\chi}$ and $\ul{\theta}$ denote the null second fundamental form and null expansion of $\Sigma$ with respect to $\ul{L}$, respectively. We observe that $\gamma_r$, $\ul{\chi}_r$ extend to well-defined transversal tensor fields on $\Gamma(T\mathcal{N})$, cf. \cite[Section 3]{marssoria}, and $\ul{\theta}_r$ to a well-defined function on $\mathcal{N}$ (independent of the choice of background foliation). Hence, we will often omit the subscript $r$ for these objects without ambiguity. In particular, the properties of $\ul{\theta}$ are fully determined by $\mathcal{N}$. We say $\mathcal{N}$ is a \emph{null cone} if $\ul{\theta}>0$ on $\mathcal{N}$, cf. \cite[Definition 2.1]{roesch2}, and note that $A$, $\tau$ are well-defined for any spacelike cross section of a null one.
	
	We refer to \cite[Proposition 4.22]{wolff_thesis} on how $\chi$, $\zeta$, $\theta$ can be expressed via $\omega$ and the background foliation. Motivated by these formulas, we choose to extend $\chi_r$, $\zeta_r$ transversally to $\Gamma(T\mathcal{N})$\footnote{Unlike $\ul{\chi}$, this particular choice of extension does depend on the background foliation and is not necessarily consistent with the formal definition, but is most convenient for our purposes in view of \cite[Proposition 4.22]{wolff_thesis}. In contrast, the formal definition of $\zeta_r$ yields $\frac{1}{2}\overline{g}(\overline{\nabla}_{\ul{L}}\ul{L},L_r)=\kappa$, which does not agree with our choice of extension unless $\ul{L}$ is geodesic, cf. \cite{roeschscheuer}.}, i.e.,
	\begin{align*}
		\chi_r(\cdot, \ul{L})=\chi_r(\ul{L},\cdot)&=0,\qquad 
		\zeta_r(\ul{L})=0,
	\end{align*}
	and thus find that $\chi_r(\widetilde{V},\widetilde{W})=\chi_r(V,W)$, $\zeta_r(\widetilde{V})=\zeta_r(V)$ when restricted to $\Gamma(T\Sigma)$, cf. \cite[Proposition 4.22]{wolff_thesis}.
	
	Finally, by considering coordinates $\partial_I$ on $\Sbb^2$ (which we will always denote by capital letters), we can use $\Phi_{\ul{L}}$ to push these coordinates forward onto $\Gamma(T\mathcal{N})$ along the integral curves of $\ul{L}$ using parallel transport such that $[\partial_I,\ul{L}]=0$, where we identify the coordinates $\partial_I$ and their pushforward tangent to $S_r$ by slight abuse of notation. In particular, this induces coordinates $\partial_i$ on any spacelike cross section $\Sigma=\Sigma_\omega$, where
	\begin{align}\label{eq_prelim_nullgeometry_partiali}
		\partial_i=\partial_I+\partial_I\omega\ul{L}.
	\end{align}
	Recall that we have extended $\omega$ constantly along the integral curves of $\ul{L}$, and by the properties of $\gamma_\omega$ we find
	\begin{align*}
		\nabla\omega&=(\gamma_\omega)^{ij}(\partial_j\omega)\partial_i=(\gamma_\omega)^{IJ}(\partial_J\omega)(\partial_I+\partial_I\omega\ul{L})=\nabla^I\omega\partial_I+\btr{\nabla\omega}^2_{\gamma_\omega}\ul{L}
	\end{align*}
	by slight abuse of notation, where $\nabla^I\omega:=(\gamma_\omega)^{IJ}\partial_J\omega$. In particular, we find
	\begin{align}\label{eq_prelim_nullgeometry_L}
		L=L_r-\btr{\nabla\omega}^2_{\gamma_{\omega}}\ul{L}-2\nabla\omega^I\partial_I.
	\end{align}
	Equivalently, we can now pull back all objects onto $\Sbb^2$. Ultimately, this allows us to perform our analysis of spacelike cross sections $\Sigma$ of $\mathcal{N}$ purely in terms of objects on $\Sbb^2$. In particular, we can identify tensor fields on $\Sigma$, and tensor fields of the background foliation, as tensor fields $T$ on $\Sbb^2$ depending on the function $\omega$, and smooth families of tensor fields $(T_r)$ on $\Sbb^2$ (evaluated along the set $\{r=\omega\}$), where we perform all computations with respect to the same coordinates $\partial_I$, and use the decompositions \eqref{eq_prelim_nullgeometry_partiali}, \eqref{eq_prelim_nullgeometry_L} when necessary.
	\subsection{Weighted tensor norms}\label{subsec_norms}
Recall that for a metric $\gamma$ on $\Sbb^2$, we define the \emph{area radius} $\rho$ via the relation
	\[
		4\pi\rho^2=\int_{\Sbb^2}\d\mu_\gamma.
	\]
Now we define weighted norms as follows. Set
	\[
	\norm{f}_{L^p(\gamma)}:=\left(\int_{\Sbb^2}\btr{f}^p\d\mu_\gamma\right)^{\frac{1}{p}},
	\]
	and 
	\[
	\norm{f}_{L^\infty(\gamma)}:=\operatorname{ess sup}\limits_{\Sbb^2}\btr{f}.
	\]
	and define weighted Sobolev norms iteratively by
	\[
	\norm{f}_{W^{0,p}(\gamma)}:=\norm{f}_{L^p(\Sigma)},\text{  }\norm{f}_{W^{k+1,p}(\gamma)}:=\norm{f}_{L^p(\gamma)}+\rho\norm{\nabla f}_{W^{k,p}(\gamma)}.
	\]
We define weighted $C^k$-norms by setting $p=\infty$ in the above, i.e.,
	\[
		\norm{f}_{C^k(\gamma)}:=\norm{f}_{W^{k,\infty}(\gamma)}.
	\]
	As $\rho=1$ for the standard round metric $\hatgamma$, the weighted norms defined here agree with the usual (unweighted) Sobolev norms in this case, so we may use both without ambiguity. 	These norms are similar to the weighted norms used e.g. in \cite{cederbaumsakovich} and extend to weighted norms on tensor bundles in a canonical way.
	
	{As a null hypersurface $\mathcal{N}$ carries a degenerate induced metric, we formulate decay conditions along $\mathcal{N}$ with respect to a choice of background foliation, and identify (transversal\footnote{trivial in $\ul{L}$-direction, cf. \cite[Section 4]{marssoria}}) tensors fields along the background foliation with a smooth family of tensor fields on $\Sbb^2$ via the induced pullback as described in Subsection \ref{subsec_generalsetup}. In the following, we will usually omit the pullback for simplicity and measure decay with respect to the round sphere $(\Sbb^2,\hatgamma)$.}
	\begin{defi}\label{defi_prelim_O}
		Let $(T_r)$ be a (smooth) family of $(m,n)$-tensors on $\Sbb^2$, $m,n\in\mathbb{N}_0$, and let $k,l\in\mathbb{N}_0$, $\alpha\in\Z$. We say $(T_r)$ is in $O_{k,l}(r^{\alpha})$ (or simply $T_r=O_{k,l}(r^{\alpha})$) if
		\[
			\btr{\widehat{\nabla}^i\partial_r^jT_r}_{\hatgamma}\le Cr^{\alpha-j}
		\]
		for all $0\le i\le k$, $0\le j\le l$. We often omit the subscript $r$ for simplicity.
	\end{defi}
	\begin{bem}\label{bem_prelim_O}
		Note that $\widehat{\nabla}$ and $\partial_r$ commute, i.e., the above estimates also holds for any combination of $\widehat{\nabla}$ and $\partial_r$ derivatives. In the following, we will always associate $\partial_r$ with the null generator $\ul{L}$ of a given background foliation. Recall that $[\ul{L},\partial_I]=0$, so we can also understand the $r$-partial derivatives as Lie-derivatives, i.e.,
		\[
			\partial_r^jT=\mathcal{L}_{\ul{L}}^jT,
		\]
		cf. \cite[Section 3]{marssoria}. However, compared to the formulation of decay in \cite{marssoria}, we chose to define the decay in tangent directions of $\Sbb^2$ with respect to the Levi-Civita connection of a given metric on $\Sbb^2$ (namely $\hatgamma$) instead of Lie-derivatives.
	\end{bem}
	{In addition, note that due to the degeneracy of the metric it is not straightforward to define a Levi-Civita connection of $\mathcal{N}$. In particular, we will regard spacelike cross sections as codimension $2$ surfaces and employ the Gauss- and Codazzi equations, Propositions \ref{prop_nullgauss} and \ref{prop_nullcodazzi}, in the ambient spacetime. Thus, curvature terms of the ambient spacetime will frequently appear in our analysis below. Using the frame $\{\partial_I,\ul{L},L_r\}$ of $TM$ and its dual frame of $T^*M$ along $\mathcal{N}$ induced by the background foliation, we formulate decay conditions for tensor fields of the ambient spacetime along $\mathcal{N}$ with respect to a choice of background foliation in the following way:}
	\begin{defi}\label{defi_prelim_O_spacetime}
		Let $(S_r)\subset \overline{M}$ be a smooth family of topological codimension $2$ spheres in an ambient spacetime $(\overline{M},\overline{g})$, and let $\{\ul{L}_r,L_r\}$ be a (smooth) null frame of $S_r$, $\alpha\in\Z$. We say an $(m,n)$-tensor $T$ of $\overline{M}$ is in $\overline{O}_{k,l}(r^\alpha)$ (or simply $T=\overline{O}_{k,l}(r^{\alpha})$) (with respect to the null frame $\{\ul{L}_r,L_r\}$ and $\hatgamma$) if for all subsets $\mathcal{A}\subseteq \{1,\dotsc,n\}$, $\mathcal{B}\subseteq \{1,\dotsc,m\}$
		the (smooth) family of $(m-\btr{\mathcal{B}},n-\btr{\mathcal{A}})$-tensors
		\[
			(T_r)_{\alpha_1\dotsc\alpha_n}^{\beta_1\dotsc\beta_m}
		\]
		on $\Sbb^2$ defined via $\alpha_i\in\{\ul{L}_r,L_r\}$ $\forall i\in\mathcal{A}$, ${\beta_j\in\{\ul{L}_r^*,L_r^*\}}$ $\forall j\in\mathcal{B}$, and the pullback of $T$ with respect to $S_r$ is in $O_{k,l}(r^{\alpha+n-\btr{\mathcal{A}}-m+\btr{\mathcal{B}}})$ in the sense of Definition \ref{defi_prelim_O}.
	\end{defi}
	\begin{bem}\label{bem_prelim_O_spacetime}
		{
		Here, we chose the formulation of decay for tensor fields of the ambient spacetime such that the scaling will be consistent with respect to any choice of subsets $\mathcal{A}$, $\mathcal{B}$ and the induced metric $\gamma_r$ on $S_r$: If $T=\overline{O}(r^\alpha)$, then
		\[
			\btr{T_{\alpha_1\dotsc\alpha_n}^{\beta_1\dotsc\beta_m}}_{\gamma_r}\le Cr^\alpha
		\] 
		for all possible combinations of indices $\alpha_i,\beta_j\in\{\partial_I,\ul{L}_r,L_r\}$ if $\gamma_r$ is uniformly equivalent to the conformally round metric $r^2\hatgamma$, cf. Definition \ref{defi_asymclassS} below. Note that if one formulates the decay assumptions in Definition \ref{defi_prelim_O_spacetime} directly with respect to $\gamma_r$ (or any $r$-dependent family of metrics on $\Sbb^2$), then the Levi-Civita connection $\nabla^{\gamma_r}$ and $\partial_r$ will not commute in general. In particular, the decay $O$ in Definition \ref{defi_prelim_O} and $\overline{O}$ in Definition \ref{defi_prelim_O_spacetime} is formulated such that is convenient to use in their respective (mutually exclusive) cases.}
	\end{bem}

	\subsection{An associated $4$-vector in the Minkowski spacetime}\label{subsec_4vector}
		To obtain the appropriate a-priori estimates, we quickly summarize the definition and properties of an associated $4$-vector in the Minkowski spacetime. 
Recall that every spacelike cross section of the standard Minkowski lightcone is of the form $\Sigma=\{t=r=\omega\}=:\Sigma_\omega$ for a function $\omega:\Sbb^2\to\R$ and the  induced metric is $\omega^2\hatgamma$. We define the associated $4$-vector $\textbf{Z}=\textbf{Z}(\Sigma_{\omega})$ by setting its components as
		\begin{align*}
		\textbf{Z}^t:=\frac{1}{\btr{\Sigma}}\int_{\Sbb^2}\omega^3\d\widehat{\mu},\qquad
		\textbf{Z}^i:=\frac{1}{\btr{\Sigma}}\int_{\Sbb^2}\omega^3f^i\d\widehat{\mu},\quad i=1,2,3
		\end{align*}
		where $\d\widehat{\mu}$ is the volume element of the standard metric $\widehat{\gamma}$ on $\Sbb^2$ and $f^i$ denote the first spherical harmonics. This $4$-vector was {introduced for spacelike cross sections of the Minkowski lightcone} by the second named author in \cite{wolff4}, and we refer to \cite{wolff4} for more details. {We note that this $4$-vector is up to scaling equivalent to a notion of hyperbolic center independently defined by Cederbaum--Cortier--Sakovich, cf.\ \cite{cederbaumcortiersakovich}.}
		
		 As shown in \cite{wolff4}, $\textbf{Z}$ is a timelike, {future-pointing} vector, so there exists a unique $\vec{a}\in\R^3$ such that
		\[
		\textbf{Z}=\btr{\textbf{Z}}\begin{pmatrix}
		\sqrt{1+\btr{\vec{a}}^2}\\\vec{a}
		\end{pmatrix}.
		\]
		Let $\Lambda_{\vec{a}}$ be the Lorentz boost associated to the vector $\vec{a}\in\R^3$.
The boosted surface $\Lambda_{\vec{a}}(\Sigma_{\omega})$ is now given by $\Sigma_{\omega_{\vec{a}}}$,
 where
%
		\[
		\omega_{\vec{a}}(\vec{x})=\frac{\omega\circ\Phi(\vec{x})}{\sqrt{1+\btr{\vec{a}}^2}-\vec{a}\cdot\vec{x}},
		\]
		and $\Phi\colon\Sbb^2\to\Sbb^2$ is a diffeomorphism\footnote{In fact, $\Phi$ is a M\"obius transformation in the M\"obius group, which is isomorphic to $\operatorname{SO}^+(1,3)$.} uniquely determined by $\Lambda_{\vec{a}}$. It is shown that $\textbf{Z}$ transforms equivariantly under Lorentz boosts, i.e.,
		\[
		\textbf{Z}(\Sigma_{\omega_{\vec{a}}})=\Lambda_{\vec{a}}(\textbf{Z}(\Sigma_\omega)).
		\]
		For $\rho>0$, $\vec{a}\in\R^3$,
	we set	
		\[
		b_{\rho,\vec{a}}(\vec{x}):=\rho_{\vec{a}}(\vec{x})=\frac{\rho}{\sqrt{1+\btr{\vec{a}}^2}-\vec{a}\cdot\vec{x}}
		\]
		and abbreviate $b_{\vec{a}}:=b_{1,\vec{a}}$. 

		Although $\textbf{Z}$ was used in \cite{wolff4} to obtain integral estimates controlling the conformal factor, we desire pointwise estimates in the following. The following proposition gives the corresponding estimate, and is obtained by a straightforward modification of the arguments in the proof of \cite[Theorem 3]{shiwangwu} {using the balancing condition Equation (34) in \cite[Proposition 19]{wolff4}}:
		\begin{prop}\label{prop_shiwangwu}
			Let $\gamma=\omega^2\hatgamma$ be a conformally round metric with area radius $\rho$, and associated $4$-vector $\textbf{Z}$ corresponding to a vector $\vec{a}\in\R^3$. Assume further that $\frac{1}{10}\rho\le\omega\le 10\rho$. Define $\widetilde{\omega}:=\rho^{-1}\omega$, and denote the Gauss curvature of $\widetilde{\omega}^2\widehat{\gamma}$ by $\widetilde{K}$. Then there exist (uniform) constants $\varepsilon>0$ $C>0$, such that if
			\[
			\norm{\widetilde{K}-1}_{C^1\left(\hatgamma\right)}\le \varepsilon
			\]
			then
			\begin{align*}
			\norm{\widetilde{\omega}_{-\vec{a}}-1}_{C^{2,\alpha}\left(\hatgamma\right)}\le C\norm{\widetilde{K}-1}_{C^\alpha\left(\hatgamma\right)},
			\end{align*}
			where $C$ is only depending on $\alpha$.
			In particular,
			\begin{align*}
			\norm{\omega-b_{\rho,\vec{a}}}_{C^{2,\alpha}\left(\hatgamma\right)}\le C(\vec{a})\rho C\norm{\widetilde{K}-1}_{C^1\left(\hatgamma\right)}.
			\end{align*}
		\end{prop}
		Note that by definition, $\widetilde{\omega}_{-\vec{a}}$ satisfies the balancing condition \cite[Proposition 19, Equation (34)]{wolff4} . The last estimate on the original conformal factor $\omega$ follows as in the proof of \cite[Theorem 23]{wolff4}\footnote{Note the difference in notation.}
	
	\subsection{Asymptotically Schwarzschildean lightcones}\label{subsec_asymcond}
		We now specify our asymptotic conditions. To a large extend, these constitute a version of the asymptotic flatness conditions in \cite{marssoria} suitably modified and strengthened for our setting.
		\begin{defi}[Asymptotically Schwarzschildean lightcones]\label{defi_asymclassS}
			Let $\mathcal{N}$ be a null hypersurface in an ambient spacetime. We say $\mathcal{N}$ is an \emph{asymptotically Schwarzschildean lightcone} (of mass $m$), if there exists a geodesic null generator $\ul{L}$ and an (asymptotic) background foliation $(S_r)_{r\in(r_0,\infty)}$, $r_0>0$, of topological $2$-spheres, and a smooth, real-valued function $h:(r_0,\infty)\to(0,\infty)$ with $h(r)=1-\frac{2m}{r}+O_4(r^{-2})$, such that
			\begin{align*}
				\gamma_r&=r^2\hatgamma+O_{3,3}(1),\\
				\ul{\chi}_r&=r\hatgamma+O_{3,3}(r^{-1}),\\
				\chi_r&=rh(r)\hatgamma+O_{1,1}(r^{-1}),\\
				\zeta_r&=O_{2,2}(r^{-2}),
			\end{align*}
			and the ambient curvature tensor satisfies
			\[
				\overline{\Riem}-\overline{\Riem}^{Schw}=\overline{O}_{2,2}(r^{-4}),
			\]
			and
			\[
				\overline{\nabla}\overline{\Riem}-\overline{\nabla}\overline{\Riem}^{Schw}=\overline{O}_{1,1}(r^{-5}),
			\]
			where $\overline{\Riem}^{Schw}$, $\overline{\nabla}\overline{\Riem}^{Schw}$ denote the curvature in the Schwarzschild spacetime of mass $m$, 
			{and where we evaluate $\overline{\Riem}^{Schw}$, $\overline{\nabla}\overline{\Riem}^{Schw}$ in $\ul{L}$, $L_r$ with respect to the corresponding null vectors $\ul{L}$, $L_r$ in the Schwarzschild spacetime by slight abuse of notation.} 
		\end{defi}
		\begin{bem}\label{bem_asymclassS}\,
				\begin{enumerate}
					\item[(i)]{
				While we only assume the above decay assumptions to hold along $\mathcal{N}$, Definition \ref{defi_asymclassS} heuristcally covers the case when the ambient semi-Riemannian metric $\overline{g}$ is sufficiently close to the Schwarzschild metric $\overline{g}^{Schw}$ in a neighbourhood of $\mathcal{N}$. Indeed, we expect that this is the case if
				\[
					\btr{\partial^a\left(\overline{g}_{\alpha\beta}-\overline{g}^{Schw}_{\alpha\beta}\right)}\le \frac{C}{r^{2+|a|}}
				\]
				in suitable spacetime coordinates for derivatives up to forth order, i.e., $\btr{a}=0,\dotsc,4$, in analogue to the assumptions in \cite{huiskenyau}.}
			
				\item[(ii)]{Up to the differences in formulating the decay, cf. Subsection \ref{subsec_norms}, an asymptotically Schwarzschildean lightcone $\mathcal{N}$ is indeed asymptotically flat in the sense of \cite{marssoria}, where in both cases the decay is formulated along a given background foliation. However, while applying a Lorentz transformation in $\operatorname{SO}^+(1,3)$ to an asymptotically flat background foliation gives rise to a different asymptotically flat background foliation, this is not true for the background foliation for which the null hypersurface is asymptotically Schwarzschildean. See Section \ref{sec_outlook} for a more detailed discussion.}
			
				\item[(iii)]{We note that our formal assumptions could possibly be weakened or possibly used to derive improved decay by utilizing the full force of the Null Structure Equations. For example, the decay assumptions of $\overline{\nabla}{\Riem}$ in directions tangent to $S_r$ readily follow from the other assumptions, as the Gauss Equation yields
				\begin{align*}
				\overline{\nabla}_I\overline{\Riem}_{KJAL_r}=&\,{}^{\gamma_r}\nabla_I\left(\overline{\Riem}_{KJAL_r}\right)+\overline{\Riem}_{KJAM}(\chi_r)_j^m-(\zeta_r)_I\overline{\Riem}_{KJAL_r}\\
				&\,-\frac{1}{2}\overline{\Riem}_{\ul{L}JAL_r}(\chi_r)_{IK}-\frac{1}{2}\overline{\Riem}_{K\ul{L}AL_r}(\chi_r)_{IJ}-\frac{1}{2}\overline{\Riem}_{KJ\ul{L}L_r}(\chi_r)_{IA}\\
				&\,-\frac{1}{2}\overline{\Riem}_{L_rJAL_r}\ul{\chi}_{IK}--\frac{1}{2}\overline{\Riem}_{KL_rAL_r}\ul{\chi}_{IJ}.
				\end{align*}
				To avoid further lengthy computations, we have stated all the decay assumptions as needed for simplicity.}
			\end{enumerate}
		\end{bem}
		\noindent Note that it is direct to see that
		\begin{align}\label{eq_asym_inversemetric}
			\gamma_r^{kl}=\frac{1}{r^2}\hatgamma^{kl}+O_{3,3}(r^{-4})
		\end{align}
		which yield the following immediate consequences along the leaves of background foliation $(S_r)$:
		\begin{lem}\label{lem_prelim_backgroundfol}
			Let $\mathcal{N}$ be asymptotically Schwarzschildean. For a leave $S_r$ of the background foliation, we find
			\begin{align*}
				\ul{\theta}_r&=\frac{2}{r}+O_{3,3}(r^{-3}),\\
				\theta_r&=\frac{2}{r}h(r)+O_{1,1}(r^{-3}),
			\end{align*}
			and 
			\begin{align*}
				\accentset{\circ}{\ul{\chi}}_r&=O_{3,3}(r^{-1}),\\
				\accentset{\circ}{\chi}_r&=O_{1,1}(r^{-1}).
			\end{align*}
			In particular, we have
			\begin{align*}
				\accentset{\circ}{A}_r&=O_{1,1}(r^{-2}),\\
				\tau_r&=O_{2,2}(r^{-2}),\\
				\mathcal{H}^2_r&=\frac{2h}{r^2}+O_{1,1}(r^{-4})
			\end{align*}
			for $r$ sufficiently large.
		\end{lem}
		\begin{bem}\label{bem_prelim_backgroundfol}
			Note that $\gamma_r$ is (uniformly) equivalent to  $\widetilde{\gamma}_r=r^2\hatgamma$ provided $r$ is sufficiently large. In particular, any $(n,m)$-tensor satisfies
			\[
				\btr{T}_{\gamma_r}\le C r^{n-m}\btr{T}_{\hatgamma}
			\]
		\end{bem}
Consider a cross-section $\Sigma=\Sigma_{\omega}$ in $\mathcal{N}$ with induced metric $\gamma_\omega$, where $\omega:\Sbb^2\to\R$ is defined as in Subsection \ref{subsec_generalsetup}. 		
		We collect some prelimary consequences to be used later, where we always assume that $\mathcal{N}$ is asymptotically Schwarzschildean.
		\begin{lem}\label{lem_prelim_chitau}
			Let $(\Sigma_\omega,\gamma_\omega)$ be a spacelike cross section of $\mathcal{N}$, such that $\frac{1}{2}{\rho}\le \omega \le 2{\rho}$, $\rho>2r_0$, and assume that
			\[
			\norm{\frac{\omega}{\rho}}_{C^3\left(\hatgamma\right)}\le c.
			\]
			Then there is a constant $C$ only depending on $c$ such that
			\begin{align}
				\norm{\ul{\theta}-\frac{2}{\omega}}_{C^3(\gamma_\omega)}&\le \frac{C}{\rho^3},\label{eq_prelim_chitau_ultheta}\\
				\newnorm{\accentset{\circ}{\ul{\chi}}}_{C^3(\gamma_\omega)}&\le \frac{C}{\rho^3},\label{eq_prelim_chitau_ulchi}\\
				\norm{\tau}_{C^2(\gamma_\omega)}&\le \frac{C}{\rho^3}\label{eq_prelim_chitau_tau},
			\end{align}
			for $\rho$ sufficiently large.
		\end{lem}
		\begin{proof}
			Equations \eqref{eq_prelim_chitau_ultheta} and \eqref{eq_prelim_chitau_ulchi} directly follow from Lemma \ref{lem_prelim_backgroundfol} and \ref{lem_appx_gammaomegadecay}. Recall that 
			\[
				\tau_j=\zeta_j-\d\ln(\ul{\theta})_j.
			\]
			Using \eqref{eq_prelim_nullgeometry_partiali}, we find
			\[
				\d\ln(\ul{\theta})_j=\d\ln(\ul{\theta})_J+\frac{\d\omega(\partial_J)}{\ul{\theta}}\ul{L}(\ul{\theta}).
			\]
			The identity for $\zeta$, see \cite[Proposition 4.22]{wolff_thesis}, and the evolution equations, see \cite[Proposition 4.27 (v)]{wolff_thesis}, then yield
			\[
			\tau=\tau_r\vert_{r=\omega}-\accentset{\circ}{\ul{\chi}}(\nabla\omega,\cdot)-\frac{1}{\ul{\theta}}\left(\newbtr{\accentset{\circ}{\ul{\chi}}}^2+\overline{\Ric}(\ul{L},\ul{L})\right)\d\omega.
			\]
			Hence, \eqref{eq_prelim_chitau_tau} follows from \eqref{eq_prelim_chitau_ultheta}, \eqref{eq_prelim_chitau_ulchi}, and using Lemma \ref{lem_prelim_backgroundfol}, \ref{lem_appx_c3control}, \ref{lem_appx_gammaomegadecay}, and \ref{lem_appx_curvatureestimates1}.
		\end{proof}
			\begin{lem}\label{lem_prelim_spacetimemeancurvature}
				Let $(\Sigma_\omega,\gamma_\omega)$ be a spacelike cross section of $\mathcal{N}$, such that $\frac{1}{2}{\rho}\le \omega \le 2{\rho}$, $\rho>2r_0$, 
				\[
				\newbtr{\accentset{\circ}{A}}_{\gamma_\omega}\le \frac{K_1}{\rho^2},\text{ }\newbtr{\nabla\accentset{\circ}{A}}_{\gamma_\omega}\le \frac{K_2}{\rho^3}
				\]and assume that
				\[
				\norm{\frac{\omega}{\rho}}_{C^2(\hatgamma)}\le c_1.
				\]
				Then there exists a constant $C$ only depending on $c_1,K_1,K_2$
				\[
				\norm{\mathcal{H}^2-2R-\frac{4(h-1)}{\omega^2}}_{C^1(\gamma_{\omega})}\le \frac{C}{\rho^4},
				\]
				for $\rho$ sufficiently large.
			\end{lem}
			\begin{proof}
				By \eqref{eq_prelim_gauss3} and the identity
\[
\newbtr{\vec{\two}}^2=\frac{1}{2}\mathcal{H}^2+\langle  \accentset{\circ}{\ul{\chi}},\accentset{\circ}{{\chi}}\rangle=\frac{1}{2}\mathcal{H}^2+\spann{\ul{\theta}^{-1}\accentset{\circ}{\ul{\chi}},\accentset{\circ}{A}},
\]				
				 we have
				\[
					\mathcal{H}^2-2R=-2\spann{\ul{\theta}^{-1}\accentset{\circ}{\ul{\chi}},\accentset{\circ}{A}}+2\left(\overline{R}-2\overline{\Ric}(\ul{L},L)-\frac{1}{2}\overline{\Riem}(\ul{L},L,\ul{L},L)\right).
				\]
				Using the assumption on $\accentset{\circ}{A}$ and Lemma \ref{lem_prelim_chitau}, we get
				\[
					\norm{\spann{\ul{\theta}^{-1}\accentset{\circ}{\ul{\chi}},\accentset{\circ}{A}}}_{C^1(\gamma_\omega)}\le \frac{C}{\rho^4},
				\]
				and using Lemma \ref{lem_appx_curvatureestimates1}, we find
				\[
					\norm{2\left(\overline{R}-2\overline{\Ric}(\ul{L},L)-\frac{1}{2}\overline{\Riem}(\ul{L},L,\ul{L},L)\right)-\frac{4(h-1)}{\omega^2}}_{C^1(\gamma_\omega)}\le \frac{C}{\rho^4},
				\]
				which finishes the proof.
			\end{proof}
		
	\section{The a-priori class}\label{sec_apriori}
	From now on, $\mathcal{N}$ will always be an asymptotically Schwarzschildean lightcone according to Definition \ref{defi_asymclassS}. The function representing a spacelike hypersurface $\Sigma\in\mathcal{N}$ is denoted by $\omega:\Sbb^2\to\R$. The induced metric on $\Sigma$ is denoted by $\gamma_{\omega}$ and its area radius is denoted by $\rho$. Further, we define the associated $4$-vector $\textbf{Z}$ of $\Sigma$ and the $3$-vector $\vec{a}\in\R^3$ in terms of $\omega$ as in Subsection \ref{subsec_4vector}.
	We define a suitable a-priori class of surfaces as follows:
		\begin{defi}\label{defi_aprioriclass}
			Let $\sigma$, $B_1$, $B_2$, $B_3$ be strictly positive constants. We define the a-priori class $B_\sigma(B_1,B_2,B_3)$ as
			\[
			B_\sigma(B_1,B_2,B_3):=\left\{\Sigma\colon\sigma-B_1\le \omega\le \sigma+ B_1,\text{ } \newbtr{\accentset{\circ}{A}}\le \frac{B_2}{\sigma^4},\text{ }\newbtr{\nabla\accentset{\circ}{A}}\le \frac{B_3}{\sigma^5}\right\},
			\]
			where $\newbtr{\cdot}=\newbtr{\cdot}_{\gamma_\omega}$.
			We say a surface $\Sigma$ lies strictly in $B_\sigma(B_1,B_2,B_3)$ if $\Sigma\in B_\sigma(B_1,B_2,B_3)$ and all inequalities hold as strict inequalities.
		\end{defi}
		\begin{bem}\label{bem_aprioriclass}
			Note that the leaves $S_r$ of the background foliation lie strictly in $B_\sigma(B_1,B_2,B_3)$ for suitable (fixed) values of $B_1$, $B_2$, $B_3$ (depending on the asymptotics of $\mathcal{N}$) by Lemma \eqref{lem_prelim_backgroundfol} and \eqref{lem_appx_gammaomegadecay}. In fact, one can check that boosted spheres corresponding to $b_{\sigma,\vec{a}}$ as defined in Subsection \ref{subsec_4vector} above, still remain strictly in $B_\sigma(B_1,B_2,B_3)$ for sufficiently small $\vec{a}$, cf. Proposition \ref{prop_apriori} below. The a-priori estimates Proposition \ref{prop_apriori} show that all surfaces in $B_\sigma(B_1,B_2,B_3)$ are close to boosted spheres (under some mild additional assumptions).
		\end{bem}
	\noindent	From the definition, it is immediately clear that for $\Sigma\in B_\sigma(B_1,B_2,B_3)$ we find
		\begin{align}\label{eq_areareadiusB1}
		\sigma-2B_1\le \rho,\widetilde{\rho}\le \sigma+2B_1,
		\end{align}
		where $\widetilde{\rho}$ is the area radius of the metric $\tildegamma_{\omega}=\omega^2\hatgamma$.
		In particular, we find
		\[
		(1+\delta)^{-1}\sigma\le \omega,\rho,\widetilde{\rho}\le (1+\delta)\sigma
		\]
		for any choice of $1\ge \delta>0$ provided $\sigma\ge\sigma_0(B_1,\delta)$. 
		\subsection{A-priori estimates for asymptotically Schwarzschildean $\mathcal{N}$}
		We obtain the following a-priori estimates:
		\begin{prop}\label{prop_apriori}
			Let $\mathcal{N}$ be an asymptotically Schwarzschildean lightcone, and let $\Sigma\in B_\sigma(B_1,B_2,B_3)$. Assume further that 
			\[
				\norm{\frac{\omega}{\rho}}_{C^3\left(\hatgamma\right)}\le c_1
			\]
			for some constant $c_1\ge 10$.
			Then
			\begin{align}\label{eq_apriori_veca}
			\btr{\vec{a}}\le \frac{C(B_1)}{\sigma},
			\end{align}
			provided $\sigma\ge\sigma_0=\sigma_0(B_1,B_3,\delta,c_1)$.
			For $C_0>0$ and $\sigma_0(B_1,B_3,C_0,\delta,c_1)$ such that $\btr{\vec{a}}\le C_0$, we find
			\begin{align}\label{eq_aprioic2alpha}
			\norm{\omega-b_{\widetilde{\rho},\vec{a}}}_{C^{2,\alpha}\left(\hatgamma\right)}\le \frac{C(B_3,C_0,c_1)}{\sigma}.
			\end{align}
			 Moreover,
			\begin{align}\label{eq_aprioriH2}
			\norm{\mathcal{H}^2-\frac{4}{\rho^2}+\frac{8m}{b_{\rho,\vec{a}}^3}}_{C^1(\gamma_\omega)}\le \frac{C(B_3,C_0,c_1)}{\sigma^4}.
			\end{align}
		\end{prop}
		\begin{bem}\label{bem_apriori}
			Note that using $\sigma-B_1\le \omega\le \sigma+B_1$, \eqref{eq_aprioic2alpha} and \eqref{eq_aprioriH2} in particular imply that
			\begin{align}\label{eq_aprioriH2_constant}
				\btr{\mathcal{H}^2-\frac{4}{\rho^2}+\frac{8m}{\sigma^3}}\le \frac{C(B_1,B_3,C_0,c_1)}{\sigma^4}.
			\end{align}
			{Moreover, the estimate on $\btr{\vec{a}}$ implies that there exists a constant $C$ only depending on $B_1$ such that for $\Sigma_{\omega_1}$, $\Sigma_{\omega_2}\in B_\sigma(B_1,B_2,B_3)$ we find
			\[
				\btr{\widetilde{\rho}_1-\widetilde{\rho}_2}+\sigma\btr{\vec{a}_1-\vec{a}_2}\le C\norm{\omega_1-\omega_2}_{L^\infty(\widehat{\gamma})}
			\]
			for $\sigma\ge \sigma_0(B_1,B_2,B_3,c_1)$, cf. Lemma \ref{lem_apriori_vecalipschitz} below.}
		\end{bem}
		In the following, it will suffice to pick some fixed values for $\delta$, $C_0$, e.g. $\delta=1$, $C_0=10$. Thus, we will drop the dependency of $\sigma_0$ on $\delta$, $C_0$ {going forward}.
		\begin{proof}[Proof of Proposition \ref{prop_apriori}]
	Before proving the estimates in the Lemma, we will first establish a $C^1$ bound for the scalar curvature. Contracting Equation \eqref{eq_codazziA} with $\gamma_\omega^{ik}$ yields that
			\begin{align}\label{eq_codazzi}
			\dive\accentset{\circ}{A}_j=\frac{1}{2}\d\mathcal{H}^2_j+\frac{1}{2}\mathcal{H}^2\tau_j-\accentset{\circ}{A}(\vec{\tau},\partial_j)+\ul{\theta}\gamma_\omega^{ik}\overline{\Riem}_{ijkL}.
			\end{align}
			{
			Taking a trace of the identity for $\chi$ in \cite[Proposition 4.22]{wolff_thesis} and multiplying by $\ul{\theta}$, we obtain
			\[
				\mathcal{H}^2=\mathcal{H}^2_r-2\ul{\theta}\Delta_{\gamma_\omega}\omega+\ul{\theta}^2\btr{\nabla\omega}^2-2\ul{\theta}\zeta_r(\nabla\omega).
			\]
			Under the assumptions on $\omega$, it is straightforward to check that $\btr{\mathcal{H}^2}\le \frac{C(c_1)}{\sigma^2}$ by Definition \ref{defi_asymclassS}, Lemma \ref{lem_prelim_backgroundfol}, Remark \ref{bem_prelim_backgroundfol}, and Lemma \ref{lem_appx_c3control}.
			Rearranging \eqref{eq_codazzi} and using the a-priori assumptions on $\mathring{A}$ and \eqref{eq_prelim_chitau_tau}, we get			
			}
			\begin{align}\label{eq_prelim_codazzi2}
			\btr{\frac{1}{2}\d\mathcal{H}^2_j+\ul{\theta}\gamma_\omega^{ik}\overline{\Riem}_{ijkL}}_{\gamma_{\omega}}\le \frac{C(B_3,c_1)}{\sigma^5},
			\end{align}
			provided $\sigma\ge \sigma_0(B_1,B_2,B_3,c_1)$. 
			{Using Lemma \ref{lem_prelim_chitau} and Lemma \ref{lem_appx_curvatureestimates1} we further obtain}
			\begin{align}\label{eq_codazzi_curvature}
			\norm{\ul{\theta}\gamma_\omega^{ik}\overline{\Riem}_{ijkL}-\left(\frac{4(h-1)}{\omega^3}-\frac{2h'}{\omega^2}\right)\d\omega_j}_{C^1(\gamma_\omega)}\le \frac{C(c_1)}{\sigma^5}.
			\end{align}
			Note that
			\[
				\d\left(\frac{(h-1)}{\omega^2}\right)=\left(\frac{h'}{\omega^2}-\frac{2(h-1)}{\omega^3}\right)\d\omega.
			\]
			Hence by Lemma \ref{lem_prelim_spacetimemeancurvature}, Equations \eqref{eq_prelim_codazzi2}, \eqref{eq_codazzi_curvature} imply that
			\begin{align*}
			\btr{\d \scal}_{\gamma_\omega}\le \btr{\frac{1}{2}\d\mathcal{H}^2-2\d\left(\frac{(h-1)}{\omega^2}\right)}_{\gamma_\omega}+\frac{C(c_1)}{\sigma^5}\le \frac{C(B_3,c_1)}{\sigma^5}
			\end{align*}
			provided $\sigma\ge\sigma_0(B_1,B_2,B_3,c_1)$.
			Using Gauss--Bonnet and the fact that $\Sbb^2$ is path-connected, there exists a point $p\in\Sbb^2$ such that $\scal(p)=\fint \scal=\frac{2}{\rho^2}$. Consider any point $q$ on $(\Sbb^2,\gamma_\omega)$ that can be connected to $p$ via a unit-speed curve $s(t)$ of length $L\le 10\sigma$. Then
			\[
			\scal(q)-\scal(p)=\int_s\spann{\nabla R,\dot\gamma}\d t\ge -\frac{C(B_3,c_1)}{\sigma^4}.
			\]
			Hence, for $\sigma\ge\sigma_0(B_1,B_2,B_3,c_1,c_2)$ we have $\scal(q)\ge \frac{1}{\sigma^2}$. Using the Bonnet--Myers Theorem, this implies that geodesics of this length already cover all of $(\Sbb^2,\gamma_\omega)$ for $\sigma$ sufficiently large, which implies
			\begin{align*}
			\btr{\scal-\fint \scal}\le \frac{C(B_3,c_1)}{\sigma^4}.
			\end{align*}
			In summary we have shown that
						\begin{align}\label{eq_aprioriscal}
			\norm{{\scal}-\frac{2}{{\rho}^2}}_{C^1(\widetilde{\gamma}_\omega)}&\le \frac{C(B_3,c_1)}{\sigma^4}.
\end{align}
Now we are going to show \eqref{eq_apriori_veca}. {First, let} $\widetilde{\scal}$ denote the scalar curvature of the metric $\widetilde{\gamma}_\omega=\omega^2\hatgamma$. In particular $\fint\widetilde{\scal}=\frac{2}{\widetilde{\rho^2}}$. Using \eqref{eq_appx_arearadiusdifference} and Lemma \ref{lem_appx_scalarcurv}, \eqref{eq_aprioriscal} implies the respective $C^1$-estimate for $\tildegamma_\omega$, i.e.,
			\begin{align*}
			\norm{\widetilde{\scal}-\frac{2}{\widetilde{\rho}^2}}_{C^1(\widetilde{\gamma}_\omega)}&\le \frac{C(B_3,c_1)}{\sigma^4}.
			\end{align*}
			Considering $\accentset{\vee}{\omega}:=\widetilde{\rho}^{-1}\omega$, we find
			\[
			\norm{\accentset{\vee}{K}-1}_{C^1(\hatgamma)}\le \frac{C(B_3,c_1)}{\sigma^2}
			\]
			by the scaling properties of the scalar curvature, where $\accentset{\vee}{K}$ denotes the Gauss curvature of the metric $\accentset{\vee}{\omega}^2\hatgamma$. Now, for $\sigma\ge\sigma_0(B_1,B_2,B_3,c_1)$ we may utilize Proposition \ref{prop_shiwangwu} to obtain
			\begin{align}\label{eq_proofaprioir1}
			\norm{\accentset{\vee}{\omega}_{-\vec{a}}-1}_{C^{2,\alpha}(\hatgamma)}\le \frac{C(B_3,c_1)}{\sigma^2}.
			\end{align}
			In particular, we find
			\[
			\btr{\frac{\omega\circ\Phi(\vec{x})}{\sqrt{1+\btr{\vec{a}}^2}+\vec{a}\cdot\vec{x}}-\widetilde{\rho}}\le \frac{C(B_3,c_1)}{\sigma}
			\]
			for any $\vec{x}\in\Sbb^2$. Assume $\vec{a}\not=0$, otherwise the first claim is trivial. Then, at $\vec{x}:=\frac{\vec{a}}{\btr{\vec{a}}}$ we find
			\[
			\left(\widetilde{\rho}-\frac{C(B_3,c_1)}{\sigma}\right)\left(\sqrt{1+\btr{\vec{a}}^2}+\btr{\vec{a}}\right)\le \omega\circ\Phi(\vec{x})\le \sigma+B_1=\sigma\left(1+\frac{B_1}{\sigma}\right).
			\]
			For $\sigma\ge \sigma_0(B_1,B_2,B_3,c_1)$ the left-hand side is positive, and thus squaring yields
			\[
			\left(\widetilde{\rho}-\frac{C(B_3,c_1)}{\sigma}\right)^2\left(1+2\btr{\vec{a}}\left(\sqrt{1+\btr{\vec{a}}^2}+\btr{\vec{a}}\right)\right)\le \sigma^2\left(1+\frac{2B_1}{\sigma}+\frac{B_1^2}{\sigma^2}\right),
			\]
			and we conclude
			\begin{align*}
			2\btr{\vec{a}}&\le 2\btr{a}\left(\sqrt{1+\btr{\vec{a}}^2}+\btr{\vec{a}}\right)
			\\&\le \left(\frac{\sigma^2}{\left(\widetilde{\rho}-\frac{C(B_3,c_1)}{\sigma}\right)^2}-1\right)+\frac{\sigma^2}{\left(\widetilde{\rho}-\frac{C(B_3,c_1)}{\sigma}\right)^2}\left(\frac{2B_1}{\sigma}+\frac{B_1^2}{\sigma^2}\right)\\&\le \frac{C(B_1)}{\sigma}
			\end{align*}
			for $\sigma\ge \sigma_0(B_1,B_2,B_3,c_1)$. This proves \eqref{eq_apriori_veca}.
			
			Now assume that $\sigma\ge \sigma_0(B_1,B_2,B_3,C_0,c_1)$ such that $\btr{\vec{a}}\le C_0$ (as we can choose $C_0$ to be a fixed constant going forward we will drop the dependency on $C_0$ in the following). Now, to obtain the desired estimate for $\omega$, \eqref{eq_aprioic2alpha}, we recall that Proposition \ref{prop_shiwangwu} also yields an estimate on $\accentset{\vee}{\omega}$ (without applying a balancing/boost) of the form
			\begin{align}\label{eq_proofaprioir2}
			\norm{\accentset{\vee}{\omega}-b_{\vec{a}}}_{C^{2,\alpha}(\Sbb^2)}\le C(\vec{a})\frac{C(B_3,c_1)}{\sigma^2}.
			\end{align}
			Using $\btr{\vec{a}}\le C_0$, and multiplying the equation by $\widetilde{\rho}$ gives \eqref{eq_aprioic2alpha}. Due to \eqref{eq_appx_differencetensorconformallyround}, we may replace $\widetilde{\rho}$ by $\rho$ in \eqref{eq_aprioic2alpha}. Lastly, the estimate for $\mathcal{H}^2$ follows from combining \eqref{eq_aprioic2alpha}, \eqref{eq_aprioriscal} and Lemma \ref{lem_prelim_spacetimemeancurvature}.
		\end{proof}
		{We now show that Proposition \ref{prop_apriori} implies improved bounds on the derivatives of $\omega$ up to third order. This will later allow us to argue that the assumptions of Proposition \ref{prop_apriori} are preserved under the flow as long as the solutions remains in $B_\sigma(B_1,B_2,B_3)$.}
		\begin{kor}\label{kor_aprior}
			Let $\mathcal{N}$ be an asymptotically Schwarzschildean lightcone, $\Sigma\in B_\sigma(B_1,B_2,B_3)$, and assume that all conditions of Proposition \ref{prop_apriori} are satisfied.
			Then
			\begin{align*}
				\btr{\widehat{\nabla}^l\frac{\omega}{\rho}}_{\hatgamma}\le \frac{C(B_1)}{\sigma}
			\end{align*}
			for all $1\le l\le 3$ provided $\sigma\ge \sigma_0(B_1,B_2,B_3,c_1,c_2)$.
		\end{kor}
		\begin{bem}\label{bem_koraprior}\,
			\begin{enumerate}
				\item[(i)] 
				{Corollary \ref{kor_aprior} in particular implies that within Class $B_\sigma(B_1,B_2,B_3)$ the initial bounds on $\norm{\frac{\omega}{\rho}}_{C^3(\hatgamma)}$ can always be improved to some fixed constant, e.g. $c_1=10$, provided $\sigma\ge\sigma_0(B_1,B_2,B_3,c_1,c_2)$. Hence, although we always have to assume these initial bounds exist a-priori, we omit the dependency of all constants (expect $\sigma_0$) on $c_1$ in the following for simplicity.}
				\item[(ii)] The proof of Corollary \ref{kor_aprior} further shows that the $C^1$-norm of the full scalar second fundamental form is controlled. More precisely,
				\begin{align*}
					\btr{A}_{\gamma_\omega}&\le \frac{10}{\sigma^2}+\frac{C(B_2)}{\sigma^4},\\
					\btr{\nabla A}_{\gamma_\omega}&\le \frac{C(m,B_1,B_3,c_1,c_2)}{\sigma^5}
				\end{align*}
				provided $\sigma\ge\sigma_0(B_1,B_2,B_3,c_1)$, cf. Equation \eqref{eq_scalar2FF_C1} below, and where the bound on $A$ directly follows from Proposition \ref{prop_apriori} and the definition of the a-priori class.
			\end{enumerate}
		\end{bem}
		\begin{proof}[Proof of Corollary \ref{kor_aprior}]
			First note that Equation \eqref{eq_appx_arearadiusdifference} directly implies
			\[
			\norm{\frac{\omega}{\rho}-\frac{\omega}{\widetilde{\rho}}}_{C^2(\hatgamma)}=\norm{\frac{\omega}{\rho}}_{C^2(\hatgamma)}\btr{1-\frac{\rho}{\widetilde{\rho}}}\le \frac{C(c_1)}{\sigma^2}.
			\]
			Combining this with \eqref{eq_proofaprioir2} we get
			\begin{align*}
			\btr{\widehat{\nabla}\frac{\omega}{\rho}}_{\hatgamma}\le \btr{\widehat{\nabla}b_{\vec{a}}}_{\hatgamma}+\frac{C(B_3,c_1)}{\sigma^2},\\
			\btr{\widehat{\Hess}\frac{\omega}{\rho}}_{\hatgamma}\le \btr{\widehat{\Hess}\,b_{\vec{a}}}_{\hatgamma}+\frac{C(B_3,c_1)}{\sigma^2}.
			\end{align*}
			{In view of \eqref{eq_apriori_veca}, it suffices to show that $\btr{\widehat{\nabla}b_{\vec{a}}}_{\hatgamma}, \btr{\widehat{\Hess}\,b_{\vec{a}}}_{\hatgamma}\le C\btr{\vec{a}}$ in order to get the bound on the first two derivatives}. Without loss of generality we can assume that $\vec{a}=(0,0,a)$, $a:=\btr{\vec{a}}$, as rotations in $\R^3$ act as isometries on $(\Sbb^2,\hatgamma)$. In angular coordinates $(\theta,\varphi)$, we find
			\[
			\hatgamma=\d\theta^2+\sin(\theta)^2\d\varphi^2,
			\]
			and
			\[
			b:=b_{\vec{a}}=\frac{1}{\sqrt{1+a^2}-a\cos\theta}.
			\]
			Direct computation gives
			\begin{align*}
			b_{,\theta}&=\frac{-a\sin\theta}{\left(\sqrt{1+a^2}-a\cos\theta\right)^2},\\
			b_{,\theta\theta}&=\frac{a^2+a^2\sin\theta^2-a\sqrt{1+a^2}\cos\theta}{\left(\sqrt{1+a^2}-a\cos\theta\right)^3},
			\end{align*}
			which gives
			\begin{align*}
			\btr{\widehat{\nabla}b}^2_{\Sbb^2}&\le \frac{\btr{\vec{a}}^2}{{\left(\sqrt{1+a^2}-a\cos\theta\right)^4}},\\
			\btr{\widehat{\Hess}\,b}^2_{\Sbb^2}&\le C\btr{\vec{a}}^2\left(\frac{1}{{\left(\sqrt{1+a^2}-a\cos\theta\right)^6}}+\frac{1}{\left(\sqrt{1+a^2}-a\cos\theta\right)^4}\right).
			\end{align*}
			{
			Now note that there exists a unique $\alpha\in[0,\infty)$ such that $\sinh(\alpha)=a$, $\cosh(\alpha)=\sqrt{1+a^2}$, and that
			\[
			\frac{1}{\cosh(\alpha)-\sinh(\alpha)}=\cosh(\alpha)+\sinh(\alpha)=e^\alpha
			\]
			Hence, 
			\[
			\frac{1}{\sqrt{1+a^2}-a\cos\theta}\le\frac{1}{\sqrt{1+a^2}-a}\le C=C(C_0)
			\]
			as the exponential function is strictly increasing. This yields the claim up until $l=2$.}
			
			To prove the claim for $l=3$ first recall that 
			{by \cite[Proposition 4.22]{wolff_thesis}, we have
			\begin{align}\label{eq_Adef}
			A=\ul{\theta}\chi=A_r-2\ul{\theta}\Hess_{\gamma_\omega}\omega+\btr{\nabla\omega}^2_{\gamma_\omega}\ul{\theta}\ul{\chi}-2\ul{\theta}\d\omega\otimes\zeta_r-2\ul{\theta}\zeta_r\otimes\d\omega,
			\end{align}
			as we assume the background foliation to be geodesic ($\kappa=0$).}
			Rearranging the above equation and taking a tensor derivative yields
			\begin{align}
			\begin{split}\label{eq_apriori_nablahess}
			\nabla_i\Hess_{\gamma_\omega}\omega_{jk}
			=&\,-\frac{1}{2\ul{\theta}}\nabla_i A_{jk}+\frac{1}{2\ul{\theta}^2}\d\ul{\theta}_i A_{jk}+\frac{1}{2\ul{\theta}}\nabla (A_r)_{jk}-\frac{1}{2\ul{\theta}^2}\d\ul{\theta}_i (A_r)_{jk}\\
			&\,+\frac{1}{2}\d\left(\btr{\nabla\omega}_{\gamma_\omega}^2\right)_i\ul{\chi}_{jk}+\frac{1}{2}\btr{\nabla\omega}^2_{\gamma_\omega}\nabla_i\ul{\chi}_{jk}-2\Hess_{\gamma_\omega}\omega_{ij}\otimes(\zeta_r)_k\\
			&\,-2(\zeta_r)_j\otimes\Hess_{\gamma_\omega}\omega_{ik}-2\d\omega_j\otimes\nabla_i(\zeta_r)_k-2\nabla_i(\zeta_r)_j\otimes\d\omega_j.
			\end{split}
			\end{align}
			We first aim to bound $\nabla_iA_{jk}$. 
			{Note that Equation \eqref{eq_aprioriH2} directly yields
			\[
			\btr{\nabla\mathcal{H}^2}_{\gamma_\omega}\le \btr{\frac{24m}{b_{\rho,\vec{a}}^4}\nabla b_{\rho,\vec{a}}}_{\gamma_\omega}+\frac{C(B_3,c_1)}{\sigma^5}\le \frac{C(m,B_1,B_3,c_1)}{\sigma^5},
			\]
			where we utilized the bounds on $b_{\rho,\vec{a}}$ derived above and the scaling properties of $\gamma_\omega$. As $\mathcal{H}^2=\tr A$ and $\Sigma\in B_\sigma(B_1,B_2,B_3)$, this implies}
			\begin{align}\label{eq_scalar2FF_C1}
			\btr{\nabla A}_{\gamma_\omega}\le \frac{C(m,B_1,B_3,c_1)}{\sigma^5}.
			\end{align}
			{By Lemma \ref{lem_prelim_chitau}, we have
			\[
			\btr{\frac{1}{\ul{\theta}}\nabla A}_{\gamma_\omega}\le \frac{C(m,B_1,B_3,c_1)}{\sigma^4}.
			\]
			We further note that by Lemma \ref{lem_prelim_backgroundfol} and Lemma \ref{lem_appx_gammaomegadecay}, the already proven improved bound on $\nabla\omega$ now implies}
			\[
			\btr{\nabla\ul{\theta}}_{\gamma_\omega}\le\frac{C(B_1)}{\sigma^3},\text{ }\btr{\nabla\mathcal{H}^2_r}_{\gamma_\omega}\le\frac{C(B_1)}{\sigma^4}.
			\]
			Then, using Lemma \ref{lem_prelim_backgroundfol}, Equation \eqref{eq_aprioriH2}, and Lemma \ref{lem_appx_gammaomegadecay} we find
			\[
			\btr{\frac{1}{2\ul{\theta}^2}\d\ul{\theta}_i A_{jk}+\frac{1}{2\ul{\theta}}\nabla (A_r)_{jk}-\frac{1}{2\ul{\theta}^2}\d\ul{\theta}_i (A_r)_{jk}}_{\gamma_\omega}\le \frac{C(B_1)}{\sigma^3}
			\]
			provided $\sigma\ge \sigma_0(B_1,B_2,B_3,c_1,c_2)$ {}\footnote{By using Lemma \ref{lem_appx_gammaomegadecay}, the constant will a-priori also depend on $c_1$. However, as we already have proven the claim for $l=1$, we may assume $c_1=10$ for $\sigma_0$ sufficiently big without loss of generality.}. It remains to find sufficient bounds on the second and third line of \eqref{eq_apriori_nablahess}. Using the claim for $l=2$, the decay conditions on $\ul{\chi}$, $\zeta_r$ in Definition \ref{defi_asymclassS}, and Lemma \ref{lem_appx_c3control} (with $l=2$) and \ref{lem_appx_gammaomegadecay}, it is straightforward to check that
			\begin{align*}
			\btr{\frac{1}{2}\d\left(\btr{\nabla\omega}_{\gamma_\omega}^2\right)_i\ul{\chi}_{jk}+\frac{1}{2}\btr{\nabla\omega}^2_{\gamma_\omega}\nabla_i\ul{\chi}_{jk}-2\Hess_{\gamma_\omega}\omega_{ij}\otimes(\zeta_r)_k}_{\gamma_\omega}\le \frac{C(B_1,c_1,c_2)}{\sigma^4},\\
			\btr{-2(\zeta_r)_j\otimes\Hess_{\gamma_\omega}\omega_{ik}-2\d\omega_j\otimes\nabla_i(\zeta_r)_k-2\nabla_i(\zeta_r)_j\otimes\d\omega_j}_{\gamma_\omega}\le \frac{C(B_1,c_1,c_2)}{\sigma^4}.
			\end{align*}
			{Hence, Equation \eqref{eq_apriori_nablahess} and the above bounds imply}
			\[
			\btr{\nabla\Hess\omega}_{\gamma_{\omega}}\le\frac{C(B_1)}{\sigma^3}
			\]
			provided $\sigma\ge\sigma_0(B_1,B_2,B_3,c_1,c_2)$. The claim then follows from Lemma \ref{lem_appx_c3control}.
		\end{proof}
		{
		\begin{lem}\label{lem_apriori_vecalipschitz}
			Assume $\Sigma_{\omega_1},\Sigma_{\omega_2}\in B_\sigma(B_1,B_2,B_3)$ and the assumptions of Proposition \ref{prop_apriori} hold. Then
			\[
			\btr{\vec{a}_1-\vec{a}_2}\le \frac{{C(B_1)}}{\sigma}\norm{\omega_1-\omega_2}_{L^\infty(\widehat{\gamma})}
			\]
			for $\sigma\ge\sigma_0(B_1,B_2,B_3,c_1)$.
		\end{lem}}
		\begin{proof}
			{
			Recall that $\vec{a}$ is uniquely determined via the associated $4$-vector ${\textbf{Z}}$ as defined in Subsection \ref{subsec_4vector}. In particular, we find
			\[
				\frac{\vec{a}^i}{\sqrt{1+\btr{\vec{a}}^2}}=\frac{\int_{\Sbb^2}\omega^3f_i\d\widehat{\mu}}{\int_{\Sbb^2}\omega^3\d\widehat{\mu}}.
			\]
			We define $f_0:=1$, and note that is immediate to check that
			\[
			\btr{\int_{\Sbb^2}\left(\omega_1^3-\omega_2^3\right)f_\alpha\d\widehat{\mu}}\le C\sigma^2\norm{\omega_1-\omega_2}_{L^\infty(\widehat{\gamma})}
			\]
			for all $\alpha\in \{0,1,2,3\}$ provided $\sigma\ge\sigma_0(B_1)$,  as $\btr{f_\alpha}\le 1$. This then implies that
			\[
			\btr{\frac{\vec{a}_1}{\sqrt{1+\btr{\vec{a}_1}^2}}-\frac{\vec{a}_2}{\sqrt{1+\btr{\vec{a}_2}^2}}}\le \frac{C(B_1)}{\sigma}\norm{\omega_1-\omega_2}_{L^\infty(\widehat{\gamma})}.
			\]
			Finally, using \eqref{eq_apriori_veca} and the third binomial formula, we find
			\begin{align*}
				\btr{\frac{\vec{a}_1}{\sqrt{1+\btr{\vec{a}_1}^2}}-\frac{\vec{a}_2}{\sqrt{1+\btr{\vec{a}_2}^2}}}&\ge \frac{\btr{\vec{a}_1-\vec{a}_2}}{\sqrt{1+\btr{\vec{a}_1^2}}}-\btr{\vec{a}_2}\btr{\frac{1}{\sqrt{1+\btr{\vec{a}_1}^2}}-\frac{1}{\sqrt{1+\btr{\vec{a}_2}^2}}}\\
				&\ge C\btr{\vec{a}_1-\vec{a}_2}-\frac{C(B_1)}{\sigma}\left(\btr{\vec{a}_1}^2-\btr{\vec{a}_2}^2\right)\\
				&\ge C\left(1-\frac{C(B_1)}{\sigma^2}\right)\btr{\vec{a}_1-\vec{a}_2},
			\end{align*}
			where we again used the third binomial formula and the reverse triangle inequality in the last line. This then implies the claim for $\sigma\ge \sigma_0(B_1,B_2,B_3,c_1)$.}
		\end{proof}
	
	\subsection{Stability of surfaces in the a-priori class}\label{subsec_stability}
	
		{Recall that the well-known Raychaudhuri optical equations, cf.\ \cite[Proposition 4.27]{wolff_thesis}, show that for a variation $x\colon(-\varepsilon,\varepsilon)\times\Sigma\to\mathcal{N}$ of a spacelike cross section $\Sigma$ with $\frac{\d}{\d s}=f\ul{L}$, we find that
		\[
			\frac{\d}{\d s}\btr{\Sigma_s}=\int_{\Sigma_s}\ul{\theta}f\d\mu_s.
		\]
		Moreover, if $\ul{\theta}>0$ on $\operatorname{Im}(x)\subseteq \mathcal{N}$ and $\tau$ is defined via \eqref{eq_prelim_deftau}, then \cite[Proposition 4.27 (v), (vi)]{wolff_thesis} and the product rule yield that}
		\[
			\frac{\d}{\d s}\mathcal{H}^2=J(\ul{\theta}f),
		\]
		where the 
		{Jacobi} operator $J$ 
		{in this setting} is given by
		\begin{align}
			\begin{split}\label{eq_stability1}
			J(f):=&\,-2\Delta f-4\tau(\nabla f)-\mathcal{H}^2f-f\left(\overline{\Ric}(\ul{L},L)-\frac{1}{2}\overline{\Riem}(\ul{L},L,L,\ul{L})\right)\\
			&\,-f\frac{\mathcal{H}^2}{\ul{\theta}^2}\left(\newbtr{\accentset{\circ}{\ul{\chi}}}^2+\overline{\Ric}(\ul{L},\ul{L})\right)-f\left(\frac{1}{\ul{\theta}}\spann{\accentset{\circ}{\ul{\chi}},\accentset{\circ}{A}}+2\dive\tau+2\btr{\tau}^2\right).
			\end{split}
		\end{align}
		{
		\begin{defi}\label{defi_apriori_stability}
			We say a spacelike cross section $\Sigma$ of a null cone $\mathcal{N}$ ($\ul{\theta}>0$) is stable under area preserving variations if there exists a non-negative constant $c\ge 0$ such that
			\[
				\int_{\Sigma}fJ(f)\d\mu\ge c\int_{\Sigma}f^2\d\mu
			\]
			for all $f\in W^{2,2}(\Sigma)$ such that $\int f\d\mu=0$. We say $\Sigma$ is strictly stable (under area preserving variations) if $c>0$.
		\end{defi}
		\begin{bem}\label{bem_apriori_stability}
			Here, stability of spacelike cross sections under area preserving variations is formulated in analogy to the stability of surfaces under volume preserving variations in the Riemannian setting. In particular, we formulate the stability of STCMC surfaces in analogy to the stability of CMC surfaces. Although $J$ is in general not self-adjoint, stability (under area preserving variations) implies that $J$ has a non-negative first eigenvalue in a similar sense to the setting of MOTS stability, cf. \cite{anderssonmarssimon}. For details, we refer to \cite[Section 4]{anderssonmarssimon}.
		\end{bem}
		}
		We will now show that surfaces within the a-priori class are strictly stable (under area preserving variations) in asymptotically Schwarzschildean lightcones (of positive mass):
		\begin{prop}\label{prop_strictstability}
			Let $\mathcal{N}$ be an asymptotically Schwarzschildean lightcone (of positive mass $m>0$). Let $\Sigma=\Sigma_{\omega}$ in $B_\sigma(B_1,B_2,B_3)$, and assume that all assumptions of Proposition \ref{prop_apriori} are satisfied.
			Then 
			{$\Sigma$ is strictly stable under area preserving variations} provided $\sigma\ge \sigma_0(m,B_1,B_2,B_3,c_1)$ with
			\begin{align}\label{eq_stability2}
			\int_\Sigma J(f)f\d\mu\ge \frac{6m}{\sigma^3}\int_\Sigma f^2\d\mu
			\end{align}
			for all $f\in {W}^{2,2}(\Sigma)$ with $\int_\Sigma f\d\mu=0$, where $\d\mu=\d\mu_{\gamma_{\omega}}$.
		\end{prop}
		\begin{proof}
			Note that integration by parts yields
			\begin{align*}
				\int_{\Sigma}J(f)f\d\mu
				=&\,\int_\Sigma \left(2\btr{\nabla f}^2-2\scal f^2-\left((\mathcal{H}^2-2\scal)+\overline{\Ric}(\ul{L},L)-\frac{1}{2}\overline{\Riem}(\ul{L},L,L,\ul{L})\right)f^2\right)\d\mu\\
				&\,-\int_\Sigma \left(\frac{\mathcal{H}^2}{\ul{\theta}^2}\left(\newbtr{\accentset{\circ}{\ul{\chi}}}^2+\overline{\Ric}(\ul{L},\ul{L})\right)f^2+\left(\frac{1}{\ul{\theta}}\spann{\accentset{\circ}{\ul{\chi}},\accentset{\circ}{A}}+2\btr{\tau}^2\right)f^2\right)\d\mu.
			\end{align*}
			Using Lemma \ref{lem_prelim_chitau}, \ref{lem_prelim_spacetimemeancurvature}, Proposition \ref{prop_apriori} and Lemma \ref{lem_appx_curvatureestimates1} we find
			\begin{align*}
				\int_{\Sigma}J(f)f\d\mu\ge \int_{\Sigma}\left(2\btr{\nabla f}^2-2\scal f^2+\left(\frac{4(1-h)}{\omega^2}+\frac{2h'}{\omega}\right)f^2\right)\d\mu-\frac{C(B_1,B_2,B_3)}{\sigma^4}\int_\Sigma f^2\d\mu.
			\end{align*}
		Using \eqref{eq_aprioriscal}	and that $\mathcal{N}$ is asymptotically Schwarzschildean, i.e., $h(r)=1-\frac{2m}{r}+O_4(r^{-2})$ , we find
			\begin{align*}
			\int_{\Sigma}J(f)f\d\mu&\ge \int_{\Sigma}\left(2\btr{\nabla f}^2-2\int_{\Sigma}\left(\min\limits_{\Sigma}\scal\right) f^2+\frac{12m}{\omega^3}f^2\right)\d\mu-\frac{C(B_1,B_2,B_3)}{\sigma^4}\int_\Sigma f^2\d\mu\\
			&\ge \frac{6m}{\sigma^3}\int_\Sigma f^2\d\mu,
			\end{align*}
			provided $\sigma\ge \sigma_0(m,B_1,B_2,B_3,c_1)$. Here, we used that the integral over the first two terms is non-negative due to \cite[Lemma 2.9]{sauter}. 			
		\end{proof}
		Additionally, we find the following Sobolev inequality:
		\begin{lem}\label{prop_sobolevineq}
			Let $\mathcal{N}$ be asymptotically Schwarzschildean, $\Sigma\in B_\sigma(B_1,B_2,B_3)$, and assume that all assumptions in Proposition \ref{prop_apriori} are satisfied. Then there exists a uniform constant $C>0$ such that
			\begin{align}\label{eq_propsobolev1}
			\norm{f}_{L^2(\gamma_\omega)}\le C\int_\Sigma\left(\btr{\nabla f} +\rho\mathcal{H}^2\btr{f}\right)\d\mu
			\end{align}
			for all $f\in W^{1,2}(\Sigma)$ provided $\sigma\ge \sigma_0(m,B_1,B_2,B_3,c_1)$. In particular,
			\begin{align}\label{eq_propsobolev2}
			\norm{f}_{L^2(\gamma_\omega)}\le C\rho^{-1}\norm{f}_{W^{1,1}(\gamma_\omega)}.
			\end{align}
		\end{lem}
		
		\begin{proof}
Because of the estimate \eqref{eq_aprioriH2} on $\mathcal{H}^2$, \eqref{eq_propsobolev1} directly implies \eqref{eq_propsobolev2}. Therefore, it remains to show the first inequality. Recall that the Michael--Sobolev inequality yields that for any immersed surface $S$ in $(\R^3,\delta)$, we have
			\[
			\norm{f}_{L^2(S)}\le c\int_S\left(\btr{\nabla f}+\btr{H}\btr{f}\right)\d\mu
			\]
			for a universal constant $c$, and where $H$ denotes the mean curvature of $S$. In particular, the estimate holds for the round sphere $(\Sbb^2,\hatgamma)$, where $H=2$. Because of \eqref{eq_aprioriH2}, we have $\rho^2\mathcal{H}^2\ge 2$ for $\sigma\ge \sigma_0(m,B_1,B_2,B_3,c_1)$. Because $\gamma_{\omega}\leq C\rho \hatgamma$, it follows that
			\begin{align*}
			\norm{f}_{L^2(\gamma_\omega)}&\le C\rho\norm{f}_{L^2(\hatgamma)}\\
			&\le C\rho\int_{\Sbb^2}\left(\btr{\nabla f}_{\hatgamma}+2\btr{f}\right)\d\widehat{\mu}\\
			&\le C\rho\int_{\Sigma}\left(\omega^{-1}\btr{\nabla f}_{\gamma_{\omega}}+\mathcal{H}^2\btr{f}\right)\d\mu\\
			&\le C\int_{\Sigma}\left(\btr{\nabla f}_{\gamma_{\omega}}+\rho\mathcal{H}^2\btr{f}\right)\d\mu,
			\end{align*}
			which finishes the proof.
		\end{proof}
\begin{bem}
			In view of the Michael--Simon Sobolev inequality employed in the proof above, we expect the optimal estimate to be 
			\begin{align}
			\norm{f}_{L^2(\gamma_\omega)}\le C\int_\Sigma\left(\btr{\nabla f} +\frac{1}{2}\rho\mathcal{H}^2\btr{f}\right)\d\mu,
			\end{align}
			which should follow with some additional effort in analogy to the work of Huisken--Yau, cf. \cite[Proposition 5.4]{huiskenyau}.
		\end{bem}		
		
		Using \cite[Lemma A.2]{cederbaumsakovich}, we directly obtain the following Corollary:
		\begin{kor}\label{kor_sobolevineq}
			Under the assumptions of Lemma \ref{prop_sobolevineq}, we have
			\[
			\norm{f}_{L^\infty(\gamma_\omega)}\le C\sigma^{-1}\norm{f}_{W^{2,2}(\gamma_\omega)},
			\]
			where $C$ is a universal constant, for all $f\in W^{2,2}(\Sigma)$ provided $\sigma\ge \sigma_0(m,B_1,B_2,B_3,c_1)$.
		\end{kor}
		Next, we refine our spectral analysis of the stability operator $J$ in Proposition \ref{prop_strictstability}.
		\begin{prop}\label{prop_stability2}
			Let $\mathcal{N}$ be asymptotically Schwarzschildean. Let $\Sigma\in B_\sigma(B_1,B_2,B_3)$, and assume that all assumptions of Proposition \ref{prop_apriori} are satisfied. If $\sigma\ge \sigma_0(m,B_1,B_2,B_3,c_1)$, then 
			\[
			J\colon W^{2,2}(\gamma_\omega)\to L^2(\gamma_\omega)
			\] 
			is invertible, and there exists a constant $C=C(m,B_1,B_2,B_3)$ such that
			\[
			\norm{f}_{W^{2,2}(\gamma_\omega)}\le C\sigma^3\norm{J(f)}_{L^2(\gamma_\omega)}
			\]
			for all $f\in W^{2,2}(\gamma_\omega)$.
		\end{prop}
		\begin{proof} Throughout the proof, norms will always be measured with respect to the metric $\gamma_{\omega}$.		
			Let $f\in W^{2,2}$. As $L^2\subset L^1$, $f_0:=\fint_{\Sigma} f\d\mu$ is well defined, and by definition, we find
			\[
			\norm{f}_{L^2}^2=\norm{f_0}_{L^2}^2+\norm{f-f_0}_{L^2}^2.
			\]
			In the following, $C$ will always denote a constant depending only on $m,B_1,B_2,B_3$, which may vary from line to line.
			Note that by \eqref{eq_aprioriH2_constant} and Lemma \ref{lem_appx_curvatureestimates1}, we find
			\begin{align}\label{eq_prop_stability1}
			\btr{\left(-\mathcal{H}^2-\overline{\Ric}(\ul{L},L)+\frac{1}{2}\overline{\Riem}(\ul{L},L,L,\ul{L})\right)-\left(-\frac{4}{\rho^2}+\frac{12m}{\sigma^3}\right)}\le \frac{C}{\sigma^4},
			\end{align}
			which in view of the estimates in the proof of Proposition \ref{prop_strictstability} directly implies that
			\begin{align}\label{eq_prop_stability2}
				\btr{J(f_0)-f_0\left(-\frac{4}{\rho^2}+\frac{12m}{\sigma^3}\right)}\le \frac{C}{\sigma^4}\btr{f_0}.
			\end{align}
			Using \eqref{eq_prop_stability2}, we obtain 
			\begin{align*}
			\btr{\int_{\Sigma} J(f_0)J(f-f_0)\d\mu}
			&\le \frac{C}{\sigma^4}\int_\Sigma\btr{f_0J(f-f_0)\d\mu}+\btr{-\frac{4}{\rho^2}+\frac{12m}{\sigma^3}}\btr{\int_\Sigma f_0J(f-f_0)\d\mu}.
			\end{align*}
			Using integration by parts and that $f_0$ is constant, we find
			\begin{align*}
				\int_\Sigma f_0J(f-f_0)\d\mu=&\,-\int_\Sigma f_0(f-f_0)\left(\mathcal{H}^2+\overline{\Ric}(\ul{L},L)-\frac{1}{2}\overline{\Riem}(\ul{L},L,L,\ul{L})\right)\d\mu\\
				&\,-\int_\Sigma f_0(f-f_0)\left(\frac{\mathcal{H}^2}{\ul{\theta}^2}\left(\newbtr{\accentset{\circ}{\ul{\chi}}}^2+\overline{\Ric}(\ul{L},\ul{L})\right)+\frac{1}{\ul{\theta}}\spann{\accentset{\circ}{\ul{\chi}},\accentset{\circ}{A}}-2\dive\tau+2\btr{\tau}^2\right)\d\mu.
			\end{align*}
			Using \eqref{eq_prop_stability1} to estimate the first line, and Lemma \ref{lem_prelim_chitau}, \ref{lem_prelim_spacetimemeancurvature}, Proposition \ref{prop_apriori} and Lemma \ref{lem_appx_curvatureestimates1} to estimate the second line, we find
				\begin{align*}
				\btr{\int_\Sigma f_0J(f-f_0)\d\mu}&\le \btr{-\frac{4}{\rho^2}+\frac{12m}{\sigma^3}}\btr{\int_{\Sigma}f_0(f-f_0)\d\mu}+\frac{C}{\sigma^4}\int_{\Sigma}\btr{f_0(f-f_0)\d\mu}
				\\				&=\frac{C}{\sigma^4}\int_{\Sigma}\btr{f_0(f-f_0)}\d\mu,
			\end{align*}
			and the equality holds because $\int_{\Sigma} (f-f_0)\d\mu=0$.
			Hence,
			\begin{align}\label{eq_Jf_f_0}
				\btr{\int_{\Sigma} J(f_0)J(f-f_0)\d\mu}
				&\le \frac{C}{\sigma^4}\int_{\Sigma}\btr{f_0J(f-f_0)}\d\mu+\frac{C}{\sigma^6}\int_{\Sigma} \btr{f_0(f-f_0)}\d\mu.
			\end{align}
			Moreover, Proposition \ref{prop_strictstability} and the H\"older inequality yield
			\[
			\norm{J(f-f_0)}_{L^2}\norm{f-f_0}_{L^2}\ge \int_{\Sigma} J(f-f_0)\cdot(f-f_0)\d\mu\ge \frac{6m}{\sigma^3}\norm{f-f_0}^2_{L^2},
			\]
			which implies
			\begin{align}\label{eq_f_f0}
			\norm{J(f-f_0)}_{L^2}^2\ge \frac{(6m)^2}{\sigma^6}\norm{f-f_0}^2_{L^2}.
			\end{align}
			Similarly, \eqref{eq_prop_stability2} implies
			\begin{align}\label{eq_f0}
			\norm{J(f_0)}_{L^2}^2\ge \left(\frac{16}{\sigma^4}-\frac{C}{\sigma^5}\right)\norm{f_0}_{L^2}^2.
			\end{align}
			Using \eqref{eq_Jf_f_0}, \eqref{eq_f_f0}, and \eqref{eq_f0}, we conclude that
			\begin{align*}
			\norm{J(f)}^2_{L^2}
			=&\,\norm{J(f_0)}_{L^2}^2+\norm{J(f-f_0)}_{L^2}^2+2\int_{\Sigma} J(f_0)J(f-f_0)\d\mu\\
			\ge&\, \left(\frac{16}{\sigma^4}-\frac{C}{\sigma^5}\right)\norm{f_0}_{L^2}^2+\norm{J(f-f_0)}_{L^2}^2\\
			&\,-\frac{C}{\sigma^4}\norm{f_0}_{L^2}\norm{J(f-f_0)}_{L^2}-\frac{C}{\sigma^6}\norm{f_0}_{L^2}\norm{f-f_0}_{L^2}\\
			\ge &\,\left(\frac{16}{\sigma^4}-\frac{C}{\sigma^5}\right)\norm{f_0}_{L^2}^2+\left(1-\frac{C}{\sigma^3}\right)\norm{J(f-f_0)}_{L^2}^2\\
			&\,-\frac{C}{\sigma^6}\norm{f_0}_{L^2}\norm{f-f_0}_{L^2}\\
			\ge&\, \left(\frac{16}{\sigma^4}-\frac{C}{\sigma^5}\right)\norm{f_0}_{L^2}^2+\frac{(6m)^2}{\sigma^6}\left(1-\frac{C}{\sigma^3}\right)\norm{f-f_0}_{L^2}^2\\
			&\,-\frac{C}{\sigma^6}\norm{f_0}_{L^2}\norm{f-f_0}_{L^2}\\
			\ge&\, \left(\frac{16}{\sigma^4}-\frac{C}{\sigma^5}\right)\norm{f_0}_{L^2}^2+\frac{(6m)^2}{\sigma^6}\left(1-\frac{C}{\sigma}-\frac{C}{\sigma^3}\right)\norm{f-f_0}_{L^2}^2\\
			\ge &\, \frac{9m^2}{\sigma^6}\left(\norm{f_0}_{L^2}^2+\norm{f-f_0}_{L^2}^2\right)=\frac{9m^2}{\sigma^6}\norm{f}_{L^2}^2,
			\end{align*}
			provided $\sigma\ge \sigma_0(m,B_1,B_2,B_3,c_1)$, where we used Young's inequality in the third and fifth line. 
			Hence, for $\sigma\ge \sigma_0(m,B_1,B_2,B_3,c_1)$
			\begin{align}\label{eq_prop_sobolevineq2}
			\norm{f}_{L^2}\le \frac{\sigma^3}{3m}\norm{J(f)}_{L^2},
			\end{align}
			which in particular implies that $J$ is injective. Next, note that integration by parts shows that the adjoint $J^*$ of $J$ is given by
			\begin{align*}
				J^*(f):=&\,-2\Delta f+4\tau(\nabla f)-\mathcal{H}^2f-f\left(\overline{\Ric}(\ul{L},L)-\frac{1}{2}\overline{\Riem}(\ul{L},L,L,\ul{L})\right)\\
				&\,-f\frac{\mathcal{H}^2}{\ul{\theta}^2}\left(\newbtr{\accentset{\circ}{\ul{\chi}}}^2+\overline{\Ric}(\ul{L},\ul{L})\right)-f\left(\frac{1}{\ul{\theta}}\spann{\accentset{\circ}{\ul{\chi}},\accentset{\circ}{A}}-2\dive\tau+2\btr{\tau}^2\right),
			\end{align*}
			which agrees with $J$ up to replacing $\tau$ by $-\tau$. Since this sign was inconsequential for establishing the estimates as above, we also find that $J^*$ is injective. Hence, $J$ is invertible by the Fredholm-Alternative \cite[Appendix I, Theorem 31]{besse}.
			
			It remains to prove the $W^{2,2}$-estimate. In view of the a-priori estimates, Proposition \ref{prop_apriori}, and using \cite[Corollary 2.10]{sauter}, we obtain
			\[
			\norm{f}_{W^{2,2}}\le C\sigma^2\norm{\Delta f}_{L^2}+\norm{f_0}_{L^2}\le C\sigma^2\norm{\Delta f}_{L^2}+\norm{f}_{L^2}.
			\]
			Together with
			\[
			\norm{\Delta f}_{L^2}\le C\norm{J(f)}_{L^2}+4\norm{\tau(\nabla f)}_{L^2}+\frac{C}{\sigma^2}\norm{f}_{L^2},
			\]
			and \eqref{eq_prop_sobolevineq2}, we therefore get
			\begin{align*}
			\norm{f}_{W^{2,2}}&\le C\left(\sigma^2(\norm{J(f)}_{L^2}+4\norm{\tau(\nabla f)}_{L^2})+\norm{f}_{L^2}\right)\\
			&\le C\left(\sigma^2+\frac{\sigma^3}{3m}\right)\norm{J(f)}_{L^2}+C\sigma^2\norm{\tau(\nabla f)}_{L^2}\\
						&\le C\sigma^3\norm{J(f)}_{L^2}+\frac{C}{\sigma^2}\norm{f}_{W^{2,2}}.
			\end{align*}
			In the last inequality, we used Lemma \ref{lem_prelim_chitau} and the definition of the weighted norms.
	Choosing $\sigma_0$ sufficiently large and bringing the last term to left hand side yields		the desired inequality.
		\end{proof}		
	\subsection{Uniqueness of surfaces with prescribed $\mathcal{H}^2$ in $B_\sigma(B_1,B_2,B_3)$}
	The goal of this subsection is to prove a uniqueness statement about cross sections with prescrived spacetime mean curvature within the a-priori class.
\begin{prop}\label{prop_uniqueness}
		For any positive constants $B_1,B_2,B_3$, and $c_1\ge 10$ there exists a number $\sigma_0(m,B_1,B_2,B_3,c_1)$ such that the following holds for all $\sigma\ge\sigma_0$:
		If $\Sigma_0,\Sigma_1$ in $B_\sigma(B_1,B_2,B_3)$ with
		\[
			\norm{\frac{\omega_i}{\rho_i}}_{C^3(\widehat{\gamma})}\le c_1,\text{ }i\in\{0,1\}
		\]
		satisfy $\mathcal{H}^2_0=\mathcal{H}^2_1$,
		then $\omega_0=\omega_1$, i.e., $\Sigma_0=\Sigma_1$.
	\end{prop}	
	
		The proof of this proposition relies on a very carefully chosen interpolation between $\Sigma_0$ and $\Sigma_1$, for which we have to establish a couple of technical estimates first. We therefore postpone the proof of the proposition to the end of this subsection.
		
		Consider $\Sigma_0$, $\Sigma_1\in B_\sigma(B_1,B_2,B_3)$ corresponding to the functions $\omega_0$, $\omega_1$ which satisfy the assumptions of Proposition \ref{prop_apriori}. For simplicity, we write $b_{\rho_i,\vec{a}_i}=\rho_ib_i$, $i=0,1$, cf. Subsection \ref{subsec_4vector}. Then, we choose the smooth path
	\[
	\omega_t:=(\omega_0-\rho_0b_0)+\rho(t)b_t+t\left((\omega_1-\rho_1b_1)-(\omega_0-\rho_0b_0)\right)
	\]
	from $\omega_0$ to $\omega_1$, where
	\begin{align*}
	\rho(t)&=\rho_0+t(\rho_1-\rho_0),\\
	\vec{a}(t)&=\vec{a}_0+t(\vec{a}_1-\vec{a}_0),\\
	b_t(\vec{x})&=\frac{1}{\sqrt{1+\btr{\vec{a}(t)}^2}-\vec{a}(t)\cdot\vec{x}}.
	\end{align*}
	Let $\Sigma_t:=\Sigma_{\omega_t}$ denote the corresponding cross section of $\mathcal{N}$, and $\gamma_t$ its induced metric. Note that $\rho(t)b_t$ is a smooth path of 
	{boosted spheres} from $\rho_0b_0$ to $\rho_1b_1$ that linearly interpolates between both the area radius $\rho(t)$ and boost vector $\vec{a}(t)$, respectively. On the other hand, $\omega_t-\rho(t)b_t$ interpolates linearly between the difference $\omega_i-\rho_ib_i$. By \eqref{eq_aprioic2alpha}, it is immediate to check that
	\begin{align}\label{eq_omegat1}
	\norm{\omega_t-\rho(t)b_t}_{C^2(\widehat{\gamma})}\le
		(1-t)\norm{\omega_0-\rho_0b_0}_{C^2(\widehat{\gamma})}+	t\norm{\omega_1-\rho_1b_1}_{C^2(\widehat{\gamma})}\le
	 \frac{C}{\sigma}
	\end{align}
	for all $t\in[0,1]$. Moreover, \eqref{eq_areareadiusB1} and \eqref{eq_apriori_veca} yield that
	\begin{align}\label{eq_omegat2}
	\norm{\rho(t)b_t-\rho_ib_i}_{C^2(\widehat{\gamma})}\le C(B_1).
	\end{align}
	Hence, it is easy to check that
	\begin{align}\label{eq_omegat3}
	\norm{\omega_t-\omega_i}_{C^2(\widehat{\gamma})}+\norm{\omega_t-\rho_ib_i}_{C^2(\widehat{\gamma})}\le C
	\end{align}
	for all $t \in[0,1]$.
	We first show that $\Sigma_t$ lies in the a-priori class. In the following, we  denote the area radii of $\Sigma_t$ by $\rho_t$.
	\begin{lem}\label{lem_ometatapriori}
		Assume $\Sigma_0, \Sigma_1\in B_\sigma(B_1,B_2,B_3)$ and $\omega_0,\omega_1$ satisfy the assumptions of Proposition \ref{prop_apriori}. Then $\Sigma_t\in B_\sigma(\widetilde{B}_1,\widetilde{B}_2,\widetilde{B}_3)$, and
		\begin{align}\label{eq_derivatives_along_path}
		\btr{\widehat{\nabla}^l\frac{\omega_t}{\rho_t}}\le c_1,\qquad 1\leq l\leq 3,
		\end{align}
		where $\widetilde{B}_1,\widetilde{B}_2,\widetilde{B}_3$ only depend on $B_1$, $B_2$, $B_3$, $c_1$ for $\sigma\ge\sigma_0(B_1,B_2,B_3,c_1)$.
	\end{lem}
	\begin{proof}
			Note first that
			\[
			\omega_t=(1-t)\omega_0+t\omega_1+\left((\rho(t)b_t-\rho_0b_0)-t(\rho_1b_1-\rho_0b_0)\right).
			\]
			By \eqref{eq_omegat2} and because
				\[
			\sigma-
			{B_1}\le \omega_i\le \sigma+
			{B_1}
			\]		
		 for $i=0,1$ as $\Sigma_0,\Sigma_1\in B_\sigma(B_1,B_2,B_3)$, it is easy to check that 
		 {the function} $\omega_t$ of the cross section $\Sigma_t$ satisfies
			\[
			\sigma-C(B_1)\le \omega_t\le \sigma+C(B_1)
			\]
			{In particular, $\frac{1}{2}\sigma\le \omega_t,\rho_t\le 2\sigma$ for $\sigma\ge\sigma_0(B_1)$}. 
			{As in the proof of Corollary \ref{kor_aprior}, one can directly check that Inequality \eqref{eq_apriori_veca} yields
			\begin{align}\label{eq_gradientboost}
			\btr{\widehat{\nabla}^lb_t}\le \frac{C(B_1)}{\sigma} \text{ }(1\le l \le3).
			\end{align}}
			Recall furthermore that by Corollary \ref{kor_aprior}, we have
				\begin{align*}
			\btr{\widehat{\nabla}^l\frac{\omega_i}{\rho_i}}\le \frac{C(B_1)}{\sigma} \text{ }(1\le l \le3)
			\end{align*}
			for $i=0,1$. Combining these two estimates, we get
			\begin{align}\label{eq_uniqueness_bettergradient}
			\btr{\widehat{\nabla}^l\frac{\omega_t}{\rho_t}}\le \frac{C(B_1)}{\sigma} \text{ }(1\le l \le3).
			\end{align}
{Hence, the bound \eqref{eq_derivatives_along_path} follows for $\sigma\ge\sigma_0(B_1,c_1)$}. To finish the proof, it remains to establish estimates on $\accentset{\circ}{A}$. 
			 Using the trace-free part of Equation \eqref{eq_Adef}, we note that Lemma \ref{lem_prelim_backgroundfol} and \ref{lem_prelim_chitau} imply that
			\[
			\norm{\accentset{\circ}{A}_t+2\ul{\theta}_t\accentset{\circ}{\Hess}_{\gamma_t}\omega_t}_{C^1(\gamma_t)}\le \frac{C}{\sigma^4}.
			\]
			In particular, $\omega_t\in B_\sigma(\widetilde{B}_1,\widetilde{B}_2,\widetilde{B}_3)$ as claimed if and only if 
			\[
			\norm{\accentset{\circ}{\Hess}_{\gamma_t}\omega_t}_{C^1(\gamma_t)}\le\frac{C}{\sigma^3}
			\]
			with $C$ only depending on $B_1,B_2,B_3$. As $\norm{Q_{IJ}^K}_{C^1(\gamma_t)}\le \frac{C}{\sigma^3}$, cf. Lemma \ref{lem_appx_differencetensors}, it remains to show that $\norm{\accentset{\circ}{{\Hess}}_{\omega_t}\omega_t}_{C^1(\widetilde{\gamma}_t)}\le \frac{C}{\sigma^3}$, where ${\Hess}_{\omega_t}$ denotes the Hessian with respect to the conformally round metric $\widetilde{\gamma_t}=\omega_t^2\widehat{\gamma}$.
			
			We note that for any function $f$, and conformally round metrics $v_1^2\widehat{\gamma}$, $v_2^2\widehat{\gamma}$, it holds that
			\[
			\accentset{\circ}{\Hess}_{v_1}f-\accentset{\circ}{\Hess}_{v_2}f={\Hess}_{v_1}f-{\Hess}_{v_2}f=f_K\left({}^{v_2}\widetilde{\Gamma}_{IJ}^K-{}^{v_1}\widetilde{\Gamma}_{IJ}^K\right).
			\]
			Using \eqref{eq_appx_differencetensorconformallyround}, we note that if $\norm{v_1-v_2}_{C^2(\widehat{\gamma})}\le C_1$, $\norm{\widehat{\nabla f}}_{C^1(\widehat{\gamma})}\le C_2$, then 
			\begin{align}\label{eq_tracefree_hession}
			\norm{\accentset{\circ}{\Hess}_{v_1}f-\accentset{\circ}{\Hess}_{v_2}f}_{C^1(\widehat{\gamma})}\le \frac{C(C_1,C_2)}{\sigma}.
			\end{align}
			
			By linearity and the uniform equivalence of the $C^1$-norms along the path, see Lemma \ref{lem_appx_norms_equivalence}, we have
			\begin{align*}
			\norm{\accentset{\circ}{{\Hess}}_{\omega_t}\omega_t}_{C^1(\widetilde{\gamma}_t)}&\le\,(1-t)\left(\norm{\accentset{\circ}{{\Hess}}_{\omega_t}\omega_0}_{C^1(\widetilde{\gamma}_t)}+\norm{\accentset{\circ}{{\Hess}}_{\omega_t}\rho_0b_0}_{C^1(\widetilde{\gamma}_t)}\right)\\
			&\qquad \,+t\left(\norm{\accentset{\circ}{{\Hess}}_{\omega_t}\omega_1}_{C^1(\widetilde{\gamma}_t)}+\norm{\accentset{\circ}{{\Hess}}_{\omega_t}\rho_1b_1}_{C^1(\widetilde{\gamma}_t)}\right)+ \norm{\accentset{\circ}{{\Hess}}_{\omega_t}\rho(t)b_t}_{C^1(\widetilde{\gamma}_t)}\\
			&\le  C\left(\norm{\accentset{\circ}{{\Hess}}_{\omega_0}\omega_0}_{C^1(\omega_0^2\widehat{\gamma})}+\norm{\accentset{\circ}{{\Hess}}_{\rho_0b_0}\rho_0b_0}_{C^1(\rho_0^2b_0^2\widehat{\gamma})}+\norm{\accentset{\circ}{{\Hess}}_{\omega_1}\omega_1}_{C^1(\omega_1^2\widehat{\gamma})} \right)\\
			&\qquad \,+C\left(\norm{\accentset{\circ}{{\Hess}}_{\rho_1b_1}\rho_1 b_1}_{C^1(\rho_1^2b_1^2\widehat{\gamma})}+\norm{\accentset{\circ}{{\Hess}}_{\rho(t)b_t}\rho(t)b_t}_{C^1(\rho(t)^2b_t^2\widehat{\gamma})}\right)+\frac{C(B_1,c_1)}{\sigma^3}.
			\end{align*}
			In the last inequality, we combined \eqref{eq_tracefree_hession} with \eqref{eq_omegat1} and \eqref{eq_omegat3}. Furthermore, we used Corollary \ref{kor_aprior}, \eqref{eq_gradientboost}, the inequality
			\[
				\norm{T}_{C^1(\widetilde{\gamma}_t)}\le \frac{C(c_1)}{\sigma^2}\norm{T}_{C^1(\widehat{\gamma})},
			\]
			which holds for any $(0,2)$ tensor by the scaling properties of the metric, and \eqref{eq_appx_differencetensorconformallyround}. Finally, using that
			\[
			\norm{\accentset{\circ}{{\Hess}}_{\omega_0}\omega_0}_{C^1(\omega_0^2\widehat{\gamma})}+\norm{\accentset{\circ}{{\Hess}}_{\omega_1}\omega_1}_{C^1(\omega_1^2\widehat{\gamma})}\le \frac{C(B_1, B_2,B_3)}{\sigma^3},
			\]
			as $\Sigma_1,\Sigma_2\in B_\sigma(B_1,B_2,B_3)$, and using that $\rho(t)b_t$ is a family of exact 
			, i.e.,
			\[
			\accentset{\circ}{{\Hess}}_{\rho(t)b_t}\rho(t)b_t
			{=-\rho(t)b_t^2\accentset{\circ}{\Hess}_{\hatgamma}\left(\frac{1}{b_t}\right)}=0
			\]
			for all $t\in[0,1]$, cf. \cite[Remark 2]{wolff1}, the desired estimate follows. This shows that $\Sigma_t\in B_\sigma(\widetilde{B}_1, \widetilde{B}_2,\widetilde{B}_3)$.
	\end{proof}
	The key ingredient to showing uniqueness for prescribed 
	{spacetime mean} curvature surfaces within the a-priori class is the following technical lemma. A crucial observation in establishing this technical lemma is that $\rho(t)b_t'$ lies in the kernel of the 
	{Jacobi} operator in the Minkowski lightcone along the path $\rho(t)b_t$ of 
	.
		\begin{lem}\label{lem_appx_bootstrap}
			Assume $\Sigma_0$, $\Sigma_1\in B_\sigma(B_1,B_2,B_3)$ and $\omega_0$, $\omega_1$ satisfy the assumptions of Proposition \ref{prop_apriori}. Consider $\omega_t$ as above. Assume additionally, that
			\[
			\btr{\rho_0-\rho_1}\le \frac{C_\rho}{\sigma}
			\]
			for some positive constant $C_\rho>0$. Then, for $\sigma\ge \sigma_0(B_1,B_2,B_3,c_1,C_{\rho})$, there exists a constant $C=C(B_1,B_2,B_3,C_\rho)$ such that
			\[
			\norm{J_s(\ul{\theta}_sf_s)-J_t(\ul{\theta}_tf_t)}_{L^2(\gamma_s)}\le \frac{C}{\sigma^5}\sup\limits_{\tau\in[0,1]}\norm{f_\tau}_{W^{2,2}(\gamma_\tau)}
			\]
			for all $s,t\in[0,1]$, where $f_t:=\frac{\d}{\d t}\omega_t$, and where we write $\ul{\theta}_t$ to emphasize that $\ul{\theta}$ is evaluated along $\Sigma_t$.
		\end{lem}
		\begin{proof}
			Note that
			\[
			f_t:=\frac{\d}{\d t}\omega_t=(\rho_1-\rho_0)b_t+\rho(t)b_t'+(\omega_1-\rho_1b_1)-(\omega_0-\rho_0b_0),
			\]
			and for simplicity, we will write 
			\[
			f_t=\widetilde{f}_t+R+\widetilde{R}_t,
			\]
			where
			\begin{align*}
			\widetilde{f}_t&:=\rho(t)b_t',\\
			R&:=(\omega_1-\rho_1b_1)-(\omega_0-\rho_0b_0)+(\rho_1-\rho_0)b_0,\\
			\widetilde{R}_t&:=(\rho_1-\rho_0)(b_t-b_0).
			\end{align*}
			By assumption, we have
			\begin{align*}
			\btr{\rho_0-\rho_1}\le \frac{C_\rho}{\sigma}.
			\end{align*}
			and due to \eqref{eq_appx_arearadiusdifference}, a similar estimate holds for $\widetilde{\rho}_0$, $\widetilde{\rho}_1$. By Lemma \ref{lem_ometatapriori} $\Sigma_t$ remains in the a-priori class, where we write $B_1$, $B_2$, $B_3$ for simplicity and choose $\sigma\ge \sigma_0$ sufficiently large. A direct computation using the area radius of the conformally round metrics then yields that
			\begin{align}\label{eq_arearadiusUniqueness}
			\btr{\rho_t-\rho(t)}+\btr{\rho_t-\rho_s}\le \frac{C(C_\rho)}{\sigma}
			\end{align}
			for all $t,s\in[0,1]$, 
			{where we recall that $\rho_t$ denotes the area radius of $\Sigma_t$ and $\rho(t)$ the linear interpolation between $\rho_0$ and $\rho_1$}. Using \eqref{eq_apriori_veca} and \eqref{eq_aprioic2alpha}, it is straightforward to check that
			\begin{align}
			\norm{\widetilde{f}_s-\widetilde{f}_t}_{L^2(\gamma_s)}&\le C\sigma^2\btr{\vec{a}_1-\vec{a}_0}^2\le C\sigma\btr{\vec{a}_1-\vec{a}_0},\label{eq_bootstrap_ftilde}\\
			\norm{R}_{C^2(\gamma_s)}&\le \frac{C}{\sigma},\label{eq_bootstrap_R}\\
			\norm{\widetilde{R}_t}_{W^{2,2}(\gamma_s)}&\le C\btr{\vec{a}_1-\vec{a}_0}\label{eq_bootstrap_Rtilde},
			\end{align}
			where here and in the following, we will frequently use that all $\gamma_t$-norms are uniformly equivalent due \eqref{eq_uniqueness_bettergradient} and Lemma \ref{lem_appx_norms_equivalence}.
			Consider the auxiliary operator
			\[
			\widetilde{J}(f):=-\frac{2}{\omega^2}\widehat{\Delta}f-\mathcal{H}^2f+\frac{4m}{\omega^3}f,
			\]
			and note that $\widetilde{J}$ agrees with $J$ in the Schwarzschild lightcone. Similar to the proof of Proposition \ref{prop_strictstability}, using Lemmas \ref{lem_prelim_chitau}, \ref{lem_appx_curvatureestimates1}, and also using Lemma \ref{lem_appx_differencetensors} to estimate the difference tensor $\widetilde{Q}_{IJ}^K$, we find that 
			\begin{align}\label{eq_bootstrap_estimate1}
			\norm{J_t(\ul{\theta}_tf_t)-\widetilde{J}_t\left(\frac{2}{\omega_t}f_t\right)}_{L^2(\gamma_t)}\le \frac{C}{\sigma^5}\norm{f_t}_{W^{2,2}(\gamma_t)}.
			\end{align}
			For convenience, let us define $\widetilde{\ul{\theta}}_t:=\frac{2}{\omega_t}$, and note that we can write $\widetilde{J}_t(\widetilde{\ul{\theta}}_tf_t)$ as
			\[
			\widetilde{J}_t(\widetilde{\ul{\theta}}_tf_t)=\left(-\frac{2}{\omega_t^2}\widehat{\Delta}\left(\widetilde{\ul{\theta}}_tf_t\right)-\frac{4}{\rho(t)^2}\widetilde{\ul{\theta}}_tf_t\right)+\left(\frac{8m}{\omega_t^3}-\left(\mathcal{H}^2_t-\frac{4}{\rho(t)^2}\right)\right)\widetilde{\ul{\theta}}_tf_t
			\]
			Using Equations \eqref{eq_aprioriH2_constant}, \eqref{eq_arearadiusUniqueness}, it follows that for $B_t:=\left(\frac{8m}{\omega_t^3}-\left(\mathcal{H}^2_t-\frac{4}{\rho(t)^2}\right)\right)$
			\begin{align*}
			\norm{B_s\widetilde{\ul{\theta}}_sf_s-B_t\widetilde{\ul{\theta}}_tf_t}_{L^2(\gamma_s)}\le \frac{C}{\sigma_5}\norm{f_s}_{L^2(\gamma_s)}+\frac{C}{\sigma^4}\left(\norm{\widetilde{f}_s-\widetilde{f}_t}_{L^2(\gamma_s)}+\norm{\widetilde{R}_s}_{L^2(\gamma_s)}+\norm{\widetilde{R}_t}_{L^2(\gamma_s)}\right),
			\end{align*}
			where we used the decomposition of $f_t$, and note that $R$ is time-independent. Using \eqref{eq_bootstrap_ftilde}, \eqref{eq_bootstrap_Rtilde}, we conclude that
			\begin{align}\label{eq_bootstrap_estimate2}
			\norm{B_s\widetilde{\ul{\theta}}_sf_s-B_t\widetilde{\ul{\theta}}_tf_t}_{L^2(\gamma_s)}\le \frac{C}{\sigma_5}\norm{f_s}_{W^{2,2}(\gamma_s)}+\frac{C}{\sigma^3}\btr{\vec{a}_1-\vec{a}_0}.
			\end{align}
			It remains to consider the operator $\overline{J}_{\omega_t}:=-\frac{2}{\omega_t^2}\widehat{\Delta}-\frac{4}{\rho(t)^2}$. It is easy to check that for any smooth function $f$, it holds that
			\begin{align*}
			\norm{\overline{J}_{\omega_t}\left(\frac{2f}{\omega_t}\right)}_{L^2(\gamma_t)}&\le \frac{C}{\sigma^3}\norm{f}_{W^{2,2}(\gamma_t)},\\
			\norm{\overline{J}_{\omega_s}\left(\frac{2f}{\omega_s}\right)-\overline{J}_{\omega_t}\left(\frac{2f}{\omega_t}\right)}_{L^2(\gamma_s)}&\le \frac{C}{\sigma^4}\norm{f}_{C^2(\gamma_s)}\norm{\omega_s-\omega_t}_{W^{2,2}(\gamma_s)}.
			\end{align*}
			Finally notice that $\overline{J}_{\rho(t)b_t}\left(\frac{2b_t'}{b_t}\right)=0$, as $\overline{J}_{\rho(t)b_t}$ agrees with $J$ in the Minkowski lightcone for $\omega_t=\rho(t)b_t$	
			\footnote{One can either check that this is true by direct computation or noting that for $\omega_t=\rho(t)b_t$ one finds $\mathcal{H}^2_t=\frac{4}{\rho(t)}$ in the Minkowski lightcone. In particular, $\mathcal{H}^2_t$ only changes with respect to the area radius, but not under the variation of $\vec{a}(t)$.}.
			One then finds that
			\begin{align*}
			\btr{\overline{J}_{\omega_t}\left(\frac{2\rho(t)b_t'}{\omega_t}\right)}&=\btr{\overline{J}_{\omega_t}\left(\frac{2\rho(t)b_t'}{\omega_t}\right)-\overline{J}_{\rho(t)b_t}\left(\frac{2b_t'}{b_t}\right)}\le \frac{C}{\rho^3}\norm{{\omega}_t-\rho(t)b_t}_{C^2(\widehat{\gamma})}\norm{b_t'}_{C^2(\widehat{\gamma})}.
			\end{align*}
			We note that
			\[
			b_t'=-b_t^2\spann{\frac{\vec{a(t)}}{\sqrt{1+\btr{\vec{a}(t)}^2}}-\vec{x},\vec{a}_1-\vec{a}_0},
			\]
			and hence $\norm{b_t'}_{C^2(\widehat{\gamma})}\le C\btr{\vec{a}_1-\vec{a}_0}$. Using \eqref{eq_omegat1}, we find
			\begin{align*}
			\norm{\overline{J}_{\omega_t}\left(\frac{2\rho(t)b_t'}{\omega_t}\right)}_{L^2(\gamma_t)}\le \frac{C}{\sigma^3}\btr{\vec{a}_1-\vec{a}_0}.
			\end{align*}
			Combining the above estimates and using the decomposition $f_t=\widetilde{f}_t+R+\widetilde{R}_t$, where $\widetilde{f}_t=\rho(t)b_t'$, we find that
			\begin{align*}
			\norm{\overline{J}_{\omega_s}\left(\frac{2f_s}{\omega_s}\right)-\overline{J}_{\omega_t}\left(\frac{2f_t}{\omega_t}\right)}_{L^2(\gamma_s)}
			\le&\, \norm{\overline{J}_{\omega_s}\left(\frac{2\widetilde{f}_s}{\omega_s}\right)}_{L^2(\gamma_s)}+\norm{\overline{J}_{\omega_t}\left(\frac{2\widetilde{f}_t}{\omega_t}\right)}_{L^2(\gamma_t)}\\
			&\,+\norm{\overline{J}_{\omega_s}\left(\frac{2R}{\omega_s}\right)-\overline{J}_{\omega_t}\left(\frac{2R}{\omega_t}\right)}_{L^2(\gamma_s)}\\
			&\,+\norm{\overline{J}_{\omega_s}\left(\frac{2\widetilde{R}_s}{\omega_s}\right)}_{L^2(\gamma_s)}+\norm{\overline{J}_{\omega_t}\left(\frac{2\widetilde{R}_t}{\omega_t}\right)}_{L^2(\gamma_t)}\\
			\le&\,\frac{C}{\sigma^3}\left(\btr{\vec{a}_1-\vec{a}_0}+\norm{\widetilde{R}_s}_{W^{2,2}(\gamma_s)}+\norm{\widetilde{R}_t}_{W^{2,2}(\gamma_t)}\right)\\
			&\,+\frac{C}{\sigma^4}\norm{R}_{C^2(\gamma_s)}\norm{\omega_s-\omega_t}_{W^{2,2}(\gamma_s)}.
			\end{align*}
			Hence, using \eqref{eq_bootstrap_R}, \eqref{eq_bootstrap_Rtilde}, we find
			\begin{align}\label{eq_bootstrap_estimate3}
			\norm{\overline{J}_{\omega_s}\left(\frac{2f_s}{\omega_s}\right)-\overline{J}_{\omega_t}\left(\frac{2f_t}{\omega_t}\right)}_{L^2(\gamma_s)}\le \frac{C}{\sigma^3}\btr{\vec{a}_1-\vec{a}_0}+\frac{C}{\sigma^5}\norm{\omega_s-\omega_t}_{W^{2,2}(\gamma_s)}.
			\end{align}
			Combining \eqref{eq_bootstrap_estimate1}, \eqref{eq_bootstrap_estimate2}, \eqref{eq_bootstrap_estimate3}, we conclude that for all $t,s\in[0,1]$
			\begin{align}\label{eq_bootstrapEnd}
			\norm{J_s(\ul{\theta}_sf_s)-J_t(\ul{\theta}_tf_t)}_{L^2(\gamma_s)}\le \frac{C}{\sigma^3} \btr{\vec{a}_1-\vec{a}_0}+\frac{C}{\sigma^5}\left(\sup\limits_{\tau\in[0,1]}\norm{f_\tau}_{W^{2,2}(\gamma_\tau)}+\norm{\omega_s-\omega_t}_{W^{2,2}(\gamma_s)}\right).
			\end{align}
			Notice that by using the H\"older inequality, it is straightforward to check that for all $s,t\in[0,1]$
			\[
			\norm{\omega_s-\omega_t}_{W^{2,2}(\gamma_s)}\le C\sup\limits_{\tau\in[0,1]}\norm{f_\tau}_{W^{2,2}(\gamma_\tau)}.
			\]
			Moreover, by Lemma \ref{lem_apriori_vecalipschitz} and the Sobolev inequality, Corollary \ref{kor_sobolevineq}, we find that
			\[
			\btr{\vec{a}_0-\vec{a}_1}\le \frac{C}{\sigma^2}\norm{\omega_0-\omega_1}_{W^{2,2}(\gamma_0)}\le \frac{C}{\sigma^2}\sup\limits_{\tau\in[0,1]}\norm{f_\tau}_{W^{2,2}(\gamma_\tau)}.
			\]
			Combining these observations with \eqref{eq_bootstrapEnd} yields the claim.
		\end{proof}
	We are now ready to prove the desired uniqueness statement.
	\begin{proof}[Proof of Proposition \ref{prop_uniqueness}]
		We first note that for $\Sigma_0,\Sigma_1\in B_\sigma(B_1,B_2,B_3)$ with $\mathcal{H}^2_0=\mathcal{H}^2_1$, \eqref{eq_aprioriH2_constant} implies that
		\[
		\btr{\frac{1}{\rho_0^2}-\frac{1}{\rho_1^2}}\le \frac{C(B_1,B_3,C_0,c_1)}{\sigma^4}.
		\]
		Hence
		\begin{align*}
		\btr{\rho_0-\rho_1}\le \frac{C_\rho}{\sigma},
		\end{align*}
		where $C_\rho=C_{\rho}(B_1,B_3,C_0,c_1)$. By Lemma \ref{lem_ometatapriori} $\Sigma_t$ remains in the a-priori class, where we write $B_1$, $B_2$, $B_3$ for simplicity and choose $\sigma\ge \sigma_0$ sufficiently large. 
		Now, by assumption, we have
		\[
		\mathcal{H}^2_0(\vec{x})=\mathcal{H}^2_1(\vec{x}),
		\]
		for any $\vec{x}\in\Sbb^2$. In particular, the fundamental theorem of calculus implies
		\[
			0=\mathcal{H}^2_1(\vec{x})-\mathcal{H}^2(\vec{x})=\int\limits_0^1\frac{\d}{\d t}\mathcal{H}^2_t(\vec{x})\d t=\int\limits_0^1(J_t(\ul{\theta}_tf_t))(\vec{x})\d t.
		\]
		Again, we write $\ul{\theta}_t$ to emphasize that $\ul{\theta}$ is evaluated along $\Sigma_t$.
		In particular, 
		\[
		(J_s(\ul{\theta}_sf_s))(\vec{x})=\int\limits_0^1(J_s(\ul{\theta}_sf_s))(\vec{x})-(J_t(\ul{\theta}_tf_t))(\vec{x})\d t
		\]
		for all $s\in[0,1]$. By Propsition \ref{prop_stability2}, and Lemma \ref{lem_prelim_chitau}, we find
		\begin{align}\label{eq_bootstrapStart}
		\frac{C}{\sigma^4}\norm{f_s}_{W^{2,2}(\gamma_s)}\le \norm{J_s(\ul{\theta}_sf_s)}_{L^2(\gamma_s)}\le \sup\limits_{t\in[0,1]}\norm{J_s(\ul{\theta}_sf)-J_t(\ul{\theta}_tf_t)}_{L^2(\gamma_s)}
		\end{align}
		for all $s\in[0,1]$, where $C=C(B_1,B_2,B_3,m)$, and we obtained the upper bound on the $L^2$-norm from the above equality by exchanging the order of integration and applying the H\"older inequality. Since the estimate is valid for all $s$, Lemma \ref{lem_appx_bootstrap} implies that
		\[
			\sigma\left(\sup\limits_{t\in[0,1]}\norm{f_t}_{W^{2,2}(\gamma_t)}\right)\le C(B_1,B_2,B_3,m)\left(\sup\limits_{t\in[0,1]}\norm{f_t}_{W^{2,2}(\gamma_t)}\right).
		\]
		This implies that $\sup\limits_{t\in[0,1]}\norm{f_t}_{W^{2,2}(\gamma_t)}=0$ for $\sigma$ sufficiently large. In particular $\omega_0=\omega_1$, concluding the proof.
	\end{proof}
	\section{Area preserving null mean curvature flow}\label{sec_APNMCF}
	
	We now want to employ \emph{area preserving null mean curvature flow} (APNMCF) as a tool to find STCMC surfaces. Thus, the speed of the flow $x\colon [0,T)\times \Sigma\to\mathcal{N}$ is defined as
	\begin{align}\label{eq_APNMCF}
	\frac{\d}{\d t} x=-\frac{1}{2\ul{\theta}}\left(\mathcal{H}^2-\fint\mathcal{H}^2\right)\ul{L}.
	\end{align}
	Note that the flow is invariant under the choice of null generator $\ul{L}$ and, it is immediate from \eqref{eq_APNMCF} that the flow is only defined on lightcones, i.e., when $\ul{\theta}>0$. Moreover, STCMC surfaces are indeed the stationary points of the flow, and the flow preserves area as the name suggests, cf. Subsection \ref{subsec_stability}. In the (standard) lightcones of the Minkowski, de\,Sitter and Anti-de\,Sitter spactime, the flow is equivalent to Hamilton's Ricci flow \cite{wolff1, wolff6}
	As argued in \cite{huisken1}, short-time existence follows similarly as for the respective mean curvature flow, in this case null mean curvature flow \cite{roeschscheuer}. For more details, we also refer to the doctoral thesis of Pihan \cite{pihan}.
	
	Although we would hope that the flow is well-behaved for general surfaces at least in the Schwarzschild lightcone, here we will only consider solutions within the a-priori class $B_\sigma(B_1,B_2,B_3)$. In particular, we aim to show that the a-priori class is preserved under the flow for suitable choices of $B_1$, $B_2$, $B_3$, which will lead to long-time existence and convergence of the flow. Our approach is strongly motivated by the work of Huisken--Yau \cite{huiskenyau}.
	
	As the area is preserved under the flow, we have
	\[
	\rho(t)=\rho(0),
	\]
	and consequently 
	\[
	\btr{\rho(t)-\sigma}\le C(\Sigma_0).
	\]
	\subsection{Long-time existence and convergence in $B_\sigma(B_1,B_2,B_3)$}
	
	The first aim of this section is to prove the following theorem:
	\begin{thm}\label{thm_longtime}
		Let $B_1^0$, $B_2^0$, $B_3^0$, and $c_1\ge 10$ be strictly positive constants. Then there exist constants $B_2(B_2^0,m)$, $B_3(B_2^0,B_3^0,m)$, $B_1(B_2,B_3,B_1^0,m)$ and $\sigma_0(m,B_1,B_2,B_3,c_1)$ such that the following holds:\newline
		If $\sigma\ge\sigma_0$, and $\Sigma_0\in B_\sigma(B_1^0,B_2^0,B_3^0)$ such that 
		\[
			\norm{\frac{\omega_0}{\rho_0}}_{C^3(\hatgamma)}\le c_1,
		\]
		then the solution of APNMCF exists for all times and remains in $B_\sigma(B_1,B_2,B_3)$.
	\end{thm}

	In the following Propositions, we will always assume that there exist fixed constants ${C_1, C_2\ge 10}$ such that
	\[
		\norm{\frac{\omega_t}{\rho_0}}_{C^3(\hatgamma)}\le C_1,\text{ and }\btr{\mathcal{H}_t^2}\le \frac{C_2}{\rho^2},
	\]
	for all $\Sigma_t$ and for all times $t\in[0,T)$. Later, we will easily verify that these conditions are indeed satisfied and preserved along the flow as long as it remains in the a-priori class by using Proposition \ref{prop_apriori} and Corollary \ref{kor_aprior}.
	
	We begin by establishing $C_0$-estimates under the flow. Recall that 
	{the asymptotics ensure that }$\ul{\theta}>0$ in the asymptotic region $(r_0,\infty)\times\Sbb^2$ of $\mathcal{N}$ {}\footnote{Here, we identify $(r_0,\infty)\times\Sbb^2$ with a subset in $\mathcal{N}$ via the diffeomorphism induced by the background foliation} 
	{for $r_0$ sufficiently large}. Provided the solution lies in $B_\sigma(B_1,B_2,B_3)$, we can ensure $\omega_t\ge \sigma-B_1>r_0$ provided $\sigma\ge \sigma_0(B_1)$ is sufficiently large. From now on, we will always assume that this is the case.
	
	\begin{prop}\label{prop_B1}
		Let $(\Sigma_t)_{[0,T)}$ be a smooth solution to APNMCF starting at $\Sigma_0$, such that $\Sigma_t\in B_\sigma(B_1,B_2,B_3)$ for all $t\in[0,T)$. Let $B_1^0$ be a non-negative constant (independent of $\sigma$) such that $\btr{\omega(0,\cdot)-\sigma}\le B_1^0$. Then we have
		\begin{align*}
		\omega_{\max}(t)&\le \sigma+C(B_1^0,B_3,C_1,m),\\
		\omega_{\min}(t)&\ge \sigma-C(B_1^0,B_3,C_1,m),
		\end{align*}
		for all $t\in[0,T)$ provided $\sigma\ge \sigma_0(m,B_1,B_2,B_3,C_1,C_2)$.
	\end{prop}
	\begin{proof}
		Note that by \eqref{eq_appx_arearadiusdifference} and the assumption on $\Sigma_0$, we have that
		\[
			\rho=\rho(t)=\rho(0)\le \sigma+B_1^0+1
		\]
		for $\sigma\ge\sigma_0$ sufficiently large. For convenience, we define $K_0:=B_1^0+1\ge 1$, and we first aim to prove the upper bound.
		
		Let $D>K_0$ and assume there exists a first time $t_0$ and point $p_0$ such that 
		\[
		\omega_{\max}(t_0)=\omega(t_0,p_0)=\sigma+D.
		\]
		By our choice of $D$, we have $t_0\in(0,T)$. In particular, $0\le \frac{\d}{\d t}\omega(t_0,p_0)$ and since $\ul{\theta}>0$, this implies
		\[
		0\ge \mathcal{H}^2(t_0,p_0)-\fint\mathcal{H}^2(t_0).
		\]
		Then, the a-priori estimates, Proposition \ref{prop_apriori}, yield that
		\begin{align}\label{eq_speed1}
		0\ge \frac{8m}{4\pi\rho^2}\int_{\Sbb^2}\frac{1}{b_{\rho,\vec{a}}^3}\d\mu_{\gamma_\omega}-\frac{8m}{b_{\rho,\vec{a}}(p_0)^3}-\frac{C(B_3,C_1)}{\sigma^4}
		\end{align}
		if $\sigma\ge \sigma_0(B_1,B_2,B_3,C_1,C_2)$.
		
		As $\omega$ takes a maximum at $p_0\in\Sbb^2$, $f:=\frac{\widetilde{\rho}}{\omega}$ achieves a minimum at $p_0\in\Sbb^2$. In particular, $\d f_{p_0}=0$, In view of \eqref{eq_proofaprioir1} this implies that
		\[
		\btr{\spann{\vec{a},e_i}}=\btr{\d\left( \frac{1}{b_{\vec{a}}}\right)_{p_0}(e_i)}\le \frac{C(B_3,C_1)}{\sigma^2}
		\]
		for an ON-frame $\{e_1,e_2\}$ of $T_{p_0}\Sbb^2$. As $\{p_0,e_1,e_2\}$ forms an ON-frame of $\R^3$, in particular
		\[
		\vec{a}=\spann{\vec{a},p_0}p_0+\spann{\vec{a},e_1}e_1+\spann{\vec{a},e_2}e_2,
		\]
		the above estimate implies
		\[
		\btr{\btr{\vec{a}}-\btr{\spann{\vec{a},p_0}}}\le \frac{C(B_3,C_1)}{\sigma^2}.
		\]
		We now claim that in fact
		\[
		\btr{\btr{\vec{a}}-\spann{\vec{a},p_0}}\le \frac{C(B_3,C_1)}{\sigma^2}.
		\]
		For $\spann{\vec{a},p_0}\ge 0$, this follows immediately from the previous estimate. Now assume $\spann{\vec{a},p_0}<0$. As $f$ takes a minimum in $p_0$, \eqref{eq_proofaprioir2} yields
		\[
		{\Hess_{\hatgamma}\left( \frac{1}{b_{\vec{a}}}\right)_{p_0}(e_i,e_i)}+\frac{C(B_3,C_1)}{\sigma^2}\ge 0.
		\]
		As $ \frac{1}{b_{\vec{a}}}=\sqrt{1+\btr{\vec{a}}^2}-\vec{a}\cdot\vec{x}$, i.e., a linear combination of a constant and first spherical harmonics, one can directly check that
		\[
		{\Hess_{\hatgamma}\left( \frac{1}{b_{\vec{a}}}\right)_{\vec{x}}(e_i,e_i)}=\spann{\vec{a},\vec{x}}.
		\]
		Hence, the above inequality implies
		\[
		0\le -\spann{\vec{a},p_0}\le \frac{C(B_3,C_1)}{\sigma^2}.
		\]
		We conclude that
		\[
		\btr{\btr{\vec{a}}-\spann{\vec{a},p_0}}=\btr{\btr{\vec{a}}+\btr{\spann{\vec{a},p_0}}}\le \btr{\btr{\vec{a}}-\btr{\spann{\vec{a},p_0}}}+2\btr{\spann{\vec{a},p_0}}\le \frac{C(B_3,C_1)}{\sigma^2}
		\]
		as claimed. Now
		\begin{align}\label{eq_max1}
		\btr{\frac{1}{b_{\vec{a}}(p_0)}-(\sqrt{1+\btr{\vec{a}}^2}-\btr{\vec{a}})}=\btr{\btr{\vec{a}}-\spann{\vec{a},p_0}}\le \frac{C(B_3,C_1)}{\sigma^2}
		\end{align}
		implies
		\[
		\btr{\frac{1}{b_{\rho,\vec{a}}(p_0)^3}-\frac{\left(\sqrt{1+\btr{\vec{a}}^2}-\btr{\vec{a}}\right)^3}{\rho^3}}\le \frac{C(B_3,C_1)}{\sigma^5}.
		\]
		Furthermore, \eqref{eq_aprioic2alpha} and \eqref{eq_appx_volumeform} yield
		\[
		\btr{\int_{\Sbb^2} \frac{1}{b_{\rho,\vec{a}}^3}\d\mu_{\gamma_\omega}-\int_{\Sbb^2}\frac{1}{b_{\rho,\vec{a}}}\d\mu_{\hatgamma}}\le \frac{C(B_3,C_1)}{\sigma^3}.
		\]
		Hence, \eqref{eq_speed1} implies
		\begin{align*}
		0&\ge \frac{1}{\sigma^3}\left(\fint_{\Sbb^2}\frac{1}{b_{\vec{a}}}\d\mu_{\hatgamma}-\left(\sqrt{1+\btr{\vec{a}}^2}-\btr{\vec{a}}\right)^3-\frac{C(B_3,C_1,m)}{\sigma}\right)\\
		&=\frac{1}{\sigma^3}\left(\sqrt{1+\btr{\vec{a}}^2}-\left(\sqrt{1+\btr{\vec{a}}^2}-\btr{\vec{a}}\right)^3-\frac{C(B_3,C_1,m)}{\sigma}\right).
		\end{align*}
		Note that 
		\begin{align*}
		\sqrt{1+\btr{\vec{a}}^2}-\left(\sqrt{1+\btr{\vec{a}}^2}-\btr{\vec{a}}\right)^3
		&=\frac{\sqrt{1+\btr{\vec{a}}^2}\left(\sqrt{1+\btr{\vec{a}}^2}+\btr{\vec{a}}\right)-\left(\sqrt{1+\btr{\vec{a}}^2}-\btr{\vec{a}}\right)^2}{\sqrt{1+\btr{\vec{a}}^2}+\btr{\vec{a}}}\\
		&=\frac{3\btr{\vec{a}}\sqrt{1+\btr{\vec{a}}^2}-\btr{\vec{a}}^2}{\sqrt{1+\btr{\vec{a}}^2}+\btr{\vec{a}}}\\
		&\ge\frac{2\btr{\vec{a}}\sqrt{1+\btr{\vec{a}}^2}}{2\sqrt{1+\btr{\vec{a}}^2}}= \btr{a},
		\end{align*}
		and thus we obtain
		\begin{align}\label{eq_speed2}
		\btr{\vec{a}}\le \frac{C(B_3,C_1,m)}{\sigma}.
		\end{align}
		As $\omega(p_0)=\sigma+D$, \eqref{eq_proofaprioir2} and \eqref{eq_max1} imply that
		\[
		\sqrt{1+\btr{\vec{a}}^2}-\btr{\vec{a}}-\frac{C(B_3,C_1)}{\sigma^2}\le \frac{\rho}{\omega(p_0)}=\frac{\rho}{\sigma+D}\le \frac{\sigma+K_0}{\sigma+D}.
		\]
		Rearranging gives
		\[
		\frac{\sigma+K_0}{\sigma+D}+\btr{\vec{a}}\ge \sqrt{1+\btr{\vec{a}}^2}-\frac{C(B_3,C_1)}{\sigma^2}
		\]
		Note that for $\sigma\ge\sigma_0(B_1,B_2,B_3,C_1,C_2)$ the right-hand side is positive, and first squaring both sides and then multiplying by $\frac{\sigma+D}{\sigma+K_0}$ yields
		\begin{align*}
		2\btr{\vec{a}}&\ge \frac{\sigma+D}{\sigma+K_0}-\frac{\sigma+K_0}{\sigma+D}-\frac{(\sigma+D)}{(\sigma+K_0)}\frac{C(B_3,C_1)}{\sigma^2}\\
		&=\frac{D^2-K_0^2+2\sigma(D-K_0)}{(\sigma+K_0)(\sigma+D)}-\frac{(\sigma+D)}{(\sigma+K_0)}\frac{C(B_3,C_1)}{\sigma^2}\\
		&\ge \frac{2\sigma(D-K_0)}{(\sigma+K_0)(\sigma+D)}-\frac{(\sigma+D)}{(\sigma+K_0)}\frac{C(B_3,C_1)}{\sigma^2}
		\end{align*}
		As $\Sigma_0,\Sigma_{t_0}\in B_\sigma(B_1,B_2,B_3)$, we have without loss of generality that $K_0,D\le B_1\le \sigma$ for $\sigma\ge \sigma_0(B_1)$. Hence
		\begin{align}\label{eq_max2}
		\btr{\vec{a}}\ge \frac{D-K_0}{4\sigma}-\frac{C(B_3,C_1)}{\sigma^2}
		\end{align}
		Combining \eqref{eq_speed2} and \eqref{eq_max2} implies the desired upper bound for $\sigma\ge \sigma_0(B_1,B_2,B_3,C_1,C_2)$. The lower bound is derived in analogy.
	\end{proof}
	
	To prove the bounds on $\newbtr{\accentset{\circ}{A}}$, $\newbtr{\nabla\accentset{\circ}{A}}$, we use 
	{Raychaudhuri optical equations}, cf.\ \cite[Proposition 4.27]{wolff_thesis}. We first prove the following Lemma:
	\begin{lem}\label{lem_evotracefree}
		Under APNMCF, we have
		\begin{align*}
		\frac{\d}{\d t}\accentset{\circ}{A}_{ij}=&\,\Delta\accentset{\circ}{A}_{ij}+\nabla_i\accentset{\circ}{A}_{jm}\tau^m+\nabla_i\accentset{\circ}{A}_{im}\tau^m-\dive\accentset{\circ}{A}_i\tau_j-\dive\accentset{\circ}{A}_j\tau_i\\
		&\,+\left(\accentset{\circ}{\left(\tau\otimes\d{\mathcal{H}}^2\right)}+\accentset{\circ}{\left(\d {\mathcal{H}}^2\otimes\tau\right)}\right)_{ij}-\left(\scal+\frac{2(1-h)}{\omega^2}+\frac{h'}{\omega}\right)\accentset{\circ}{A}_{ij}\\
		&\,+\dive\tau A_{ij}-\mathcal{H}^2\nabla_i\tau_j+(\nabla\tau\cdot A)_{ij}-(A\cdot\nabla\tau)_{ji}+\left(\mathcal{H}^2-\fint\mathcal{H}^2\right)\accentset{\circ}{\nabla\tau}_{ij}\\
		&\,+\mathcal{H}^2\tau_i\tau_j-\btr{\tau}^2A_{ij}+\btr{\nabla\omega}^2\left(h''+\frac{2(1-h)}{\omega^2}\right)\ul{\theta}\accentset{\circ}{\ul{\chi}}_{ij}+\left(\mathcal{H}^2-\fint\mathcal{H}^2\right)\accentset{\circ}{\left(\tau\otimes\tau\right)}_{ij}\\
		&\,+\frac{1}{2}\left(\mathcal{H}^2-\fint\mathcal{H}^2\right)\left(\frac{1}{\ul{\theta}^2}\left(\newbtr{\accentset{\circ}{\ul{\chi}}}^2+\overline{\Ric}(\ul{L},\ul{L})\right)\accentset{\circ}{A}_{ij}-\frac{1}{\ul{\theta}}\accentset{\circ}{\left(A\cdot\overset{\circ}{\ul{\chi}}\right)}_{ij}+\accentset{\circ}{\left(\overline{\Riem}(\ul{L},\cdot,L,\cdot)\right)}_{ij}\right)\\
		&\,+\left(\ul{\theta}^2\left(2\frac{(1-h)}{\omega^2}+\frac{2h'}{\omega}-h''\right)-\ul{\theta}\left(16\frac{(1-h)}{\omega^3}+12\frac{h'}{\omega^2}\right)\right)\accentset{\circ}{\left(\d\omega\otimes\d\omega\right)}_{ij}\\
		&\,-2\ul{\theta}\left(2\frac{(1-h)}{\omega^2}+\frac{h'}{\omega}\right)\left(\accentset{\circ}{\left(\tau\otimes\d\omega\right)}+\accentset{\circ}{\left(\d \omega\otimes\tau\right)}\right)_{ij}+\frac{1}{2}\mathcal{H}^2\left(\mathcal{H}^2-\fint\mathcal{H}^2\right)\frac{1}{\ul{\theta}}\accentset{\circ}{\ul{\chi}}_{ij}\\
		&\,+\ul{\theta}f^{(1)}_{ij}+\accentset{\circ}{(A*f^{(2)})}_{ij}+\ul{\theta}\accentset{\circ}{(\ul{\chi}*f^{(3)})}_{ij}+\ul{\theta}\accentset{\circ}{(\tau*f^{(4)})}_{ij}
		\end{align*}
	\end{lem}
	\begin{proof}
		As $A=\ul{\theta}\chi$, using the evolution equations for a general flow, \cite[Proposition 4.27]{wolff_thesis}, we find
		\begin{align*}
		\frac{\d}{\d t}A_{ij}&=\Hess_{ij}\mathcal{H}^2+(\tau\otimes\d\mathcal{H}^2+\d\mathcal{H}^2\otimes\tau)_{ij}+\frac{1}{2}\left(\mathcal{H}^2-\fint\mathcal{H}^2\right)\left(\frac{1}{\ul{\theta}^2}\left(\newbtr{\accentset{\circ}{\ul{\chi}}}^2+\overline{\Ric}(\ul{L},\ul{L})\right)A_{ij}\right)
		\\
		&\,-\frac{1}{2}\left(\mathcal{H}^2-\fint\mathcal{H}^2\right)\left(\frac{1}{\ul{\theta}}\left(A\cdot\accentset{\circ}{\ul{\chi}}\right)_{ij}-\overline{\Riem}(\ul{L},\partial_i,L,\partial_i)-2\nabla_i\tau_j-2\tau_i\tau_j\right)\\ 
		\frac{\d}{\d t}\gamma_{ij}&=-\left(\mathcal{H}^2-\fint\mathcal{H}^2\right)\frac{1}{\ul{\theta}}\ul{\chi}_{ij},\\
		\frac{\d}{\d t}\gamma^{ij}&=\left(\mathcal{H}^2-\fint\mathcal{H}^2\right)\frac{1}{\ul{\theta}}\ul{\chi}^{ij},
		\end{align*}
		where we used that
		\begin{align*}
			\left(\newbtr{{\ul{\chi}}}^2+\overline{\Ric}(\ul{L},\ul{L})\right)A_{ij}&=\frac{1}{2}\ul{\theta}^2A_{ij}+\left(\newbtr{\accentset{\circ}{\ul{\chi}}}^2+\overline{\Ric}(\ul{L},\ul{L})\right)A_{ij},\\
			(\chi\cdot\ul{\chi})_{ij}&=\frac{1}{2}A_{ij}+\frac{1}{\ul{\theta}}(A\cdot\accentset{\circ}{\ul{\chi}})_{ij}.
		\end{align*}
		{Using \eqref{eq_stability1} with $f=-\frac{1}{2}\left(\mathcal{H}^2-\fint\mathcal{H}^2\right)$ to compute the evolution of $\mathcal{H}^2$, the claim follows from product rule and the contracted null Simons' identity Lemma \ref{lem_appx_nullsimon}, as $\accentset{\circ}{A}=A-\frac{1}{2}\mathcal{H}^2\gamma$}.
	\end{proof}
	\begin{prop}\label{prop_B2}
		Let $(\Sigma_t)_{[0,T)}$ be a solution of APNMF starting at $\Sigma_0$. Assume that $\Sigma_t\in B_{\sigma}(B_1,B_2,B_3)$ for all $t\in[0,T)$, and there exists $B_2^0$ such that
		\[
		\btr{\accentset{\circ}{A}}_{\Sigma_0}\le \frac{B_2^0}{\sigma^4}.
		\]
		Then 
		\[
		\btr{\accentset{\circ}{A}}\le \frac{C(m,B_2^0)}{\sigma^4}
		\]
		for all $t\in[0,T)$ provided $\sigma\ge \sigma_0(m,B_1,B_3,C_1,C_2)$.
	\end{prop}
	\begin{bem}\label{bem_B2}
		In the proof of Proposition \ref{prop_B2} and Proposition \ref{prop_B3} below, we use that Corollary \ref{kor_aprior} and Lemma \ref{lem_appx_c3control} imply
		\[
			\btr{\nabla\omega}+\sigma\btr{\Hess\omega}\le \frac{C}{\sigma^{\frac{1}{2}}}
		\]
		for some uniform constant $C$ provided $\sigma\ge\sigma_0$ sufficiently large. In fact, for $\sigma\ge\sigma_0$ sufficiently large, we have
		\[
		\btr{\nabla\omega}+\sigma\btr{\Hess\omega}\le \frac{C}{\sigma^{1-\varepsilon}}
		\]
		for any choice of $\varepsilon>0$, and we can make the estimates on $\btr{\accentset{\circ}{A}}$, $\btr{\nabla\accentset{\circ}{A}}$ independent of the mass parameter $m$.
	\end{bem}
	\begin{proof}
		As 
		\[
			\frac{\d}{\d t}\newbtr{\accentset{\circ}{A}}^2=2\spann{\accentset{\circ}{A},\frac{\d}{\d t}\accentset{\circ}{A}}+2\left(\mathcal{H}^2-\fint\mathcal{H}^2\right)\frac{1}{\ul{\theta}}\ul{\chi}^{ik}\gamma_\omega^{jl}\accentset{\circ}{A}_{ij}\accentset{\circ}{A}_{kl},
		\]
		Lemma \ref{lem_evotracefree}, and using the asymptotic behavior stated in Equation \eqref{eq_aprioriH2_constant}, Equation \eqref{eq_aprioriscal}, Corollary \ref{kor_aprior}, Remark \ref{bem_koraprior} (ii), and Lemma \ref{lem_appx_nullsimon} yields
		\begin{align*}
			\frac{\d}{\d t}\newbtr{\accentset{\circ}{A}}^2&\le \Delta \newbtr{\accentset{\circ}{A}}^2-2\newbtr{\nabla \accentset{\circ}{A}}^2-\frac{3}{\sigma^2}\newbtr{\accentset{\circ}{A}}^2+\frac{C}{\sigma^3}\newbtr{\nabla\accentset{\circ}{A}}\cdot \newbtr{\accentset{\circ}{A}}+\frac{C(m)}{\sigma^6}\newbtr{\accentset{\circ}{A}}\\
			&\le \Delta \newbtr{\accentset{\circ}{A}}^2-\frac{3}{2}\newbtr{\nabla \accentset{\circ}{A}}^2-\frac{2}{\sigma^2}\newbtr{\accentset{\circ}{A}}^2+\frac{C(m)}{\sigma^6}\newbtr{\accentset{\circ}{A}}\\
			&\le \Delta \newbtr{\accentset{\circ}{A}}^2-\frac{1}{\sigma^2}\newbtr{\accentset{\circ}{A}}^2+\frac{C(m)}{\sigma^{10}},
		\end{align*}
		using Young's inequality provided $\sigma\ge\sigma_0(m,B_1,B_2,B_3,C_1,C_2)$. Hence, whenever $\newbtr{\accentset{\circ}{A}}^2$ reaches a new maximum for the first time, we have
		\[
		\newbtr{\accentset{\circ}{A}}^2\le \frac{C(m)}{\sigma^8},
		\]
		so
		\[
		\newbtr{\accentset{\circ}{A}}^2\le \max\left(\frac{C(m)}{\sigma^8},\max\limits_{\Sigma_0}\newbtr{\accentset{\circ}{A}}^2\right)\le \frac{C(m,B_2^0)}{\sigma^8}.
		\]
	\end{proof}
	
	Before we prove the estimate on $\newbtr{\nabla\accentset{\circ}{A}}$, let us first recall some well-known identities:\newline
	For a smooth path of metrics we have
	\[
	\frac{\d}{\d t}\Gamma_{ij}^k=\frac{1}{2}\gamma^{kl}\left(\nabla_i\frac{\d}{\d t}\gamma_{jl}+\nabla_j\frac{\d}{\d t}\gamma_{il}-\nabla_l\frac{\d}{\d t}\gamma_{ij}\right).
	\]
	Moreover, for a symmetric $(0,2)$-Tensor $T$ we have
	\[
	\frac{\d}{\d t}\nabla_kT_{ij}=\nabla_k\frac{\d}{\d t}T_{ij}-T_{in}\frac{\d}{\d t}\Gamma_{jk}^n-T_{jn}\frac{\d}{\d t}\Gamma_{ik}^n,
	\]
	and
	\begin{align*}
	\nabla^3_{ijk}T_{mn}-\nabla^3_{jki}T_{mn}
	=&\,T_n^l\nabla_j\Riem_{kilm}+T_m^l\nabla_j\Riem_{kiln}+\nabla^lT_{mn}\Riem_{jilk}\\
	&\,+\nabla_kT_n^l\Riem_{jilm}+\nabla_kT_m^l\Riem_{jiln}+\nabla_jT_n^l\Riem_{kilm}+\nabla_jT_m^l\Riem_{kiln}.
	\end{align*}
	Using that in $2$ dimensions we have $\Riem_{ijkl}=\frac{1}{2}\scal\left(\gamma_{ik}\gamma_{jl}-\gamma_{jk}\gamma_{ik}\right)$, we obtain
	\begin{align*}
	2\spann{\nabla_i\Delta T_{mn},\nabla_iT_{mn}}
	=&\,\Delta\btr{\nabla T}^2-\btr{\nabla^2 T}^2-\frac{1}{2}\scal\btr{\nabla T}^2+2T(\nabla \scal,\dive T)\\
	&\,-2\spann{T\otimes\d \scal,\nabla T}+2\scal\btr{\dive T}^2-2\scal\nabla_mT_{in}\nabla^iT^{mn}.
	\end{align*}
	\begin{lem}\label{lem_evo_nablaA}
		Under APNMCF, we have
		\begin{align*}
		\frac{\d}{\d t}\newbtr{\nabla\accentset{\circ}{A}}^2
		=&\,\Delta \newbtr{\nabla\accentset{\circ}{A}}^2-2\newbtr{\nabla^2\accentset{\circ}{A}}^2+\spann{\accentset{\circ}{A}\otimes\d(\mathcal{H}^2-2\scal),\nabla\accentset{\circ}{A}}+\spann{\d(\mathcal{H}^2-2\scal)\otimes\accentset{\circ}{A},\nabla\accentset{\circ}{A}}\\
		&\,-\left(\frac{5}{2}\scal-\left(\mathcal{H}^2-\fint\mathcal{H}^2\right)\left(\frac{3}{2}+\frac{1}{\ul{\theta}^2}\left(\newbtr{\accentset{\circ}{\ul{\chi}}}^2+\overline{\Ric}(\ul{L},\ul{L})\right)\right)+\frac{4(1-h)}{r^2}+\frac{2h'}{r}\right)\newbtr{\nabla\accentset{\circ}{A}}^2\\
		&\,+2\scal\btr{\dive\accentset{\circ}{A}}^2-2\scal\nabla_i\accentset{\circ}{A}_{kj}\nabla^k\accentset{\circ}{A}^{ij}-\accentset{\circ}{A}\left(\nabla(\mathcal{H}^2-2\scal),\dive\accentset{\circ}{A}\right)\\
		&\,+\frac{1}{\ul{\theta}^2}\left(\newbtr{\accentset{\circ}{\ul{\chi}}}^2+\overline{\Ric}(\ul{L},\ul{L})\right)\spann{\d\mathcal{H}^2\otimes\accentset{\circ}{A},\nabla\accentset{\circ}{A}}-\frac{1}{\ul{\theta}}\spann{\d\mathcal{H}^2\otimes\accentset{\circ}{\left(\overset{\circ}{\ul{\chi}}\cdot A\right)},\nabla\accentset{\circ}{A}}\\
		&\,+\spann{\d\mathcal{H}^2\otimes\accentset{\circ}{\left(\overline{\Riem}(\ul{L},\cdot,\ul{L},\cdot)\right)},\nabla\accentset{\circ}{A}}+2\spann{\d\mathcal{H}^2\otimes\left(\accentset{\circ}{\nabla\tau}+\accentset{\circ}{\left(\tau\otimes\tau\right)}\right),\nabla\accentset{\circ}{A}}\\
		&\,+\frac{1}{\ul{\theta}}\spann{\left(\mathcal{H}^2-\fint\mathcal{H}^2\right)\d\mathcal{H}^2\otimes\accentset{\circ}{\ul{\chi}}+\mathcal{H}^2\d\mathcal{H}^2\otimes\accentset{\circ}{\ul{\chi}},\nabla\accentset{\circ}{A}}\\
		&\,+\left(\mathcal{H}^2-\fint\mathcal{H}^2\right)\spann{\d\left(\frac{1}{\ul{\theta}^2}\left(\newbtr{\accentset{\circ}{\ul{\chi}}}^2+\overline{\Ric}(\ul{L},\ul{L})\right)\right)\otimes\accentset{\circ}{A}+\nabla\accentset{\circ}{\left(\overline{\Riem}(\ul{L},\cdot,\ul{L},\cdot)\right)},\nabla\accentset{\circ}{A}}\\
		&\,+\left(\mathcal{H}^2-\fint\mathcal{H}^2\right)\spann{\nabla\left(\accentset{\circ}{\nabla\tau}\right)+\nabla\accentset{\circ}{\left(\tau\otimes\tau\right)},\nabla\accentset{\circ}{A}}+2\mathcal{H}^2\left(\mathcal{H}^2-\fint\mathcal{H}^2\right)\spann{\nabla\left(\frac{1}{\ul{\theta}}\accentset{\circ}{\ul{\chi}}\right),\nabla\accentset{\circ}{A}}\\
		&\,+\spann{\nabla\left(\left(\ul{\theta}^2\left(2\frac{(1-h)}{\omega^2}+\frac{2h'}{\omega}-h''\right)-\ul{\theta}\left(16\frac{(1-h)}{\omega^3}+12\frac{h'}{\omega^2}\right)\right)\accentset{\circ}{\left(\d\omega\otimes\d\omega\right)}\right),\nabla\accentset{\circ}{A}}\\
		&\,+\spann{\nabla\left(\btr{\nabla\omega}^2\left(h''+\frac{2(1-h)}{\omega^2}\right)\ul{\theta}\accentset{\circ}{\ul{\chi}}\right),\nabla \accentset{\circ}{A}}-\partial_r\left(4\frac{(1-h)}{\omega^2}+\frac{2h'}{\omega}\right)\spann{\d\omega\otimes\accentset{\circ}{A},\nabla \accentset{\circ}{A}}\\
		&\,-2\spann{\nabla\left(\ul{\theta}\left(2\frac{(1-h)}{\omega^2}+\frac{h'}{\omega}\right)\left(\accentset{\circ}{\left(\tau\otimes\d\omega\right)}+\accentset{\circ}{\left(\d \omega\otimes\tau\right)}\right)\right),\nabla \accentset{\circ}{A}}\\
		&\,+\nabla^2\accentset{\circ}{A}*\tau*\nabla \accentset{\circ}{A}+\nabla \dive \accentset{\circ}{A}*\tau*\nabla \accentset{\circ}{A}+\frac{\left(\mathcal{H}^2-\fint\mathcal{H}^2\right)}{\ul{\theta}}\accentset{\circ}{\ul{\chi}}*\nabla \accentset{\circ}{A}*\nabla \accentset{\circ}{A}\\
		&\,+\nabla\tau*\nabla A*\nabla\accentset{\circ}{A}+\nabla A*\tau*\tau*\nabla \accentset{\circ}{A}+A*\nabla\tau*\tau*\nabla\accentset{\circ}{A}+\nabla^2\tau*A*\nabla\accentset{\circ}{A}\\
		&\,+\ul{\theta}\Riem_{ijkL}*\tau*\nabla \accentset{\circ}{A}+\left(\mathcal{H}^2-\fint\mathcal{H}^2\right)\nabla\left(\frac{1}{\ul{\theta}}\accentset{\circ}{\ul{\chi}}\right)*\nabla\accentset{\circ}{A}+2\spann{\nabla(\ul{\theta}f^{(1)},\nabla \accentset{\circ}{A}}\\
		&\,+\nabla A*f^{(2)}*\nabla \accentset{\circ}{A}+A*\nabla f^{(2)}*\nabla\accentset{\circ}{A}+\d\ul{\theta}*\ul{\chi}*f^{(3)}*\nabla\accentset{\circ}{A}+\ul{\theta}\cdot\nabla\ul{\chi}*f^{(3)}*\nabla\accentset{\circ}{A}\\
		&\,+\ul{\theta}\cdot\ul{\chi}*\nabla f^{(3)}*\nabla \accentset{\circ}{A}+\d\ul{\theta}*\tau*f^{(4)}*\nabla \accentset{\circ}{A}+\ul{\theta}\cdot\nabla\tau*f^{(4)}*\nabla \accentset{\circ}{A}+\ul{\theta}\cdot\tau*\nabla f^{(4)}*\nabla \accentset{\circ}{A}
		\end{align*}
	\end{lem}
	\begin{proof}
		First, it is immediate from the evolution of $\gamma^{ij}$ that
		\[
		\frac{\d}{\d t}\newbtr{\nabla\accentset{\circ}{A}}^2=2\spann{\frac{\d}{\d t}\nabla\accentset{\circ}{A},\nabla\accentset{\circ}{A}}+\frac{3}{2}\left(\mathcal{H}^2-\fint\mathcal{H}^2\right)\newbtr{\nabla\accentset{\circ}{A}}^2+\frac{\left(\mathcal{H}^2-\fint\mathcal{H}^2\right)}{\ul{\theta}}\accentset{\circ}{\ul{\chi}}*\nabla \accentset{\circ}{A}*\nabla \accentset{\circ}{A}.
		\]
		Moreover, under APNMCF we obtain
		\[
		\frac{\d}{\d t}\Gamma_{ij}^k=\frac{1}{4}\left(\nabla^k\mathcal{H}^2\gamma_{ij}-\d\mathcal{H}^2_i\delta_j^k-\d\mathcal{H}^2_j\delta_i^k\right),
		\]
		which gives
		\begin{align*}
		\frac{\d}{\d t}\nabla_k\accentset{\circ}{A}_{ij}
		=&\,\nabla_k\frac{\d}{\d t}\accentset{\circ}{A}_{ij}+\frac{1}{2}\d\mathcal{H}^2_k\accentset{\circ}{A}_{ij}+\frac{1}{4}\left(\accentset{\circ}{A}_{ik}\d\mathcal{H}^2_j+\accentset{\circ}{A}_{jk}\d\mathcal{H}^2_i\right)\\
		&\,-\frac{1}{4}\left(\accentset{\circ}{A}(\nabla\mathcal{H}^2,\partial_i)\gamma_{jk}+\accentset{\circ}{A}(\nabla\mathcal{H}^2,\partial_j)\gamma_{ik}\right).
		\end{align*}
		The claim then follows by Lemma \ref{lem_evotracefree} after taking a tensor derivative, and using that
		\begin{align*}
		2\spann{\nabla\Delta \accentset{\circ}{A},\nabla \accentset{\circ}{A}}
		=&\,\Delta\btr{\nabla \accentset{\circ}{A}}^2-\btr{\nabla^2 \accentset{\circ}{A}}^2-\frac{1}{2}\scal\btr{\nabla \accentset{\circ}{A}}^2+2\accentset{\circ}{A}(\nabla \scal,\dive \accentset{\circ}{A})\\
		&\,-2\spann{\accentset{\circ}{A}\otimes\d \scal,\nabla \accentset{\circ}{A}}+2\scal\btr{\dive \accentset{\circ}{A}}^2-2\scal\nabla_m\accentset{\circ}{A}_{in}\nabla^i\accentset{\circ}{A}^{mn}.
		\end{align*}
	\end{proof}
	\begin{prop}\label{prop_B3}
		Let $(\Sigma_t)_{t\in][0,T)}$ be a solution of APNMCF starting at $\Sigma_0$. Assume that $\Sigma_0$ satisfies 
		\begin{align*}
		\btr{\accentset{\circ}{A}}&\le \frac{B_2^0}{\sigma^4},\\
		\btr{\nabla\accentset{\circ}{A}}&\le \frac{B_3^0}{\sigma^5}
		\end{align*}
		for constants $B_2^0$, $B_3^0$, and assume that $\Sigma_t\in B_\sigma(B_1,B_2,B_3)$ for all $t\in[0,T)$. Then
		\[
		\btr{\nabla\accentset{\circ}{A}}^2\le \frac{C(B_2^0,B_3^0,m)}{\sigma^{10}}
		\]
		provided $\sigma\ge \sigma_0(m,B_1,B_2,B_3,C_1,C_2)$.
	\end{prop}
	\begin{proof}
		As in the proof of Proposition \ref{prop_B2} the decay estimates, i.e., Lemma \ref{lem_prelim_chitau}, \ref{lem_prelim_chitau}, \ref{lem_appx_nullsimon}, Remark \ref{bem_apriori}, \ref{bem_koraprior} (ii), Lemma \ref{lem_appx_curvatureestimates1}, Equations \eqref{eq_aprioriscal}, \eqref{eq_appx_riemdecompLbarL},
		and Proposition \ref{prop_B2} itself, yield
		\[
			\frac{\d}{\d t}\btr{\nabla \accentset{\circ}{A}}^2\le \Delta \btr{\nabla \accentset{\circ}{A}}^2+\frac{K_1}{\sigma^2}\btr{\nabla \accentset{\circ}{A}}^2+\frac{K_2(m)}{\sigma^{12}}.
		\]
		for some (fixed) non-negative constants $K_1, K_2(m)$ for $\sigma\ge \sigma_0(m,B_1,B_2,B_3,C_1,C_2)$ sufficiently large. Redefine $K_1:=K_1+1$, and recall that for $\sigma\ge \sigma_0(m,B_1,B_2,B_3,C_1,C_2)$ sufficiently large, we have seen in the proof of Proposition \ref{prop_B2} that
		\[
			\frac{\d}{\d t}\btr{\accentset{\circ}{A}}^2\le \Delta \btr{ \accentset{\circ}{A}}^2-\frac{3}{2}\btr{\nabla \accentset{\circ}{A}}^2+\frac{C(m)}{\sigma^{10}}.
		\]
		Hence $f:=\btr{\nabla \accentset{\circ}{A}}^2+\frac{K_1}{\sigma^2}\btr{ \accentset{\circ}{A}}^2$ satisfies
		\[
			\frac{\d}{\d t}f\le \Delta f-\frac{K_1}{2\sigma^2}\btr{\nabla \accentset{\circ}{A}}^2+\frac{C(m)}{\sigma^{12}}\le \Delta f-\frac{K_1}{2\sigma^2}f+\frac{C(m,B_2^0)}{\sigma^{12}},
		\]
		using Proposition \ref{prop_B2}. In particular, whenever $f$ takes a new maximum in time, we have
		\[
		f\le \frac{C(m,B_2^0)}{\sigma^{10}},
		\]
		and hence
		\[
		\newbtr{\nabla\accentset{\circ}{A}}^2\le f\le  \max\left(\max\limits_{\Sigma_0}f,\frac{C(m,B_2^0)}{\sigma^{10}}\right)\le \frac{C(m,B_2^0,B_3^0)}{\sigma^{10}}.
		\]
	\end{proof}
	We are now ready to prove Theorem \ref{thm_longtime}:
	\begin{proof}[Proof of Theorem \ref{thm_longtime}]
		Let $B_1^0$, $B_2^0$, $B_3^0$ be stictly positive constants, $c_1,c_2\ge 10$. We first define define $C_1:=2c_1+1$, $C_2:=2c_2+1$. Let $\mathcal{B}_2=C(m,B_2^0)$ denote the constant in the statement of Proposition \ref{prop_B2}, and let $\mathcal{B}_3=C(m,B_2^0,B_3^0)$ denote the constant in the statement of Proposition \ref{prop_B3}. Note that $\mathcal{B}_2\ge B_2^0$, $\mathcal{B}_3\ge B_3^0$ by definition. Define
		\begin{align*}
		B_2&:=2\mathcal{B}_2+1,\\
		B_3&:=2\mathcal{B}_3+1.
		\end{align*}
		Let $\mathcal{B}_1=C(B_1^0,B_3,C_1,m)$ be the constant in the statement of Proposition \ref{prop_B1}. Again, $\mathcal{B}_1\ge B_1^0$ by definition. Define $B_1:=2\mathcal{B}_1+1$. Finally define $\sigma_0=\sigma_0(m,B_1,B_2,B_3,C_1,C_2)$ according to Propositions \ref{prop_B1}, \ref{prop_B2}, \ref{prop_B3}, and such that
		\[
			\norm{\frac{\omega_0}{\rho_0}}_{C^3(\widehat{\gamma})}\le c_1,\text{ and }\btr{\mathcal{H}^2_0}\le \frac{c_2}{\rho_0^2},
		\]
		cf. Proposition \ref{prop_apriori}, Corollary \ref{kor_aprior}.
		
		Now by definition of $B_1$, $B_2$, $B_3$, any solution of APNMCF starting at $\Sigma_0$ will remain in $B_\sigma(B_1,B_2,B_3)$ at least for some short time interval $[0,\varepsilon)$, 
		and satisfies
		\[
		\norm{\frac{\omega}{\rho}}_{C^3(\hatgamma)}\le C_1,\text{ and }\btr{\mathcal{H}^2}\le \frac{C_2}{\rho^2}
		\]
		in $[0,\varepsilon)$. In view of the estimates in Proposition \ref{prop_B1}, \ref{prop_B2}, \ref{prop_B3} and the definition of $B_1$, $B_2$, $B_3$, the solution will in fact strictly remain in $B_\sigma(B_1,B_2,B_3)$. Moreover, the a-priori estimates, Proposition \ref{prop_apriori} and Corollary \ref{kor_aprior}, imply that in fact
		\[
		\norm{\frac{\omega}{\rho}}_{C^3(\hatgamma)}\le c_1,\text{ and }\btr{\mathcal{H}^2}\le \frac{c_2}{\rho^2}
		\]
		provided $\sigma\ge\sigma_0(m,B_1,B_2,B_3,C_1,C_2)$.
		
		Furthermore, we can extend the flow beyond any finite time interval $[0,T)$ by standard arguments as long as it remains in the a-priori class: As $\Sigma_t\in B_\sigma(B_1,B_2,B_3)$ and using Remark \ref{bem_apriori} we have uniform $C^1$ bounds on $\omega$ (independent of $t$). As the solution of APNMCF is equivalent to a quasilinear equation on $\omega$, we can use the regularity theory for parabolic equations to obtain uniform $C^k$-estimates (independent of $t$) similar as in the case of null mean curvature flow, cf.  \cite[Proposition 3.8]{roeschscheuer}. Alternatively, one can use Shi-type estimates to obtain uniform bounds on $\btr{\nabla^k\mathcal{H}^2}$ similar as in \cite[Lemma 21]{wolff1} (one however only gets a uniform upper bound and a-priori no exponential decay), and then argue as usual.
		
		Hence, we can continue the flow and the solution again remains in $B_\sigma(B_1,B_2,B_3)$ 
		and satisfies 
		\[
		\norm{\frac{\omega}{\rho_0}}_{C^3(\hatgamma)}\le C_1,\text{ and }\btr{\mathcal{H}^2}\le \frac{C_2}{\rho^2},
		\]
		for some short time. Thus, the solution can be extended indefinitely and exists for all times. In particular, the solution remains (strictly) in $B_\sigma(B_1,B_2,B_3)$ and satisfies
		\[
		\norm{\frac{\omega}{\rho_0}}_{C^3(\hatgamma)}\le C_1,\text{ and }\btr{\mathcal{H}^2}\le \frac{C_2}{\rho^2},
		\]
		for all times. This concludes the proof.	
	\end{proof}
	\begin{thm}\label{thm_convergence}
		Under the assumptions of Theorem \ref{thm_longtime} and for any initial data $\Sigma_0$ as in Theorem \ref{thm_longtime}, we have smooth, exponential convergence to an STCMC surface provided $\sigma\ge \sigma_0(m,B_1,B_2,B_3,C_1,C_2)$.
	\end{thm}
	\begin{proof}
		Choose $B_1$, $B_2$, $B_3, C_1, C_2$ as in Theorem \ref{thm_longtime} (with respect to constants $B_1^0$, $B_2^0$, $B_3^0$, and $c_1,c_2\ge 10$). Let $\sigma_0^1(m,B_1,B_2,B_3,C_1,C_2)$ as in Theorem \ref{thm_longtime}, $\sigma_0^2(m,B_1,B_3,C_1,C_2)$ as in Proposition \ref{prop_strictstability}. Hence, if $\Sigma_0\in B_\sigma(B_1^0,B_2^0,B_3^0)$ 
		and satisfies 
		\[
		\norm{\frac{\omega}{\rho_0}}_{C^3(\hatgamma)}\le c_1
		\]
		for ${\sigma\ge\sigma_0:=\max(\sigma_0^1(m,B_1,B_2,B_3,C_1,C_2),\sigma_0^2(m,B_1,B_3,C_1,C_2))}$ the solution of APNMCF exists for all times and remains in $B_\sigma(B_1,B_2,B_3)$. Moreover, each surface $\Sigma_t$ is strictly stable (under area preserving variations) by Proposition \ref{prop_strictstability}. 
		
		Recall the evolution of area and $\mathcal{H}^2$ written down in Subsection \ref{subsec_stability} under a general deformation. We note that $\frac{\d}{\d t}\fint\mathcal{H}^2$ is constant along $\Sigma_t$ and thus, we find that
		\begin{align*}
		\frac{\d}{\d t}\int \left(\mathcal{H}^2-\fint\mathcal{H}^2\right)^2\d\mu
		&=2\int \left(\mathcal{H}^2-\fint\mathcal{H}^2\right)\frac{\d}{\d t}\mathcal{H}^2-\frac{1}{2}\left(\mathcal{H}^2-\fint\mathcal{H}^2\right)^3\d\mu\\
		&=-\int\left(\mathcal{H}^2-\fint\mathcal{H}^2\right)J\left(\mathcal{H}^2-\fint\mathcal{H}^2\right)-\frac{1}{2}\left(\mathcal{H}^2-\fint\mathcal{H}^2\right)^3\d\mu\\
		&\,\le -\left(\frac{6m}{\sigma^3}-\frac{C}{\sigma^4}\right)\int \left(\mathcal{H}^2-\fint\mathcal{H}^2\right)^2\d\mu\\
		&\,\le-\frac{4m}{\sigma^3}\int \left(\mathcal{H}^2-\fint\mathcal{H}^2\right)^2\d\mu
		\end{align*}
		provided $\sigma\ge \sigma_0(m,B_1,B_2,B_3,C_1,C_2)$, 
		{using Proposition \ref{prop_strictstability} and Remark \ref{bem_apriori} in the third and fourth line, respectively}.
		Hence,
		\begin{align}\label{eq_thm_convergence2}
		\norm{\mathcal{H}^2-\fint\mathcal{H}^2}_{L^2(\Sigma_t)}^2\le \exp\left(-\frac{4m}{\sigma^3}t\right)\cdot\norm{\mathcal{H}^2-\fint\mathcal{H}^2}_{L^2(\Sigma_0)}^2.
		\end{align}
		As all derivatives $\nabla^k\mathcal{H}^2$ are uniformly bounded (as $\omega_t$ is uniformly bounded in $C^{k+2}$ for any $k$), standard interpolation inequalities yield exponential convergence in higher Sobolov norms. Then, the standard embedding theorems yield exponential convergence of $\left(\mathcal{H}^2-\fint\mathcal{H}^2\right)$ in any $C^k$-norm. As
		\[
		\frac{\d }{\d t}\gamma=-\left(\mathcal{H}^2-\fint\mathcal{H}^2\right)\frac{1}{\ul{\theta}}\ul{\chi},
		\]
		this implies the smooth, exponential convergence, cf. \cite[Appendix A]{brendleSphere}. Here, we use that $\frac{1}{\ul{\theta}}\ul{\chi}$ is a smooth, bounded tensor with bounded derivatives, where all bounds depend only on the bound of $\omega$ and its derivatives.
	\end{proof}
\section{Asymptotic foliations of STCMC surfaces}\label{sec_foliation}
	Let $(S_\sigma)_{\sigma\ge \sigma_0}$ be a smooth family of spacelike cross sections of $\mathcal{N}$ in $B_\sigma(B_1,B_2,B_3)$ and area radius $\rho_\sigma=\sigma$, where everything is as in Theorem \ref{thm_longtime}. In particular, for every $\sigma\ge \sigma_0$ there exists an STCMC surface $\Sigma_\sigma\in B_\sigma(B_1,B_2,B_3)$ with area radius $\sigma$ such that the solution of APNMCF starting at $S_\sigma$ smoothly converges to $\Sigma_\sigma$ as $t\to\infty$.
	
	{Note that we can guarantee the existence of such a smooth family $(S_\sigma)$ as coordinate spheres with $\omega=r(\sigma)$ are admissible initial data in the sense of Theorem \ref{thm_longtime} for $\sigma\ge \sigma_0$ sufficiently large. In fact, this remains true for smooth families of boosted spheres $\omega_{r(\sigma),\vec{a}}$ (and small perturbations) for $\vec{a}$ sufficiently small with respect to $\sigma$, cf. Remark \ref{bem_aprioriclass} and Proposition \ref{prop_apriori}.}
	
	\begin{thm}\label{thm_foliation}
		For $\sigma\ge \sigma_0$, the family $(\Sigma_\sigma)$ forms a smooth, asymptotic foliation of $\mathcal{N}$.
	\end{thm}
	\begin{proof}
		{Let $\Phi_{\ul{L}}\colon(r_0,\infty)\times\Sbb^2\to\mathcal{N}$ denote the smooth diffeomorphism induced by the background foliation of coordinate spheres as in Subsection \ref{subsec_generalsetup}}. From now on, we will identify $\mathcal{N}$ with $(0,\infty)\times\Sbb^2$ via $\Phi_{\ul{L}}$ for convenience. Moreover, let $\omega_\sigma\colon \Sbb^2\to(r_0,\infty)$ denote the smooth function on $\Sbb^2$ such that $\Sigma_\sigma=\{r=\omega_\sigma\}$. Then, the claim is equivalent to showing that
		\[
		\Phi\colon (\sigma_0,\infty)\times\Sbb^2\to\mathcal{N},(\sigma,\vec{x})\mapsto (\omega_{\sigma}(\vec{x}),\vec{x})
		\]
		is a smooth diffeomorphism onto $\widetilde{\mathcal{N}}:=\{(s,\vec{x})\colon s>\omega_{\sigma_0}(\vec{x})\}$.
		
		Note that the regularity of $\Phi$ follows from the regularity of $(S_\sigma)$ and standard parabolic theory.
		
		Let us first show that $\Phi$ is surjective onto $\widetilde{\mathcal{N}}$: Let $p=(s,\vec{x})\in\widetilde{\mathcal{N}}$. Since $\Sigma_\sigma\in B_\sigma(B_1,B_2,B_3)$, we have $\omega_\sigma\ge \sigma-B_1$. In particular, there exists $\sigma_1>s$ such that 
		\[
		\omega_{\sigma_1}(\vec{x})>s>\omega_{\sigma_0}(\vec{x}).
		\]
		{Recall that $\pi^1$ denotes the projection of $(\Phi_{\ul{L}})^{-1}$ onto the $s$-component}. Then 
		\[
		f_{\vec{x}}\colon [\sigma_0,\sigma_1]\to\R,\sigma\mapsto \pi^1\circ\Phi(\sigma,\vec{x})=\omega_\sigma(\vec{x})
		\]
		is a continuous map with $f_{\vec{x}}(\sigma_0)<s<f_{\vec{x}}(\sigma_1)$. By the intermediate value theorem there exists $\sigma\in(\sigma_0,\sigma_1)$ such that $\omega_{\sigma}(\vec{x})=f_{\vec{x}}(\sigma)=s$. In particular, $\Phi(\sigma,\vec{x})=(s,\vec{x})=p$.
		
		Note that
		\[
		D\Phi=\begin{pmatrix}
		\partial_\sigma\omega_\sigma&\partial_1\omega_\sigma&\partial_2\omega_\sigma\\
		0&1&0\\
		0&0&1
		\end{pmatrix},
		\]
		where $x^1$, $x^2$ denotes (local) coordinates on $\Sbb^2$. Moreover, 
		\[
		\Phi(\sigma_1,\vec{x}_1)=\Phi(\sigma_2,\vec{x}_2)\Leftrightarrow \vec{x}_1=\vec{x}_2,\text{ and }\omega_{\sigma_1}(\vec{x}_1)=\omega_{\sigma_2}(\vec{x}_2).
		\]
		Hence, it remains to show that $\partial_\sigma\omega\not=0$.
		
		To this end, we first show that 
		\begin{align}\label{eq_foliation_f2}
		f_2\colon (\sigma_0,\infty)\to (0,\infty), \sigma\mapsto \mathcal{H}^2_\sigma
		\end{align}
		is a smooth, strictly decreasing function with $\operatorname{Im}(f_2)=(0,\mathcal{H}^2_{\sigma_0})$, where $\mathcal{H}^2_\sigma$ denotes the (constant) spacetime mean curvature of $\Sigma_\sigma$. Note that the regularity of $f_2$ immediately follows from the regularity of $\Phi$.
		
		Assume that there exists $\sigma$ such that $\partial_\sigma f_2(\sigma)=\partial_\sigma\mathcal{H}^2_\sigma=0$. As $\partial_\sigma\mathcal{H}^2_\sigma=J(\ul{\theta}\partial_\sigma\omega_\sigma)$, Proposition \ref{prop_stability2} implies that $\ul{\theta}\partial_\sigma\omega_\sigma\equiv 0$. On the other hand, as $\btr{\Sigma_\sigma}=4\pi\sigma^2$, we have
		\[
		8\pi\sigma=\partial_\sigma\btr{\Sigma_\sigma}=\int\ul{\theta}\partial_\sigma\omega_\sigma\d\mu,
		\] 
		a contradiction. Hence, $\partial_\sigma f_2\not=0$, and since $\mathcal{H}^2_\sigma\to0$ as $\sigma\to\infty$ by Proposition \ref{prop_apriori}, $f_2$ is necessarily strictly decreasing. In particular, $\operatorname{Im}(f_2)=(0,\mathcal{H}^2_{\sigma_0})$.
		
		Note that for $r_0>C(m)$, $r\mapsto \frac{4}{r^2}-\frac{8m}{r^3}$ is strictly decreasing on $[r_0,\infty)$, and thus for $\sigma_0(m,B_1,B_2,B_3,C_1,C_2)$ sufficiently large there exists a unique $\sigma_0'$ such that ${\mathcal{H}^2_{\sigma_0}=\frac{4}{(\sigma_0')^2}-\frac{8m}{(\sigma_0')^3}}$ and 
		\begin{align}\label{eq_foliation_f3}
		f_3:(\sigma_0',\infty)\to (0,\mathcal{H}^2_{\sigma_0}), \sigma'\mapsto \frac{4}{(\sigma')^2}-\frac{8m}{(\sigma')^3}
		\end{align}
		is a smooth, strictly decreasing bijection. In particular,
		\[
		\sigma(\sigma'):=f_2^{-1}(f_3(\sigma'))
		\]
		is a smooth, strictly increasing bijection. Recall that by Remark \ref{bem_apriori} (with $\rho_\sigma=\sigma$)
		\[
		\btr{\mathcal{H}^2-\frac{4}{\sigma^2}+\frac{8m}{\sigma^3}}\le \frac{C}{\sigma^4},
		\]
		which implies 
		\[
		\frac{1}{2}\sigma\le \sigma'\le 2\sigma
		\]
		for $\sigma\ge\sigma_0(m,B_1,B_2,B_3,C_1,C_2)$ sufficiently large.
		In fact, using \eqref{eq_aprioriH2_constant}, and the fact that $f_3$ is bijective, one can show that $\btr{\sigma-\sigma'}\le \frac{C}{\sigma}$ for $\sigma\ge\sigma_0(m,B_1,B_2,B_3,C_1,C_2)$ by the mean value theorem.
		
		We now define the smooth map
		\[
		\Phi'\colon (\sigma_0',\infty)\times\Sbb^2\to \widetilde{\mathcal{N}}, (\sigma',\vec{x})\mapsto(\omega_{\sigma(\sigma')}(\vec{x}),\vec{x}),
		\]
		and let $u:=\partial_{\sigma'}\omega_{\sigma(\sigma')}$. By construction, we have
		\[
		J(\ul{\theta}u)=\partial_{\sigma'}\mathcal{H}^2=-\frac{8}{(\sigma')^3}+\frac{24m}{(\sigma')^4}.
		\]
		In particular, using Lemma \ref{lem_prelim_chitau}, \ref{lem_appx_curvatureestimates1} and the fact that $\mathcal{N}$ is aymptotically Schwarzschild, i.e., $h=1-\frac{2m}{r}+O_4(r^{-2})$, we find
		\begin{align*}
		\btr{J\left(\ul{\theta}u-\frac{2}{\sigma'}\right)}
		\le&\,\btr{\partial_{\sigma'}\mathcal{H}^2+\mathcal{H}^2\frac{2}{\sigma'}+\frac{1}{\sigma'}\left(2\overline{\Ric}(\ul{L},L)-\Riem(\ul{L},L,L,\ul{L})\right)}\\
		&\,+\btr{\frac{2\mathcal{H}^2}{\sigma'\ul{\theta}^2}\left(\newbtr{\accentset{\circ}{\ul{\chi}}}^2+\overline{\Ric}(\ul{L},\ul{L}\right)+\frac{2}{\sigma'}\left(\frac{1}{\ul{\theta}}\spann{\accentset{\circ}{\ul{\chi}},\accentset{\circ}{A}}+2\dive\tau+2\btr{\tau}^2\right)}\\
		\le &\,\btr{\frac{8m}{\sigma'}\left(\frac{1}{(\sigma')^3}-\frac{1}{\omega_{\sigma(\sigma')}^3}\right)}+\frac{C}{\sigma^5}\le \frac{C}{\sigma^5}.
		\end{align*}
		Hence
		\[
		\norm{J\left(\ul{\theta}u-\frac{2}{\sigma'}\right)}_{L^2(\Sigma)}\le \frac{C}{\sigma^4},
		\]
		and Proposition \ref{prop_stability2} and Corollary \ref{kor_sobolevineq} imply that
		\[
		\btr{\ul{\theta}u-\frac{2}{\sigma'}}\le \frac{C}{\sigma^2},
		\]
		so 
		\[
		\btr{u-1}\le \frac{C}{\sigma}
		\]
		for $\sigma\ge \sigma_0$. Hence, for $\sigma\ge \sigma_0$, we have $u>0$, and as
		\[
		u=\partial_{\sigma'}(\sigma(\sigma'))\partial_\sigma\omega_{\sigma}
		\]
		we also find $\partial_\sigma\omega_{\sigma}>0$. This shows that both $\Phi$ and $\Phi'$ are smooth diffeomorphisms onto $\widetilde{\mathcal{N}}$.
	\end{proof}
	
	We now state our main theorem:
	\begin{thm}\label{thm_main_foliation}
		Let $\mathcal{N}$ be an asymptotically Schwarzschildean lightcone. Then there exists an asymptotic foliation of STCMC surfaces. The foliation is unique within an a-priori class of surfaces in the following sense: There exists suitable constants $B_1,B_2,B_3$, $c_1\ge 10$, and $\sigma_0$, such that if $\Sigma$ is an STCMC surface in $B_\sigma(B_1,B_2,B_3)$ for $\sigma\ge \sigma_0$ and satisfies
		\[
			\norm{\frac{\omega}{\rho}}_{C^3(\widehat{\gamma})}\le c_1
		\]
		then $\Sigma$ is a leaf of the foliation.
	\end{thm}
	\begin{bem}
		Note that the constants chosen are non-unique. In particular, one can chose arbitrarily large (but fixed) constants $B_1,B_2,B_3,c_1$ such that the statement remains true for sufficiently large $\sigma_0$.
	\end{bem}
	\begin{proof}
		It remains to show that the foliation is unique within $B_\sigma(B_1,B_2,B_3)$ as claimed, where we pick $B_1,B_2,B_3$, $c_1\ge 10$ as in the proof of Theorem \ref{thm_convergence}. In particular, the leaves of the foliation lie in $B_\sigma(B_1,B_2,B_3)$ and satisfy the $C^3$-bound. First, recall the maps $f_2$, $f_3$ defined via \eqref{eq_foliation_f2}, \eqref{eq_foliation_f3} with respect to the foliation.
		
		Now let $\Sigma\in B_\sigma(B_1,B_2,B_3)$ be an STCMC surface satisfying 
		\[
		\norm{\frac{\omega}{\rho}}_{C^3(\widehat{\gamma})}\le c_1
		\]
		By Remark \ref{bem_apriori} Equation \eqref{eq_aprioriH2_constant} and \eqref{eq_areareadiusB1}, we find
		\begin{align}\label{eq_maintheorem_h2estimate}
			0<\frac{4}{\rho^2}-\frac{8m}{\rho^3}-\frac{C}{\sigma^4}\le\mathcal{H}^2\le \frac{4}{\rho^2}-\frac{8m}{\rho^3}+\frac{C}{\sigma^4}\le \frac{5}{\sigma^2}
		\end{align}
		for $\sigma\ge\sigma_0$ sufficiently large. In particular, $\mathcal{H}^2$ lies in the image of $f_2$, $f_3$. Since $\mathcal{H}^2$ is constant, and $f_2$, $f_3$ bijective, we can consider $\widetilde{\sigma}:=f_2^{-1}(\mathcal{H}^2)$, $\widetilde{\sigma}'=f_3^{-1}(\mathcal{H}^2)$. In particular
		\[
			\mathcal{H}^2=\mathcal{H}^2_{\widetilde{\sigma}}=\frac{4}{(\widetilde{\sigma}')^2}-\frac{8m}{(\widetilde{\sigma}')^3}.
		\]
		We now want to show that $\Sigma=\Sigma_{\widetilde{\sigma}}$, 
		{where $\Sigma_{\widetilde{\sigma}}\in B_{\widetilde{\sigma}}(B_1,B_2,B_3)$ denotes the leaf of the foliation with $\mathcal{H}^2_{\widetilde{\sigma}}=\mathcal{H}^2$. To this end, we first show that both cross sections lie in the a-priori class $B_\sigma(\widetilde{B}_1,\widetilde{B}_2,\widetilde{B}_3)$ for suitable constants.}
			
		First note that, as $f_2$ is a smooth, strictly decreasing bijection, \eqref{eq_maintheorem_h2estimate} implies that for $\sigma\ge\sigma_0$ sufficiently large, $\widetilde{\sigma}$ is sufficiently large as in the proof of Theorem \ref{thm_foliation} such that
		\[
			\frac{1}{2}\widetilde{\sigma}\le\widetilde{\sigma}'\le 2\widetilde{\sigma},
		\]
		and in particular such that $\widetilde{\sigma}'\ge 3m$. Thus, \eqref{eq_maintheorem_h2estimate} implies that
		\[
			\frac{4}{(3\widetilde{\sigma}')^2}\le\frac{C}{\sigma^2},
		\]
		and thus $\widetilde{\sigma}\ge c\sigma$ for some small, fixed constant $c>0$. Using Remark \ref{bem_apriori} Equation \eqref{eq_aprioriH2_constant} for both $\Sigma$, $\Sigma_{\widetilde{\omega}}$, we find
		\begin{align*}
			\btr{\rho^2\mathcal{H}^2-4}&\le \frac{C}{\sigma},\\
			\btr{\widetilde{\sigma}^2\mathcal{H}^2-4}&\le \frac{C}{\widetilde{\sigma}}.
		\end{align*}
		Hence using this together with \eqref{eq_areareadiusB1}, we can conclude that for any $\varepsilon>0$, there exists $\sigma_0$ such that for all $\sigma\ge\sigma_0$ we find
		\begin{align}\label{eq_main_sigmacomp1}
			(1+\varepsilon)^{-1}\widetilde{\sigma}\le \sigma\le (1+\varepsilon)\widetilde{\sigma}.
		\end{align}
		Using Remark \ref{bem_apriori} Equation \eqref{eq_aprioriH2_constant} once again, \eqref{eq_main_sigmacomp1} now implies
		\[
			\btr{\frac{4}{\widetilde{\sigma}^2}-\frac{4}{\rho^2}}\le \btr{\frac{4}{\widetilde{\sigma}^2}-\frac{8m}{\widetilde{\sigma}^3}-\mathcal{H}^2+\mathcal{H}^2+\frac{8m}{\sigma^3}-\frac{4}{\rho^2}}+\frac{8m}{\sigma^3}+\frac{8m}{\widetilde{\sigma}^3}\le \frac{C}{\sigma^3}.
		\]
		Due to \eqref{eq_areareadiusB1}, \eqref{eq_main_sigmacomp1} this implies
		\[
			\btr{\widetilde{\sigma}-\rho}\le C
		\]
		for $\sigma\ge\sigma_0$ sufficiently large. By \eqref{eq_areareadiusB1}, we conclude that
		\begin{align}\label{eq_main_comp2}
			\btr{\sigma-\widetilde{\sigma}}\le C+2B_1.
		\end{align}
		Hence, for $\sigma\ge\sigma_0$ sufficiently large, \eqref{eq_main_sigmacomp1}, \eqref{eq_main_comp2} imply that
		$\Sigma,\Sigma_{\widetilde{\sigma}}\in B_\sigma(3B_1+C,2B_2,2B_3)$. We conclude that $\Sigma=\Sigma_{\widetilde{\sigma}}$ for $\sigma\ge\sigma_0$ sufficiently large by Proposition \ref{prop_uniqueness}.
	\end{proof}

	\section{Outlook}\label{sec_outlook}
	
	In this section, we briefly comment on possible physical interpretations of our result and motivate future research directions. For a null hypersurface approaching null infinity, energy, linear momentum and mass are defined via the Bondi--Sachs formalism, cf. \cite{maedlerwinicour}. To this end, the spacetime metric is expressed in Bondi--Sachs coordinates $(u,r,x^I)$ with respect to a family of outgoing null hypersurfaces $\{u=\operatorname{const.}\}$, and energy and linear momentum are defined as integrals over a suitably defined mass aspect function. For details, we refer the interested reader to \cite{chenwangwangyau,maedlerwinicour} and the references given therein.
	
	While not suitable to discuss all phenomena at null infinity, we note that for a fixed null hypersurface, one can equivalently define energy and linear momentum along a background foliation of surfaces $(S_r)$ with
	\begin{align}\label{eq_outlook1}
		\lim\limits_{r\to\infty}r^2\mathcal{K}_r=1,
	\end{align} 
	where $\mathcal{K}_r$ denotes the Gauss curvature of $S_r$, cf. Sauter \cite[Chapter 4]{sauter}. It is not difficult to see that the background foliation $(S_r)$ of an asymptotically Schwarzschildean lightcone as specified in Definition \ref{defi_asymclassS} satisfies \eqref{eq_outlook1}. One then finds that the given asymptotically flat lightcone $\mathcal{N}$ has vanishing linear momentum and its Bondi energy and Bondi mass coincide with the mass parameter $m$ with respect to $(S_r)$.
 	By Proposition \ref{prop_apriori}, in particular in view of \eqref{eq_apriori_veca}, we find that the asymptotic foliation $(\Sigma_\sigma)$ by STCMC surfaces constructed in Theorem \ref{thm_main_foliation} also satisfies \eqref{eq_outlook1} and is thus suitable to evaluate Bondi energy and linear momentum. We find
	\begin{align*}
	E_{Bondi}(\mathcal{N},(\Sigma_\sigma))&=E_{Bondi}(\mathcal{N},(S_r))=m,\\
	\vec{P}^i_{Bondi}(\mathcal{N},(\Sigma_\sigma))&=\vec{P}^i_{Bondi}(\mathcal{N},(S_r))=0.
	\end{align*}
	Hence, the foliation $(\Sigma_\sigma)$ is centered in the sense that the linear momentum is vanishing and that Bondi energy and Bondi mass coincide. More generally, for a family $(\widetilde{\Sigma}_s)$ converging to a boosted sphere $b_{\vec{a}}$ after rescaling, we find that
	\begin{align}
	E_{Bondi}(\mathcal{N},(\widetilde{\Sigma}_s))&=m\sqrt{1+\btr{\vec{a}}^2},\label{eq_outlook2}\\
	\vec{P}_{Bondi}^i(\mathcal{N},(\widetilde{\Sigma}_s))&=m\vec{a}\,{}^i.\label{eq_outlook3}
	\end{align}
	
	To give further interpretation, let us get inspiration from special relativity. In special relativity, the \emph{dynamic mass moment} $\vec{N}$ of a matter distribution of (non-trivial) energy $E$, linear momentum $\vec{P}$ and center of mass $\vec{z}$ is given by
	\[
	\vec{N}=E\vec{z}-t\vec{P},
	\]
	and is conserved in time. See \cite{fayngold}. In particular, the conservation in time of the dynamic mass moment dictates that the center of mass changes along a straight line in direction of the total linear momentum. This implies
	\[
		\frac{\d}{\d t}\vec{z}=\frac{\vec{P}}{E}.
	\]
	We note that this has been verified for the center of mass defined via an asymptotic foliation by STCMC surfaces in the initial data set case, see \cite{nerz2, cederbaumsakovich}. See also \cite{chrusciel}. 
	Choosing a suitable origin of coordinates such that $\vec{N}=\vec{0}$, this further gives
	\begin{align}\label{eq_outlook4}
		\vec{z}=t\frac{\vec{P}}{E}.
	\end{align}
	As $t=r=\omega$ along a spacelike cross section $\Sigma_\omega$ of the Minkowski lightcone, this may serve as a good analogy to the situation studied in our paper. Recall that the boost vector $\vec{a}$ is defined via the associated $4$-vector $\textbf{Z}$ as in Subsection \ref{subsec_4vector} and satisfies the decay estimate \eqref{eq_apriori_veca} along the asympotic foliation by STCMC surfaces $(\Sigma_\sigma)$ constructed in Theorem \ref{thm_main_foliation}. Further, \eqref{eq_apriori_veca} implies that $\frac{\btr{\textbf{Z}}}{\omega}\to 1$ as $\sigma\to\infty$ along the constructed foliation. Hence, \eqref{eq_outlook3} and \eqref{eq_outlook4} together motivate to define a notion of center of mass via the spacial components of the associated $4$-vector $\textbf{Z}$, i.e.,
	\[
	\vec{z}\,{}^i=\lim_{\sigma\to\infty}\textbf{Z}(\Sigma_\sigma)^i,\qquad i=1,2,3,
	\]
	similar to the geometric notion of center of mass in the initial data set case. Observe that \eqref{eq_apriori_veca} implies that the spacial components of $\textbf{Z}$ remain bounded, so at least converge up to a subsequence. While this does not a-priori guarantee that the limiting vector is well-defined, a similar issue persists in the setting of Riemannian manifolds and initial data sets unless suitable parity conditions are satisfied, cf. \cite{cederbaumsakovich}. See \cite{cederbaumgraf} on a discussion on the genericness of these parity conditions. We refer to \cite{chenkellerwangwangyau, chenwangwangyau} and references therein for definitions of center of mass and angular momentum at null infinity. In future work, we plan to investigate if the above definition is indeed suitable to define a notion of center of mass, and if it differs or coincides with one of the already existing notions.
	
	Finally, we want to comment on our asymptotic assumptions and the definition of the a-priori class. Observe that in the Riemannian situation considered in Huisken--Yau \cite{huiskenyau}, (finite) spatial translations preserve the strong asymptotics and the a-priori class (for sufficiently large constant $B_1$). In particular, the precise choice of the initial asymptotically Schwarzschildean coordinate chart does not play a crucial role in their construction.
	
	In the null setting, instead of a (spatial) translation, one can apply a Lorentz boost in the restricted Lorentz group $\operatorname{SO}^+(1,3)$ to a given background foliation to obtain a new background foliation. We note that the property $\lim\limits_{r\to\infty}r^2\mathcal{K}_r=1$ is preserved, as $\operatorname{SO}^+(1,3)$ is isomorphic to the M\"obius group. Indeed, $\operatorname{SO}^+(1,3)$ is a subgroup of the Bondi--Metzner--Sachs group which is the group of asymptotic isometries that relate two Bondi--Sachs coordinate systems. As expected, it is a well-known fact that energy and linear momentum transform equivariantly under a Lorentz boost. In particular, applying a Lorentz transformation to the asymptotically Schwarzschildean background foliation as in Definition \ref{defi_asymclassS} will yield an asymptotically flat background foliation that is however not asymptotically Schwarschildean in the sense of Definition \ref{defi_asymclassS}. Moreover, as we formulated the a-priori class with respect to the asymptotically Schwarzschildean background foliation, any boosted background foliation for a fixed vector $\vec{a}$ will leave the a-priori class due to \eqref{eq_apriori_veca} sufficiently far out in the asymptotic region. In particular, any foliation within the a-priori class will yield vanishing linear momentum even if the leaves are not STCMC surfaces. 
	
	On the other hand, the decay estimate \eqref{eq_apriori_veca} for the boost vector $\vec{a}$ is the key ingredient to establish the improved gradient bounds in Corollary \ref{kor_aprior} which are essential to our analysis at several points. This might suggest that it will not be possible to weaken these assumptions (in a significant way), but rather motivates the question whether these assumptions, which constitute a specific gauge choice in this setting, are naturally justified in the construction of (stable) STCMC surfaces. By the work of Chen--Wang \cite{chenwang}, we know that in the Schwarzschild lightcone all STCMC surfaces are centered round spheres and in particular always lie within the a-priori class. Thus, it is reasonable to assume that this also holds for asymptotically Schwarzschildean lightcones sufficiently far out in the asymptotic region. In view of the work of Huisken--Yau \cite[Theorem 5.1]{huiskenyau}, we formulate the corresponding conjecture:
	\begin{conj}\label{conj1}
		Let $\mathcal{N}$ be an asymptotically Schwarzschildean lightcone. Then any strictly stable STCMC surface $\Sigma$ lies in the a-priori class $B_\sigma(B_1,B_2,B_3)$ for suitable values of $B_1, B_2,B_3$ if its area radius $\rho$ is sufficiently large.
	\end{conj}
	Observe that the Jacobi operator $J$ in the null setting, given by \eqref{eq_stability1}, does not depend on $\accentset{\circ}{A}$ in the Schwarzschild model, which is a notable difference compared to the Riemannian setting. In particular, one has to come up with different test functions to exploit the stability compared to the Riemannian setting. If Conjecture \ref{conj1} is true, this would a-posteriori justify our gauge choice and allow us to state our main result in a gauge independent purely geometric way:
	\begin{conj}\label{conj2}
		Let $\mathcal{N}$ be an asymptotically Schwarzschildean lightcone. Then there exists a unique asymptotic foliation of $\mathcal{N}$ by strictly stable STCMC surfaces.
	\end{conj}
	We aim to establish Conjecture \ref{conj1} in a future research project, directly obtaining Conjecture \ref{conj2} as a consequence. Conjecture \ref{conj2} would then imply that at least for an asymptotically Schwarzschildean lightcone there exists a geometric choice of background foliation that has vanishing linear momentum and is possibly suitable to define a geometric notion of center of mass (if well-defined regarding convergence along the foliation). In view of the results in the initial data set case \cite{cederbaumsakovich, tenan}, we expect that our main theorem, Theorem \ref{thm_main_foliation}, also holds for an asymptotically flat null hypersurface, and that one again can employ area preserving null mean curvature flow towards this goal. Due to explicit counterexamples constructed by Brendle--Eichmair \cite{brendleeichmair} in the Riemannian setting, one can however only expect Conjecture \ref{conj2} to hold for weaker asymptotics under a suitable energy condition as similar counterexamples should arise by analogy.
	
\appendix

\section{Appendix - Difference Tensor Estimates}\label{appx_A}
	Let $\gamma_r$ be as in Definition \ref{defi_aprioriclass}. For a function $\omega\colon\Sbb^2\to(r_0,\infty)$, we consider the metric $\gamma_\omega$ on $\Sbb^2$ defined as $(\gamma_\omega)_p:=(\gamma_\omega(p))_p$. Moreover, we consider the metric $\widetilde{\gamma}_\omega:=\omega^2\hatgamma$. Note that du to the asymptotic expansion of the metric $\gamma_\omega$, we find that
	\begin{align}\label{eq_appx_volumeform}
		\d\mu_{\gamma_\omega}=\d\mu_{\tildegamma_\omega}+O_{3,3}(1)\d\mu_{\hatgamma}.
	\end{align}
	In particular, the area radius $\rho$ of $(\Sbb^2,\gamma_\omega)$ and $\widetilde{\rho}$ of $(\Sbb^2,\tildegamma_\omega)$ satisfy
	\begin{align}\label{eq_appx_arearadiusdifference}
		\btr{\rho-\widetilde{\rho}}\le \frac{C}{\rho}
	\end{align}
	for sufficiently large $\rho$. Moreover, note that  as $\gamma_\omega=\tildegamma_\omega+O_{3,3}(1)$, one can check that
	\begin{align}\label{eq_appx_inverse}
	\gamma_\omega^{KL}=\tildegamma_\omega^{KL}+O_{3,3}(r^{-4})^{KL}.
	\end{align}
	Let $\Gamma_{IJ}^K$, $\widetilde{\Gamma}_{IJ}^K$, $\widehat{\Gamma}_{IJ}^K$ denote Christoffel symbols with respect to $\gamma_\omega$, $\tildegamma_\omega$, and $\hatgamma$, respectively\footnote{Here, we denote local coordinate vector fields on $\Sbb^2$ by capital letters, to emphasize that we work with the pullback of all objects as smooth (families of) tensor fields on $\Sbb^2$.}. Then the difference tensors
	\begin{align*}
		\widehat{Q}_{IJ}^K&:=\Gamma_{IJ}^K-\widehat{\Gamma}_{IJ}^K,\\
		\widetilde{Q}_{IJ}^K&:=\Gamma_{IJ}^K-\widetilde{\Gamma}_{IJ}^K,
	\end{align*}
	satisfy the following identities:
	\begin{lem}\label{lem_appx_differencetensors}
		\begin{align*}
			\widehat{Q}_{IJ}^K=
			&\,\gamma_{\omega}^{KL}\left(w\d\omega_I\hatgamma_{JL}+\omega\d\omega_J\hatgamma_{IL}-\omega\d\omega_L\hatgamma_{IJ}\right)\\
			&\,+\frac{1}{2}\gamma_{\omega}^{KL}\left(\d\omega_IO_{2,3}(r^{-1})_{JL}+\d\omega_JO_{2,3}(r^{-1})_{IL}-\d\omega_LO_{2,3}(r^{-1})_{IJ}\right)\\
			&\,+\gamma_{\omega}^{KL}\left(O_{3,2}(1)_{IJL}+O_{3,2}(1)_{JIL}-O_{3,2}(1)_{LIJ}\right),\\
			\,\\
			\widetilde{Q}_{IJ}^K=
			&\,\frac{1}{2}\gamma_\omega^{KL}\left(d\omega_IO_{2,3}(r^{-1})_{JL}+d\omega_JO_{2,3}(r^{-1})_{IL}-d\omega_LO_{2,3}(r^{-1})_{IJ}\right)\\
			&\,+\gamma^{KL}_\omega\left(O_{3,2}(1)_{IJL}+O_{3,2}(1)_{JIL}-O_{3,2}(1)_{LIJ}\right)\\
			&\,+\frac{1}{\omega}\gamma_\omega^{KL}\left(\widehat{\nabla}\omega^M\hatgamma_{IJ}-\d\omega_I\delta_J^M-\d\omega_J\delta_I^M\right)O_{3,3}(1)_{LM},
		\end{align*}
	\end{lem}
	\begin{proof}
		Recall the well-known formula for the difference tensor
		\[
			\widehat{Q}_{IJ}^K=\frac{1}{2}\gamma_\omega^{KL}\left(\widehat{\nabla}_I(\gamma_\omega)_{JL}+\widehat{\nabla}_J(\gamma_\omega)_{IL}-\widehat{\nabla}_L(\gamma_\omega)_{IJ}\right).
		\]
		From this, the identity for $\widehat{Q}$ follows from a direct computation. For the second identity, we additionally consider the difference tensor
		\begin{align}\label{eq_appx_differencetensorconformallyround}
			Q_{IJ}^K:=\widetilde{\Gamma}_{IJ}^K-\widehat{\Gamma}_{IJ}^K=\widetilde{\gamma}_{\omega}^{KL}\left(\omega d\omega_I\delta_J^K+\omega\d\omega_J\delta_I^K-\omega\widehat{\nabla}^K\omega\hatgamma_{IJ}\right).
		\end{align}
		Using \eqref{eq_appx_inverse}, \eqref{eq_appx_differencetensorconformallyround}, the identity for $\widetilde{Q}$ readily follows from the identitiy for $\widehat{Q}$.
	\end{proof}
	From this, we can prove the following useful Lemma:
	\begin{lem}\label{lem_appx_c3control}
		Assume $\frac{1}{2}\rho\le\omega\le 2\rho$. Then, for any $\alpha\in\N_0$, $1\le k\le 3$, and for $\rho$ sufficiently large the following are equivalent:
		\begin{enumerate}
			\item[(i)]
			\[
				\btr{\widehat{\nabla}^l\frac{\omega}{\rho}}_{\hatgamma}\le \frac{C}{\rho^\alpha}
			\]
			for all $1\le l\le k$, where $C$ is a constant independent of $\alpha$ and $\rho$,
			\item[(ii)] 
			\[
				\btr{\nabla^l\omega}_{\gamma_\omega}\le \frac{{C}}{\rho^{\alpha+l-1}}
			\]
			for all $1\le l\le k$, where ${C}$ is a constant independent of $\rho$ and $\alpha$.
			\item[(iii)] \[
			\btr{\widetilde{\nabla}^l\omega}_{\tildegamma_\omega}\le \frac{{C}}{\widetilde{\rho}^{\alpha+l-1}}
			\]
			for all $1\le l\le k$, where ${C}$ is a constant independent of $\widetilde{\rho}$ and $\alpha$.
		\end{enumerate}
	\end{lem}
	\begin{bem}\label{bem_C3control}
		Note that by \eqref{eq_appx_arearadiusdifference} (i) is equivalent to 
		\[
		\norm{\widehat{\nabla}^l\frac{\omega}{\widetilde{\rho}}}_{\hatgamma}\le \frac{C}{\widetilde{\rho}^\alpha}
		\]
		for all $1\le l\le k$, where $C$ is a constant independent of $\alpha$ and $\widetilde{\rho}$. In particular, it suffices to proof the equivalence between (i) and (ii), as the equivalence between (i) and (iii) follows as the special case $\gamma_\omega=\widetilde{\gamma}_\omega$.
	\end{bem}
	\begin{proof}
		In the following $C$ will denote a constant independent of $\alpha$ and $\rho$ which may differ from line to line. Note that $\gamma_\omega$ and $\tildegamma_\omega$ are (uniformly) equivalent for $\rho$ sufficiently large (as bilinear forms, e.g. with constant $2$). In particular, we find that
		\begin{align*}
			\btr{\nabla\omega}_{\gamma_\omega}&\le C\btr{\widetilde{\nabla}\omega}_{\tildegamma_\omega}=C\frac{\rho^2}{\omega^2}\btr{\widehat{\nabla}\frac{\omega}{\rho}}_{\hatgamma}\le C\btr{\widehat{\nabla}\frac{\omega}{\rho}}_{\hatgamma},\\
			\btr{\nabla\omega}_{\gamma_\omega}&\ge C^{-1}\btr{\widetilde{\nabla}\omega}_{\tildegamma_\omega}=C^{-1}\frac{\rho^2}{\omega^2}\btr{\widehat{\nabla}\frac{\omega}{\rho}}_{\hatgamma}\ge C^{-1}\btr{\widehat{\nabla}\frac{\omega}{\rho}}_{\hatgamma}
		\end{align*}
		This yields the proof for $k=1$.
		
		Recall that
		\[
			\Hess_{IJ}\omega=\rho\widehat{\Hess}_{IJ}\frac{\omega}{\rho}-\widehat{Q}_{IJ}^K\d\omega_K.
		\]
		Having already proven the equivalence for $k=1$, we can show that assuming either (i) or (ii), we obtain
		\begin{align*}
			\btr{\widehat{Q}}_{\hatgamma}&\le C\left(\btr{\widehat{\nabla}\frac{\omega}{\rho}}_{\hatgamma}+\frac{1}{\rho}\right)\le C,\\
			\btr{\widehat{Q}}_{\gamma_\omega}&\le \frac{C}{\rho}\left(\btr{\nabla{\omega}}_{\gamma_\omega}+\frac{1}{\rho}\right)\le \frac{C}{\rho}.
		\end{align*}
		Hence, using again the equivalence between $\gamma_\omega$ and $\tildegamma_\omega$, we find
		\begin{align*}
			\btr{\Hess\omega}_{\gamma_\omega}\le\frac{C}{\rho}\btr{\widehat{\Hess}\frac{\omega}{\rho}}+\frac{C}{\rho}\btr{\nabla\omega}_{\gamma_\omega},
		\end{align*}
		and
		\begin{align*}
			\btr{\widehat{\Hess}\frac{\omega}{\rho}}_{\hatgamma}\le \rho \btr{\Hess\omega}_{\gamma_\omega}+C\btr{\widehat{\nabla}\frac{\omega}{\rho}}_{\hatgamma}.
		\end{align*}
		This implies the equivalence for $k=2$. 
		
		Recall that 
		\begin{align*}
			\nabla_I\Hess_{JK}\omega=&\,\rho\widehat{\nabla}_I\widehat{\Hess}_{JK}\frac{\omega}{\rho}-d\omega_M\widehat{\nabla}_I\widehat{Q}_{JK}^M-\rho\widehat{Q}_{JK}^M\widehat{\Hess}_{MI}\frac{\omega}{\rho}-\widehat{Q}_{IJ}^L\Hess_{LK}\omega-\widehat{Q}_{IK}^L\Hess\omega_{IL}.
		\end{align*}
		Using that we have proven the equivalence for $k=2$, either (i) or (ii) yields that 
		\[
			\btr{\widehat{\nabla}\widehat{Q}}_{\hatgamma}\le C,
		\]
		where we used that by \eqref{eq_appx_inverse}
		\[
			\widehat{\nabla}_I\gamma_{\omega}^{KL}=-\frac{2}{\omega^3}\d\omega_I\hatgamma^{KL}+O_{3,2}(r^{-4})_I^{KL}+\d\omega_IO_{2,3}(r^{-5})^{KL}.
		\]
		Hence, similar to before, we find
		\[
			\btr{\nabla^3\omega}_{\gamma_\omega}\le \frac{C}{\rho^2}\btr{\widehat{\nabla}^3\frac{\omega}{\rho}}_{\widehat{\gamma}}+\frac{C}{\rho^2}\btr{\nabla\omega}_{\gamma_\omega}+\frac{C}{\rho}\btr{\Hess\omega}_{\gamma_\omega}+\frac{C}{\rho^2}\btr{\widehat{\Hess}\frac{\omega}{\rho}}_{\hatgamma},
		\]
		and
		\[
			\btr{\widehat{\nabla}^3\frac{\omega}{\rho}}_{\widehat{\gamma}}\le C\rho^2\btr{\nabla^3\omega}_{\gamma_\omega}+C\btr{\widehat{\nabla}\frac{\omega}{\rho}}_{\hatgamma}+\btr{\widehat{\Hess}\frac{\omega}{\rho}}_{\hatgamma}+C\rho\btr{\Hess\omega}_{\gamma_\omega}.
		\]
		This implies the equivalenc for $k=3$.	
	\end{proof}
	\begin{bem}
		As we assume $\gamma_\omega-\tildegamma_\omega=O_{3,3}(1)$, we can in fact prove the Lemma also for $k=4$. More generally, if we assume $\gamma_\omega-\tildegamma_\omega=O_{m,m}(1)$, then the prove holds for all $1\le k\le m+1$ by induction.
	\end{bem}
	The estimates in the prove of Lemma \ref{lem_appx_c3control}, in particular on the difference tensor, and \eqref{eq_appx_volumeform} directly imply the following equivalence of norms:
	\begin{lem}\label{lem_appx_norms_equivalence}
		Assume $\frac{1}{2}\rho\le\omega\le 2\rho$. For $\rho$ sufficiently large, there exists a constant $C$ such that for any smooth function $f\colon\mathbb{S}^2\to\R$ it holds that
		\begin{align*}
		C^{-1}\norm{f}_{C^l(\widehat{\gamma})}&\le \norm{f}_{C^l(\gamma_\omega)}\le C\norm{f}_{C^l(\widehat{\gamma})},\\
		C^{-1}\rho\norm{f}_{W^{l,2}(\widehat{\gamma})}&\le \norm{f}_{W^{l,2}({\gamma_\omega})}\le C \norm{f}_{W^{l,2}(\widehat{\gamma})}.
		\end{align*}
		for $l=0,1$. Moreover, if $\btr{\widehat{\nabla}^k\frac{\omega}{\rho}}_{\widehat{\gamma}}\le C_1$ for some $k\ge 1$, then the claim also holds for $l=k+1$.
	\end{lem}
	
	Moreover, with Lemma \ref{lem_appx_c3control} at hand we can prove the following useful estimate:
	\begin{lem}\label{lem_appx_gammaomegadecay} Assume $\frac{1}{2}\rho\le\omega\le2\rho$, and let $1\le k\le 3$, $\alpha\in \Z$. Then the following holds: If 
		\[
			\norm{\frac{\omega}{\rho}}_{C^k}\le c,
		\]
		and $(T_r)$ is a (smooth) family of $(m,n)$-tensors on $\Sbb^2$ with $T=O_{k,k}(r^{\alpha})$, then
		\[
			\norm{T}_{C^k(\gamma_\omega)}\le Cr^{\alpha+m-n},
		\]
		where $T$ and its tensor derivatives are evaluated along the set $\{r=\omega\}$.
	\end{lem}
	\begin{proof}
		We first claim the following, assuming $\frac{1}{2}\rho\le\omega\le2\rho$: For all $1\le k\le 3$ such that
		\[
			\norm{\frac{\omega}{\rho}}_{C^k}\le c,
		\]
		we have
		\[
			\btr{\widehat{\nabla}^l\widehat{Q}}_{\hatgamma}\le C
		\]
		for all $0\le l\le k-1$. Note that this was already shown in the proof of Lemma \ref{lem_appx_c3control} up to $k=2$. It is a straightforward computation to establish this in the case of $k=3$.
		Next, note that for any smooth family of tensors $(T_r)$
		\[
			\nabla_I T_r=\widehat{\nabla}_IT_r+\d\omega_I\partial_rT_r+T_r*\widehat{Q},
		\]
		where $T_r*Q$ denotes a linear combination of $T_r$ and $Q$. By applying derivatives to the above identity inductively\footnote{Note that the subsequent $r$-derivatives only apply to the components of $(T_r\vert_{r=\omega(p)})_p$, and not to $\widehat{Q}$, as they appear due to chain rule.} the Lemma follows from the above claim and the scaling properties of $\hatgamma$ and $\gamma_\omega$, i.e., for any $\alpha\in\Z$, and any $(m,n)$-tensor $T$ with 
		\[
			\btr{T}_{\hatgamma}\le Cr^{\alpha}
		\]
		we find
		\[
			\btr{T}_{\gamma_\omega}\le Cr^{\alpha+m-n}.
		\]
	\end{proof}
	Lastly, we can deduce the following technical estimate from Lemma \ref{lem_appx_c3control} (i) and (iii):
	\begin{lem}\label{lem_appx_differencetensorestimate}
		Assume $\frac{1}{2}\widetilde{\rho}\le \omega\le 2\widetilde{\rho}$, and assume that
		\[
			\norm{\frac{\omega}{\widetilde{\rho}}}_{C^l(\Sbb^2,\hatgamma)}\le c
		\]
		for some constant $c$, $1\le l\le 3$. Then
		\[
			\btr{\widetilde{\nabla}^k\widetilde{Q}}_{\tildegamma_\omega}\le \frac{C}{\widetilde{\rho}^{3+k}}
		\]
		for all $0\le k\le l-1$, where $C$ is a constant only depending on $c$ provided $\widetilde{\rho}$ is sufficiently large.
	\end{lem}
	\begin{bem}\label{bem_appx_differencetensorestimate}
		As already used in the proof of Lemma \ref{lem_appx_c3control} up untl $l=2$, we analogously can show that under the assumptions of Lemma \ref{lem_appx_differencetensorestimate} we also find
		\[
		\btr{\widetilde{\nabla}^k\widehat{Q}}_{\tildegamma_\omega}\le \frac{C}{\widetilde{\rho}^{1+k}}
		\]
		for all $0\le k\le l-1$.
	\end{bem}
	\begin{proof}
		{The proof is once again a straightforward computation carried out inductively as in the prove of Lemma \ref{lem_appx_c3control}.}
	\end{proof}
	{This control on the difference tensor $\widetilde{Q}$ now allows us to obtain the desired $C^1$-comparison of the scalar curvatures $\operatorname{R}$ and $\widetilde{\scal}$:}
	\begin{lem}\label{lem_appx_scalarcurv}
		Assume $\frac{1}{2}\widetilde{\rho}\le \omega\le 2\widetilde{\rho}$, and 
		\[
		\norm{\frac{\omega}{\widetilde{\rho}}}_{C^3(\Sbb^2,\hatgamma)}\le c
		\]
		for some constant $c$.
		Let $R$ denote the scalar curvature of $(\Sbb^2,\gamma_\omega)$, and let $\widetilde{R}$ denote the scalar curvature of $(\Sbb^2,\tildegamma_\omega)$. Then
		\[
			\left(1-\frac{C}{\widetilde{\rho}^2}\right)R-\frac{C}{\widetilde{\rho}^4}\le \widetilde{R}\le \left(1+\frac{C}{\widetilde{\rho}^2}\right)R+\frac{C}{\widetilde{\rho}^4},
		\]
		and
		\[
			\btr{\d\widetilde{R}}_{\tildegamma_\omega}\le \frac{C}{\widetilde{\rho}^3}\btr{R}+C\btr{\d R}_{\gamma_{\omega}}+\frac{C}{\widetilde{\rho}^5}.
		\]
		for constants only depending on $c$ provided $\widetilde{\rho}$ is sufficiently large.
	\end{lem}
	\begin{proof}
		Recall that
		\[
			R=\gamma_\omega^{AB}\left(\Gamma_{AB,J}^J-\Gamma_{JB,A}^J+\Gamma_{AB}^I\Gamma_{JI}^J-\Gamma_{AI}^J\Gamma_{BJ}^I\right),
		\]
		and in $2$ dimensions, we in fact have
		\[
			\frac{1}{2}R\left(\gamma_\omega\right)_{AB}=\Gamma_{AB,J}^J-\Gamma_{JB,A}^J+\Gamma_{AB}^I\Gamma_{JI}^J-\Gamma_{AI}^J\Gamma_{BJ}^I.
		\]
		Using the above identities, a straightforward computation gives
		\begin{align*}
			R-\widetilde{R}
			=&\,{\gamma}_\omega^{AB}\left(\widetilde{\nabla}_J\widetilde{Q}_{AB}^J-\widetilde{\nabla}_A\widetilde{Q}_{JB}^J+\widetilde{Q}_{AB}^I\widetilde{Q}_{JI}^J-\widetilde{Q}_{AI}^J\widetilde{Q}_{JB}^I\right)\\
			&\,+\frac{1}{2}\widetilde{R}(\tildegamma_\omega)_{AB}\left(\gamma_{\omega}^{AB}-\tildegamma_\omega^{AB}\right).
		\end{align*}
		Hence
		\begin{align*}
			\widetilde{R}=
			&\,\frac{1}{1+\frac{1}{2}(\tildegamma_\omega)_{AB}\left(\gamma_{\omega}^{AB}-\tildegamma_\omega^{AB}\right)}\left(R-{\gamma}_\omega^{AB}\left(\widetilde{\nabla}_J\widetilde{Q}_{AB}^J-\widetilde{\nabla}_A\widetilde{Q}_{JB}^J+\widetilde{Q}_{AB}^I\widetilde{Q}_{JI}^J-\widetilde{Q}_{AI}^J\widetilde{Q}_{JB}^I\right)\right).
		\end{align*}
		Recall that 
		\[
			\btr{\widetilde{\nabla}\gamma_\omega^{-1}}_{\widetilde{\gamma}_\omega}\le \frac{C}{\widetilde{\rho}^3},
		\]
		so we note that
		\begin{align*}
			\btr{1-\frac{1}{1+\frac{1}{2}(\tildegamma_\omega)_{AB}\left(\gamma_{\omega}^{AB}-\tildegamma_\omega^{AB}\right)}}&\le \frac{C}{\widetilde{\rho}^2}\\
			\btr{\widetilde{\nabla}\left(\frac{1}{1+\frac{1}{2}(\tildegamma_\omega)_{AB}\left(\gamma_{\omega}^{AB}-\tildegamma_\omega^{AB}\right)}\right)}_{\widetilde{\gamma}_\omega}&\le \frac{C}{\widetilde{\rho}^3}.
		\end{align*}
		The claim then directly follows from Lemma \ref{lem_appx_differencetensorestimate}, and the fact that 
		\[
			\btr{\d R}_{\tildegamma_\omega}\le C\btr{\d R}_{\gamma_{\omega}}
		\]
		by the (uniform) equivalence of the metrics for $\widetilde{\rho}$ sufficiently large.
	\end{proof}
\section{Appendix - Curvature Estimates}\label{appx_B}
In this appendix, we establish some curvature estimates on asymptotically Schwarzschildean lightcones. Throughout this appendix, we set $h(r)=1-\frac{2m}{r}$.
For the Riemann curvature tensor of a semi-Riemannian manifold $(M,g)$, we use the convention
\begin{align*}
\operatorname{Rm}(X,Y,W,Z)&=g\left( \nabla_X\nabla_YZ-\nabla_Y\nabla_XZ-\nabla_{[X,Y]}Z,W \right),
\end{align*}
and define the Ricci curvature tensor and the scalar curvature accordingly as
\begin{align*}
\operatorname{Ric}(V,W)&=\operatorname{tr}_g \operatorname{Rm}(V,\cdot,W,\cdot),\\
\operatorname{R}&=\operatorname{tr}_g\operatorname{Ric}.
\end{align*}
In the next two lemmata, we state some curvature identities in the Schwarzschild spacetime.
\begin{lem}\label{lem_appx_Riemcurv_classS}
	All non-trivial Riemann-components in the Schwarzschild spacetime are given by
	\begin{align*}
	\overline{\Riem}^{Schw}_{\ul{L}L_r\ul{L}L_r}&=2h'',\\
	\overline{\Riem}^{Schw}_{I\ul{L}JL_r}&=-rh'\hatgamma_{IJ},\\
	\overline{\Riem}^{Schw}_{IJKM}&=r^2(1-h)\left(\hatgamma_{IK}\hatgamma_{JM}-\hatgamma_{IM}\hatgamma_{KJ}\right).
	\end{align*}
\end{lem}

\begin{lem}\label{lem_appx_DerivRiemcurv_classS}
	\begin{align*}
	\overline{\nabla}_A\overline{\Riem}^{Schw}_{BCDL_r}&=-rh\left(1-h+\frac{1}{2}rh'\right)\left(\hatgamma_{BD}\hatgamma_{AC}-\hatgamma_{AB}\hatgamma_{CD}\right),\\
	\overline{\nabla}_A\overline{\Riem}^{Schw}_{BCD\ul{L}}&=-r\left(1-h+\frac{1}{2}rh'\right)\left(\hatgamma_{BD}\hatgamma_{AC}-\hatgamma_{AB}\hatgamma_{CD}\right),\\
	\overline{\nabla}_A\overline{\Riem}^{Schw}_{B\ul{L}\ul{L}L_r}&=(h'+rh'')\hatgamma_{AB},\\
	\overline{\nabla}_{\ul{L}}\overline{\Riem}^{Schw}_{BCDA}&=-2r\left(1-h+\frac{1}{2}rh'\right)\left(\hatgamma_{BD}\hatgamma_{AC}-\hatgamma_{AB}\hatgamma_{CD}\right),\\
	\overline{\nabla}_{\ul{L}}\overline{\Riem}^{Schw}_{B\ul{L}DL_r}&=\left(h'-rh''\right)\hatgamma_{BD}.
	\end{align*}
	Moreover,
	\begin{align*}
		\overline{\nabla}_A\overline{\Riem}^{Schw}_{BCDM}&=\overline{\nabla}_A\overline{\Riem}^{Schw}_{BC\ul{L}L_r}=\overline{\nabla}_A\overline{\Riem}^{Schw}_{B\ul{L}DL_r}=\overline{\nabla}_A\overline{\Riem}^{Schw}_{B\ul{L}\ul{L}M}=0,\\
		\overline{\nabla}_{\ul{L}}\overline{\Riem}^{Schw}_{BCDML_r}&=\overline{\nabla}_{\ul{L}}\overline{\Riem}^{Schw}_{BCD\ul{L}}=\overline{\nabla}_{\ul{L}}\overline{\Riem}^{Schw}_{BC\ul{L}L_r}=\overline{\nabla}_{\ul{L}}\overline{\Riem}^{Schw}_{B\ul{L}D\ul{L}}=\overline{\nabla}_{\ul{L}}\overline{\Riem}^{Schw}_{B\ul{L}\ul{L}L_r}=0.
	\end{align*}
\end{lem}
Lemma \ref{lem_appx_Riemcurv_classS} and Lemma \ref{lem_appx_DerivRiemcurv_classS} follow from direct computation. For a complete list of Christoffel symbols (in double null coordinates), see \cite[Appendix B]{cedwolff}.
For the rest of this section, we will always assume that $\mathcal{N}$ is asymptotically Schwarzschildean. Using Lemma \ref{lem_appx_c3control} above, we establish the following estimates:
\begin{lem}\label{lem_appx_curvatureestimates1}
	Assume $\Sigma_\omega$ is a spacelike cross section of $\mathcal{N}$ with $\frac{1}{2}\rho\le\omega\le2\rho$, $\rho>2r_0$, and
	\[
		\norm{\frac{\omega}{\rho}}_{C^2(\Sbb^2,\hatgamma)}\le c.
	\]
	Then there exists a constant $C>0$ depending on $c$ such that
	\begin{align*}
		\norm{\overline{\Riem}(\ul{L},L,\ul{L},L)-2h''}_{C^2(\gamma_\omega)}&\le \frac{C}{\rho^4},\\
		\norm{\overline{\Ric}(\ul{L},\ul{L})}_{C^2(\gamma_\omega)}&\le \frac{C}{\rho^4},\\
		\norm{\overline{\Ric}(\ul{L},L)+\frac{2}{\omega}h'+h''}_{C^2(\gamma_\omega)}&\le \frac{C}{\rho^4},\\
		\norm{\overline{R}-\frac{2(1-h)}{\omega^2}+\frac{4h'}{\omega}+h''}_{C^2(\gamma_\omega)}&\le \frac{C}{\rho^4},\\
		\norm{\overline{\Riem}_{ijkL}-\left(\frac{2(1-h)}{\omega^2}+\frac{h'}{\omega}\right)\left(\d\omega_i(\gamma_\omega)_{jk}-\d\omega_j(\gamma_\omega)_{ik}\right)}_{C^2(\gamma_\omega)}&\le \frac{C}{\rho^4}.
	\end{align*}
\end{lem}
\begin{proof}
	Using the decomposition of 
	\begin{align*}
		\partial_i&=\partial_I+\d\omega_I\ul{L},\\
		L&=L_r-\btr{\nabla\omega}^2_{\gamma_\omega}\ul{L}-2\nabla\omega^I\partial_I,
	\end{align*}
	we find that
	\begin{align*}
	-\frac{1}{2}\overline{\Riem}(\ul{L},L,\ul{L},L)
	=&\,-\frac{1}{2}\overline{\Riem}(\ul{L},L_r,\ul{L},L_r)+\nabla\omega^I\overline{\Riem}(\ul{L},L_s,\ul{L},\partial_I)\\
	&\,-2\nabla\omega^I\nabla\omega^J\overline{\Riem}(\ul{L},\partial_I,\ul{L},\partial_J),\\
	2\overline{\Ric}(\ul{L},L)
	=&\,\gamma_\omega^{IJ}\overline{\Riem}(\ul{L},\partial_I,L_r,\partial_J)-\btr{\nabla\omega}^2_{\gamma_\omega}\gamma_\omega^{IJ}\overline{\Riem}(\ul{L},\partial_I,\ul{L},\partial_J)\\
	&\,-2\nabla\omega^M\gamma_{\omega}^{IJ}\overline{\Riem}(\ul{L},\partial_I,\partial_M,\partial_J)\\
	&\,+\frac{1}{2}\overline{\Riem}(\ul{L},L_r,L_r,\ul{L})-\nabla\omega^I\overline{\Riem}(\ul{L},L_r,\partial_I,\ul{L}),
	\end{align*}
	where we used that at each point $p\in\Sigma_\omega$, $(\gamma_\omega)_p=(\gamma_{r=\omega(p)})_p$.
{Note that by Definition \ref{defi_asymclassS} and Lemma \ref{lem_appx_gammaomegadecay}, we find
	\[
		\norm{\overline{\Riem}_{\alpha_1\alpha_2\alpha_3\alpha_4}-\overline{\Riem}^{Schw}_{\alpha_1\alpha_2\alpha_3\alpha_4}}_{C^2(\gamma_\omega)}\le\frac{C}{\rho^4}
	\]
	for all $\alpha_i\in\{\partial_I,\ul{L},L_r\}$, cf. Remark \ref{bem_prelim_O_spacetime}. Using Lemma \ref{lem_appx_c3control} and \ref{lem_appx_Riemcurv_classS}, the claim for $\overline{\Riem}(\ul{L},L,\ul{L},L)$ and $\overline{\Ric}(\ul{L},L)$ readily follows, additionally using \eqref{eq_appx_inverse} in the latter case. All other identities follow similarly.}
\end{proof}

\section{Appendix - The Null Simons identity}\label{appendix_nullsimon}
In this appendix, we establish a contracted version of the Null Simon's identity Proposition  \ref{prop_nullsimon} which is an important ingredient for establishing estimates along APNMCF in Section \ref{sec:APNMCF}.
\begin{bem}\label{bem_nullsimon}
	Note that Proposition \ref{prop_nullgauss}, Equation \eqref{eq_prelim_gauss1} can be reformulated as
	\begin{align}
	\Riem_{kilm}A^m_j=&\overline{\Riem}_{kilm}A^m_j+\frac{1}{2}B_{kl}A_{im}A^m_j+\frac{1}{2}A_{kl}B_{im}A^m_j-\frac{1}{2}B_{il}A_{km}A^m_j-\frac{1}{2}A_{il}B_{km}A^m_j,
	\end{align}
	where $B:=\frac{1}{\ul{\theta}}\ul{\chi}$. Comparing covariant derivatives furthermore yields		
	\begin{align}
	\begin{split}\label{eq_nullsimon_derivative}
	\overline{\nabla}_i\overline{\Riem}_{kjl(\ul{\theta}L)}
	=&\nabla_i(\overline{\Riem}_{kjl(\ul{\theta}L)})+\frac{1}{2}\overline{\Riem}_{\ul{L}jlL}A_{ik}+\frac{1}{2}\overline{\Riem}_{k\ul{L}lL}A_{ij}+\frac{1}{2}\overline{\Riem}_{kj\ul{L}L}A_{il}\\
	&+\frac{1}{2}\overline{\Riem}_{(\ul{\theta}L)jl(\ul{\theta}L)}B_{ik}+\frac{1}{2}\overline{\Riem}_{k(\ul{\theta}L)l(\ul{\theta}L)}B_{ij}-\overline{\Riem}_{kjlm}A^m_i+\tau_i\overline{\Riem}_{kjl(\ul{\theta}L)}.
	\end{split}
	\end{align}
	Thus, we can state the Null Simon's identity, Proposition \ref{prop_nullsimon}, equivalently as
	\begin{align*}
	\nabla_k\nabla_l A_{ij}
	=&\nabla_i\nabla_j A_{kl}+\tau_j\nabla_iA_{kl}+\tau_i\nabla_kA_{jl}-\tau_k\nabla_iA_{jl}-\tau_l\nabla_kA_{ij}\\
	&+\frac{1}{2}B_{kl}A_{im}A^m_j+\frac{1}{2}A_{kl}B_{im}A^m_j+\frac{1}{2}B_{kj}A_{im}A^m_l+\frac{1}{2}A_{kj}B_{im}A^m_l\\
	&-\frac{1}{2}B_{il}A_{km}A^m_j-\frac{1}{2}A_{il}B_{km}A^m_j-\frac{1}{2}B_{ij}A_{km}A^m_l-\frac{1}{2}A_{ij}B_{km}A^m_l\\
	&+\nabla_i\tau_jA_{kl}+\nabla_k\tau_iA_{jl}-\nabla_i\tau_kA_{jl}-\nabla_k\tau_lA_{ij}+\overline{\Riem}_{kilm}A^m_j+\overline{\Riem}_{kijm}A^m_l\\
	&+\overline{\Riem}_{kjlm}A^m_i-\frac{1}{2}\overline{\Riem}_{\ul{L}jlL}A_{ik}-\frac{1}{2}\overline{\Riem}_{k\ul{L}lL}A_{ij}-\frac{1}{2}\overline{\Riem}_{kj\ul{L}L}A_{il}\\
	&+\overline{\Riem}_{lijm}A^m_k-\frac{1}{2}\overline{Rm}_{\ul{L}ijL}A_{kl}-\frac{1}{2}\overline{\Riem}_{l\ul{L}jL}A_{ki}-\frac{1}{2}\overline{\Riem}_{li\ul{L}L}A_{kj}\\
	&-\frac{1}{2}\overline{\Riem}_{(\ul{\theta}L)jl(\ul{\theta}L)}B_{ik}-\frac{1}{2}\overline{\Riem}_{k(\ul{\theta}L)l(\ul{\theta}L)}B_{ij}-\frac{1}{2}\overline{\Riem}_{(\ul{\theta}L)ij(\ul{\theta}L)}B_{kl}-\frac{1}{2}\overline{\Riem}_{l(\ul{\theta}L)j(\ul{\theta}L)}B_{ik}\\
	&+\overline{\nabla}_i\overline{\Riem}_{kjl(\ul{\theta}L)}+\overline{\nabla}_k\overline{\Riem}_{lij(\ul{\theta}L)}-\tau_i\overline{\Riem}_{kjl(\ul{\theta}L)}-\tau_k\overline{\Riem}_{lij(\ul{\theta}L)}.
	\end{align*}
\end{bem}
\noindent{In the following, we define
	\[
	(T_1\cdot T_2)_{ij}:=\gamma_\omega^{kl}(T_1)_{ik}(T_2)_{lj}
	\]
	for two $(0,2)$-tensors $T_1$, $T_2$, and $*$ denotes a linear combination of contractions of tensors (with respect to $\gamma_\omega$).}
\begin{lem}\label{lem_appx_nullsimon}
	Let $\mathcal{N}$ be asymptotically Schwarzschildean, and let $\Sigma$ be a spacelike cross section with $\frac{1}{2}\rho\le \omega\le 2\rho$, and assume that
	\[
		\norm{\frac{\omega}{\rho}}_{C^2(\hatgamma)}\le c.
	\]
	Then
	\begin{align*}
		\Hess \mathcal{H}^2_{ij}
		=&\, \Delta A_{ij}-\left(R+\frac{2(1-h)}{\omega^2}+\frac{h'}{\omega}\right)\accentset{\circ}{A}_{ij}+\btr{\omega}^2\left(h''+\frac{2(1-h)}{\omega^2}\right)\ul{\theta}\accentset{\circ}{\ul{\chi}}_{ij}\\
		&\,+\left(\ul{\theta}^2\left(\frac{2(1-h)}{\omega^2}+\frac{2h'}{\omega}-h''\right)-\ul{\theta}\left(16\frac{1-h}{\omega^3}+12\frac{h'}{\omega^2}\right)\right)\accentset{\circ}{\left(\d\omega\otimes\d\omega\right)}_{ij}\\
		&\, +\nabla_i\accentset{\circ}{A}_{jm}\tau^m+\nabla_j\accentset{\circ}{A}_{im}\tau^m-\dive\accentset{\circ}{A}_i\tau_j-\dive\accentset{\circ}{A}_j\tau_i\\
		&\,+\dive\tau A_{ij}-\mathcal{H}^2\nabla_i\tau_j+(\nabla\tau\cdot A)_{ij}-(A\cdot\nabla\tau)_{ji}\\
		&\, +\mathcal{H}^2 (\tau\otimes\tau)_{ij}-\btr{\tau}^2A_{ij}-2\ul{\theta}\left(2\frac{(1-h)}{\omega^2}+\frac{h'}{\omega}\right)\left(\accentset{\circ}{\left(\d\omega\otimes\tau\right)}+\accentset{\circ}{\left(\tau\otimes\d\omega\right)}\right)_{ij}\\
		&\,+\ul{\theta}f^{(1)}_{ij}+\accentset{\circ}{(A*f^{(2)})}_{ij}+\ul{\theta}\accentset{\circ}{(\ul{\chi}*f^{(3)})}_{ij}+\ul{\theta}\accentset{\circ}{(\tau*f^{(4)})}_{ij},
	\end{align*}
	where $f^{(i)}$ are smooth tensors on $\Sbb^2$ such that
	\[
		\rho\norm{f^{(1)}}_{C^1(\gamma_\omega)}+\norm{f^{(2)}}_{C^1(\gamma_\omega)}+\norm{f^{(3)}}_{C^1(\gamma_\omega)}+\norm{f^{(4)}}_{C^1(\gamma_\omega)}\le \frac{C}{\rho^4},
	\]
	and $f^{(1)}$ is symmetric and trace-free, with $C$ only depending on $c$.
\end{lem}
\begin{proof}
	The results directly follows by computation after taking a trace over the null Simons' identity, Proposition \ref{prop_nullsimon}. For the convenience of the reader, we outline some parts of the computation.
	
	First notice that as $\Sigma$ is $2$-dimensional, we have
	\[
		\Riem_{ijkl}=\frac{1}{2}\scal\left((\gamma_\omega)_{ik}(\gamma_\omega)_{jl}-(\gamma_\omega)_{jk}(\gamma_\omega)_{il}\right).
	\]
	From this, it readily follows that
	\[
		\gamma_\omega^{kl}\left(\Riem_{kijm}A^m_l+\Riem_{kilm}A^m_j\right)=\scal\accentset{\circ}{A}_{ij}.
	\]
	Using the Codazzi Equation, we further find
	\begin{align*}
		&\,\tau_l\nabla_kA_{ij}+\tau_k\nabla_iA_{lj}-\tau_i\nabla_kA_{lj}-\tau_j\nabla_iA_{kl}\\
		=\,&\tau_l\nabla_jA_{ki}+\tau_k\nabla_iA_{lj}-\tau_i\nabla_kA_{lj}-\tau_j\nabla_kA_{il}+\left(\nabla_kA_{il}-\nabla_i A_{kl}\right)\tau_j+\left(\nabla_kA_{ji}-\nabla_jA_{ki}\right)\tau_l\\
		=\,&\frac{1}{2}\tau_l\d\mathcal{H}^2_j(\gamma_\omega)_{ki}+\frac{1}{2}\tau_k\d\mathcal{H}^2_i(\gamma_\omega)_{lj}-\frac{1}{2}\tau_i\d\mathcal{H}^2_k(\gamma_\omega)_{lj}-\frac{1}{2}\tau_j\d\mathcal{H}^2_k(\gamma_\omega)_{il}\\
		\,&+\tau_l\nabla_j\accentset{\circ}{A}_{ki}+\tau_k\nabla_i\accentset{\circ}{A}_{lj}-\tau_i\nabla_k\accentset{\circ}{A}_{lj}-\tau_j\nabla_k\accentset{\circ}{A}_{il}\\
		\,&+\ul{\theta}\overline{\Riem}_{kilL}\tau_j+\tau_j\tau_iA_{kj}-\tau_k\tau_jA_{li}+\ul{\theta}\overline{\Riem}_{kjiL}\tau_l+\tau_l\tau_jA_{ki}-\tau_l\tau_kA_{ij}.
	\end{align*}
Note that
	\[
		\gamma_\omega^{kl}\left(\frac{1}{2}\tau_l\d\mathcal{H}^2_j(\gamma_\omega)_{ki}+\frac{1}{2}\tau_k\d\mathcal{H}^2_i(\gamma_\omega)_{lj}-\frac{1}{2}\tau_i\d\mathcal{H}^2_k(\gamma_\omega)_{lj}-\frac{1}{2}\tau_j\d\mathcal{H}^2_k(\gamma_\omega)_{il}\right)=0.
	\]
	Recall the decomposition 
	\begin{align*}
		\partial_i&=\partial_I+\d\omega_I\ul{L},\\
		L&=L_r-\btr{\nabla\omega}^2\ul{L}-2\nabla\omega^M\partial_M,
	\end{align*}
	with which one can decompose all ambient Riemann-terms, e.g.
	\begin{align}
		\begin{split}\label{eq_appx_riemdecompLbarL}
		\overline{\Riem}_{i\ul{L}jL}=\,&\overline{\Riem}_{I\ul{L}JL_r}-\btr{\nabla\omega^2}\overline{\Riem}_{I\ul{L}J\ul{L}}-2\nabla\omega^M\overline{\Riem}_{I\ul{L}JM}\\
		&\,+\d\omega_J\overline{\Riem}_{I\ul{L}\ul{L}L_r}-2\d\omega_J\nabla^M\overline{\Riem}_{I\ul{L}\ul{L}M}.
		\end{split}
	\end{align}
	To apply the decay assumptions in Definition \ref{defi_asymclassS} also for the tensor derivatives of $\overline{\Riem}$, we use equation \eqref{eq_nullsimon_derivative}. Observe that\footnote{Note that as $\omega$ is not constant in general, we implicitly use the bound on $\widehat{\nabla}\frac{\omega}{\rho}$.}
	\[
		\hatgamma=\frac{1}{\omega^2}\gamma_\omega+O_{1,1}(r^{-2}),\text{\,\,\,\,\,\,\,\,\,\,\, }\hatgamma\otimes\hatgamma=\frac{1}{\omega^4}\gamma_\omega\otimes\gamma_\omega+O_{1,1}(r^{-2}).
	\]
	The claim then follows from the decay assumptions in Definition \ref{defi_asymclassS}, Lemma \ref{lem_appx_gammaomegadecay}, Lemma \ref{lem_appx_Riemcurv_classS} and Lemma \ref{lem_appx_DerivRiemcurv_classS}, where the tensors $f^{(i)}$ arise from the difference between the Riemann terms and the respective values in Schwarzschild. Note that the fact that $f^{(1)}$ and the contractions with $f^{(2)}$, $f^{(3)}$, $f^{(4)}$ are individually trace-free follows from the algebraic properties of the individual terms.
\end{proof}

\bibliography{bib_schwarzschildlightcones}

\nopagebreak
\bibliographystyle{plain}
\end{document}